\documentclass[aos]{imsart}
\RequirePackage{amsthm,amsmath,amsfonts,amssymb}
\RequirePackage[numbers]{natbib}

\startlocaldefs
\usepackage[utf8]{inputenc}
\usepackage[english]{babel}
\usepackage[T1]{fontenc}

\usepackage{booktabs}
\usepackage{graphicx} 
\usepackage{tabularx}
\usepackage{xcolor} 
\usepackage{comment}
\usepackage[shortlabels]{enumitem}
\usepackage{mathrsfs}
\newcommand{\mathbbm}[1]{\text{\usefont{U}{bbm}{m}{n}#1}}
\usepackage{hyperref}
\usepackage{amssymb}
\usepackage{algorithm}
\usepackage{algpseudocode}
\usepackage{mathabx}

\usepackage{definitions}
\newtheorem{theorem}{Theorem}

\newtheorem{lemma}{Lemma}
\newtheorem{definition}{Definition}
\theoremstyle{remark}

\endlocaldefs

\begin{document}

\begin{frontmatter}
\title{Federated Nonparametric Hypothesis Testing with Differential Privacy Constraints: Optimal Rates and Adaptive Tests}
\runtitle{Federated Nonparametric Private Testing}

\begin{aug}
    \author[A]{\fnms{T. Tony}~\snm{Cai}\ead[label=e1]{tcai@wharton.upenn.edu}},
    \author[A]{\fnms{Abhinav}~\snm{Chakraborty}\ead[label=e2]{abch@wharton.upenn.edu}}
    \and
    \author[A]{\fnms{Lasse}~\snm{Vuursteen}\ead[label=e3]{lassev@wharton.upenn.edu}}
    
\address[A]{Department of Statistics and Data Science,
University of Pennsylvania\printead[presep={,\ }]{e1,e2,e3}}
    
\end{aug}

\date{\today}

\begin{abstract}
Federated learning has attracted significant recent attention due to its applicability across a wide range of settings where data is collected and analyzed across disparate locations. In this paper, we study federated nonparametric goodness-of-fit testing in the white-noise-with-drift model under distributed differential privacy (DP) constraints. 

We first establish matching lower and upper bounds, up to a logarithmic factor, on the minimax separation rate. This optimal rate serves as a benchmark for the difficulty of the testing problem, factoring in model characteristics such as the number of observations, noise level, and regularity of the signal class, along with the strictness of the $(\epsilon,\delta)$-DP requirement. The results demonstrate interesting and novel phase transition phenomena. Furthermore, the results reveal an interesting phenomenon that
distributed one-shot protocols with access to shared randomness outperform those without access to shared randomness. We also construct a data-driven testing procedure that possesses the ability to adapt to an unknown regularity parameter over a large collection of function classes with minimal additional cost, all while maintaining adherence to the same set of DP constraints.
	\\[5pt]
\end{abstract}

    

\end{frontmatter}

\section{Introduction}
\label{sec:intro}

Federated learning is a collaborative distributed machine learning technique designed to address data governance and privacy concerns. It facilitates organizations or groups to collectively train a shared global model without the need to expose raw data externally. Federated learning has garnered increasing attention due to its applicability across a wide range of settings where data is collected and analyzed across disparate locations. This includes scenarios like medical data dispersed among different hospitals, financial customer data stored across various branches or databases, and the utilization of federated learning within user networks in modern technologies such as smartphones or self-driving cars. See, for example, \cite{li2020federated,konevcny2016federated,hard2018federated,beaufays2019federated,nguyen2022deep}. In such contexts, privacy concerns often impede direct data pooling, making the development of efficient statistical inference methods that preserve privacy and harness the collective power of distributed data essential.

In this paper, we investigate federated nonparametric goodness-of-fit testing under distributed differential privacy (DP) constraints. 
Nonparametric hypothesis testing, a fundamental statistical problem, has been extensively studied in conventional settings, with a rich body of classical literature examining theoretically optimal performance. DP, introduced by Dwork et al. (2006), serves as a mathematical guarantee determining whether results or datasets can be deemed ``privacy-preserving" and thus openly published. Many differentially private statistical methods have since been developed. See, for example, \cite{arachchige2019local,dwork2010differential,dwork2017exposed}.  While several other privacy frameworks exist, DP holds a prominent position both theoretically and practically, finding application within industry giants like Google \cite{google_privacy}, Microsoft \cite{ding2017collecting}, Apple \cite{apple2017}, as well as governmental entities such as the US Census Bureau \cite{rodriguezmodernization}.
 
DP may significantly impact the quality of statistical inference, particularly in testing problems where it may diminish statistical power that could be obtained with complete data availability. In the present paper, we quantify the cost of privacy in the canonical nonparametric goodness-of-fit testing setting, delineating the theoretical limits of performance achievable under DP constraints and developing methods that attain optimal performance. 
Contrasting our results with those derived for federated nonparametric estimation under DP  \cite{cai2023private}, we observe that the differences persist and are exacerbated under privacy constraints, highlighting the unique challenges posed by privacy preservation in both testing and estimation contexts.

We first establish the minimax separation rate for the nonparametric goodness-of-fit testing problem and construct optimal tests in the oracle setting where the regularity parameters are known. The minimax separation rate serves as a benchmark for the difficulty of the testing problem. However, the regularity parameters  are typically in applications. A natural question is: Without the knowledge of the regularity parameters, is it possible to construct a test that is as good as when the parameters are  known? This is a question about {\it adaptation}, which has been a major goal in nonparametric statistics.
We construct a data-driven test and show that the proposed method can adapt to unknown regularity parameters with minimal additional cost while adhering the same DP restrictions.

\subsection{Federated privacy-constrained testing}\label{ssec:federated_testing}

We begin by formally introducing the general framework of federated inference under distributed DP constraints. Consider a family of probability measures  $\{P_f\}_{f \in \cF}$ on the measurable space $(\mathcal{X},\mathscr{X})$, parameterized by $f \in \cF$. We consider a setting where $N = mn$ i.i.d. observations are drawn from a distribution $P_f$ and distributed across $m$ servers. Each server $j=1,\dots,m$ holding an equal amount ($n$ many) observations.

Let us denote by $X^{(j)} = (X^{(j)}_i)_{i=1}^{n}$ the $n$ realizations from $P_f$ on the $j$-th server. Based on $X^{(j)}$, each server outputs a (randomized) transcript $Y^{(j)}$ to the central server that satisfies the privacy constraint. The central server, utilizing all transcripts $Y :=(Y^{(1)}, \dots,Y^{(m)})$, decides between a null hypothesis and an alternative hypothesis, through means of a test $T=T(Y)$. Since we are concerned with testing between a null and alternative hypothesis, we shall consider the decision space $\{0,1\}$, where $0$ corresponds to $\textsc{do not reject}$ and $1$ with $\textsc{reject}$. A \emph{test} is then simply to be understood as a statistic taking values in $\{0,1\}$. Figure \ref{fig:Distributed-Testing-Illustration} gives an illustration of a federated $(\epsilon,\delta)$-DP-constrained testing procedure.
\begin{figure}[ht]
    \centering
    \includegraphics[width=0.5\textwidth]{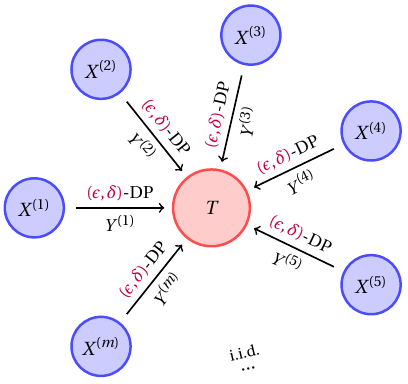}
    \caption{Illustration of federated $(\epsilon,\delta)$-DP-constrained testing.}\label{fig:Distributed-Testing-Illustration}
\end{figure}

The transcript $Y^{(j)}$ satisfies an $(\epsilon,\delta)$-DP constraint, which, loosely speaking, means that the transcript $Y^{(j)}$ cannot differ too much depending on whether a specific individual is in the data set or not.  This is achieved through randomization, which is independent of the data. We will consider two types of sources for randomization; independently among the servers or through a shared source of randomness $U$ (e.g., the same random seed). Formally, shared source of randomness means that the law of the transcript is given by a distribution conditionally on $X^{(j)}$ and $U$, $A \mapsto \P( Y^{(j)} \in A | X^{(j)},U)$, defined on a measurable space $(\cY, \mathscr{Y})$. The presence of shared randomness is a slight, but important extension of the distributed protocols where $Y^{(j)}$ is allowed to be random only through locally generated randomness. See further discussion below and in Section \ref{sec:main_results}. To assure that the source of shared randomness does not erode the notion of privacy, only the local source of randomness is used in the privacy mechanism, i.e. to guarantee privacy. We formalize this as follows. 

We shall call two data sets $x,x' \in \mathcal{X}^{n}$ \emph{neighboring} if they differ in one individual datum. That is, they are at most one apart in \emph{Hamming distance} (see Section \ref{ssec:notations}). A DP constraint is then defined as follows.

\begin{definition}\label{def:differential_privacy}
 The transcript $Y^{(j)}$ is $(\epsilon,\delta)$-differentially private ($(\epsilon,\delta)$-DP) if for all $A \in \mathscr{Y}^{(j)}$, $u \in \cU$ and neighboring data sets $x,x' \in \mathcal{X}^{n}$ differing in one individual datum it holds that
\begin{equation}\label{eq:defining_DP}
\P \left( Y^{(j)} \in A | X^{(j)} = x, U = u \right) \leq  e^{\epsilon} \P \left( Y^{(j)} \in A | X^{(j)} = x', U = u \right)  + \delta.
\end{equation}
\end{definition}

The above setting is concerned with distributed protocols for scenarios where multiple parties hold sensitive data and each publishes a differentially private summary, without sharing raw data between them. This approach is common in cases like separate studies by different hospitals on the same population, where privacy concerns prevent direct data pooling. We note that, in the above definition, the outcome of the shared source of randomness $U$ does not diminish the privacy guarantee, even when it is publicly available. This is because the shared randomness is not used in the privacy mechanism, but as a means to enhance coordination between the transcripts, which allows the transcripts to be more informative about the underlying signal whilst each transcript is effectively sharing less information about their underlying individual data. 

Formally, having access to shared randomness means that $U$ can be defined on some probability space $(\cU,\mathscr{U}, \P^U)$, such that $U$ is independent of the data (i.e. by taking the appropriate product space for $\P^{X,U}$). Having no access to shared randomness effectively corresponds to considering $U$ to be a degenerate random variable, or $\mathscr{U} = \{ \emptyset, \cU \}$. In order to stream line the notation between these two setups, we shall refer to the triplet $( T, \{\left(\P^{Y^{(j)}|X^{(j)} = x, U=u}\right)_{x \in \cX^n, u \in \cU} \}_{j=1}^m, (\cU, \mathscr{U}, \P^U))$ as a \emph{$(\epsilon,\delta)$-DP shared randomness distributed testing protocol} with $\{\left(\P^{Y^{(j)}|X^{(j)} = x,U=u}\right)_{x \in \cX^n, u \in \cU} \}_{j=1}^m$ satisfying Definition \ref{def:differential_privacy} for general $U$ (i.e. general $(\cU, \mathscr{U}, \P^U)$). The class of such triplets, but with $\mathscr{U} = \{ \emptyset, \cU \}$, shall be referred to as \emph{\emph{$(\epsilon,\delta)$}-DP local randomness distributed testing protocols}. We shall denote these classes as $\mathscr{T}^{(\epsilon,\delta)}_{\texttt{SHR}}$ and $\mathscr{T}^{(\epsilon,\delta)}_{\texttt{LR}}$, respectively, and note that the former is a superset of the latter.

\subsection{Problem formulation}

The white-noise-with-drift model serves as a benchmark model for nonparametric testing and has been extensively studied outside of the DP setting, see \cite{ermakov1990asymptotically,ingster1993asymptotically,lepskii1992asymptotically,spokoiny_adaptive_1996,ingster2003nonparametric}. Furthermore, the problem bares a close relationship with ``classical'' nonparametric goodness-of-fit testing in the sense of \cite{an1933sulla,smirnov,cramer1928composition,von1928statistik} and other nonparametric testing problems through asymptotic equivalence, see Section 1.4 in \cite{ingster_nonparametric_2003} and references therein.

In the distributed setting, the $j=1,\dots,m$ machines each observe $i=1,\dots,n$ i.i.d, $X^{(j)}_i$ taking values in $\cX \subset L_2[0,1]$ and subject to the stochastic differential equation
\begin{equation}\label{eq:snwn_function_model_dynamics}
dX^{(j)}_{t;i} = f(t)dt + \sigma dW_{t;i}^{(j)}
\end{equation}
under $P_f$, with $t \mapsto W^{(1)}_{t;i},\dots,t \mapsto W^{(m)}_{t;i}$ i.i.d. Brownian motions and $f \in L_2[0,1]$ for $i=1,\dots,n$, with $\sigma > 0$ the known noise level for each observation. For notational convenience, we shall use $N=mn$ throughout the paper and will consider asymptotic regimes where $N \to \infty$. We note that, when $m=1$, we recover the classical white-noise-with-drift model.

We consider the canonical signal detection problem, where the goal is to test for the presence or absence of the ``signal component'' $f$. More formally, we consider testing the null hypothesis $H_0: f \equiv 0$ against the alternative hypotheses that
 \begin{equation}\label{eq:alternative_hypothesis}
f \in H_{\rho}^{s,R} \equiv H_{\rho}^{s,R,p,q} := \left\{ f \in \cB^{{s},R}_{p,q} : \|f\|_{L_2} \geq \rho \text{ and } \| f\|_{\cB^{s}_{p,q}} \leq R \right\}.
 \end{equation}
Here, the alternative hypothesis consists of ${s}$-smooth functions in a Besov space, with $\|\cdot\|_{\cB^{s}_{p,q}}$ denoting the Besov-$({s},p,q)$-norm and $\cB^{{s},R}_{p,q} \subset L_2[0,1]$ corresponds to the Besov ball of radius $R$, see Section \ref{sec:appendix_wavelet_properties} in the Supplementary Material \cite{Cai2024FL-NP-Testing-Supplement} for the definitions. Besov spaces are a very rich class of function spaces. They include many traditional smoothness spaces such as H\"{o}lder and Sobolev spaces as special cases. We refer the reader to \cite{triebel1992theory} for a detailed discussion on Besov spaces. 

Using a wavelet transform, the above testing problem is equivalent the observations under the Gaussian sequence model, where each of the $j=1,\dots,m$ machines observes $i=1,\dots,n$ observations $X^{(j)}_i := (X_{lk;i}^{(j)})_{l \geq 1,k=1,\dots,2^l}$
\begin{equation}\label{eq:sequence_model}
X_{lk;i}^{(j)} = f_{lk} + \sigma Z_{lk;i}^{(j)},
\end{equation}
where the $Z_{lk;i}^{(j)}$'s are i.i.d. standard Gaussian. The equivalent hypotheses \eqref{eq:alternative_hypothesis} in the sequence model simply follows by replacing the $L_2[0,1]$-norm with the $\ell_2(\N)$-norm and the Besov space $\cB^{{s},R}_{p,q}$ set to $\{ f \in \ell_2(\N) : \| f\|_{\cB^{s}_{p,q}} < \infty \}$, where the Besov norm on the sequence space $\ell_2(\N)$ is defined as 
\begin{equation}\label{eq:besov_norm}
    \| f\|_{\cB^{s}_{p,q}} :=  \begin{cases}
       \left( \underset{l=1}{\overset{\infty}{\sum}} \left(2^{l({s}+1/2-1/p)} \left\| (f_{lk})_{k=1}^{2^l} \right\|_p\right)^{q}\right)^{1/q} &\text{ for } 1 \leq q < \infty, \\
       \; \underset{l\geq 1}{\sup} \; 2^{l({s}+1/2-1/p)} \left\| (f_{lk})_{k=1}^{2^l} \right\|_p   &\text{ for } q = \infty.
    \end{cases}
 \end{equation} 
In other words, the results for testing under DP derived for the sequence model of \eqref{eq:sequence_model} with hypothesis \eqref{eq:alternative_hypothesis} apply to the model described by \eqref{eq:snwn_function_model_dynamics} also, with the same corresponding hypothesis. 

Given a $\{0,1\}$ valued test $T$, where $T(Y) = 1$ corresponds to rejecting the null hypothesis, we define the testing risk sum of the type I and worst case type II error over the alternative class;
\begin{equation*}
\cR( H_{\rho}^{s,R} , T) = \P_{0} T(Y) + \underset{{f} \in H_{\rho}^{{s},R}}{\sup} \, \P_f T(Y).
\end{equation*}
For the range of values $2 \leq p \leq \infty$, $1 \leq q \leq \infty$, the \emph{minimax separation rate} in the unconstrained case is known to be $\rho \asymp {(\sigma^2/N)}^{\frac{{s}}{2{s} + 1/2}}$ (see e.g. \cite{ingster1993asymptotically}). This means that, for $\rho \gg  {(\sigma^2/N)}^{\frac{{s}}{2{s} + 1/2}}$, there exists a sequence of consistent tests $T \equiv T_N$ such that $\cR( H_{\rho}^{s,R} , T) \to 0$, whilst no such sequence of tests exists whenever $\rho \ll {(\sigma^2/N)}^{\frac{{s}}{2{s} + 1/2}}$. 
 
The minimax separation rate captures how the testing problem becomes easier, or more difficult, for different model characteristics. For distributed $(\epsilon,\delta)$-DP testing protocols, the minimax separation rate depends on the stringency of the privacy requirement, given by $\epsilon,\delta > 0$, as well as the model characteristics $m,n,s$ and $\sigma$. That is, we aim to find $\rho$ as a function of $m,n,s,\sigma,\epsilon,\delta$, such that $ \underset{T \in \mathscr{T}^{(\epsilon,\delta)}}{\inf} \cR( H_{\rho',R}^{s,p,q} , T) $ converges to either $0$ or $1$ depending on whether $\rho' \ll \rho$ or $\rho' \gg \rho$. The class of alternatives under consideration are subsets of the Besov ball $\cB^{{s},R}_{p,q}$, where $2 \leq p < \infty$, $1 \leq q \leq \infty$, which offers a framework for functions in \eqref{eq:snwn_function_model_dynamics} with specific smoothness characteristics. Our results extend easily to the case where $p = \infty$ at the cost of an additional logarithmic factor in the rate and a few additional technicalities in the proofs.

\subsection{Main results and our contribution}\label{ssec:our_contribution}

We quantify the difficulty of the federated testing problem outlined in the previous section in terms of the minimax separation rate. To achieve this, we present constructive distributed $(\epsilon,\delta)$-DP testing protocols that achieve consistent testing for the problem described above for certain values of $\rho$ (Theorem \ref{thm:attainment_nonadaptive} in Section \ref{sec:construction_of_tests}). Additionally, we establish matching minimax lower bounds, up to logarithmic factor, for the testing risk (Theorem \ref{thm:nonasymptotic_testing_lower_bound} in Section \ref{sec:lower-bounds}), providing a lower bound on the performance of distributed $(\epsilon,\delta)$-DP testing protocols.

Our analysis uncovers several novel and intriguing findings, which we briefly highlight here. The performance guarantees for the methods demonstrated in Section \ref{sec:construction_of_tests}, along with the lower bounds established in Section \ref{sec:lower-bounds}, indicate that the distributed $(\epsilon,\delta)$-DP testing problem for the hypotheses given in \eqref{eq:alternative_hypothesis} is governed by the minimax separation rate (up to logarithmic factors) 

\begin{equation}\label{eq:shared_randomness_rate_single_line}
    \rho^2 \asymp \left(\frac{\sigma^2}{mn}\right)^{\frac{2{s}}{2{s}+1/2}} + \left(\frac{\sigma^2}{m n^{3/2} \epsilon \sqrt{1 \wedge n \epsilon^2} } \right)^{\frac{2{s}}{2{s}+1}} \wedge \left( \left(\frac{\sigma^2}{\sqrt{m} n^{} \sqrt{1 \wedge n \epsilon^2}} \right)^{\frac{2{s}}{2{s}+1/2}} + \frac{\sigma^2}{mn^2 \epsilon^2} \right).
\end{equation}

The precise statement is deferred to Theorem \ref{thm:rate_theorem_shared_randomness}. 

The derived rate indicates that the distributed testing problem under privacy constraints undergoes multiple phase transitions, resulting in different regimes where $\epsilon$ affects the detection boundary differently. Specifically, a smaller $\epsilon$, which implies a stronger privacy guarantee, leads to an increased detection threshold. When $\delta$ decreases polynomially with $N$, its impact on the detection boundary is limited to a logarithmic factor, making its effect on the error rate minor compared to that of $\epsilon$.

For $m=1$, our theorems establish the optimal separation rate for nonparametric goodness-of-fit testing in the central DP setting, where all data is available on a single machine. When $\epsilon \lesssim 1/\sqrt{N}$, the privacy constraint affects the rate polynomially. In contrast, for $\epsilon \gtrsim 1/\sqrt{N}$, the rate approximates the classical minimax rate, up to logarithmic factors. Thus, the privacy constraint significantly impacts the rate only when $\epsilon$ is relatively small compared to the total number of observations $N$.

When $n=1$, we establish the optimal separation rate for the testing problem in the local DP setting. Here, $\epsilon$ can be seen to have a pronounced effect on the rate whenever $\epsilon \lesssim 1$.

In the general federated setting, with $m \gg 1$, we see that $m$ and $n$ come into play with different powers in the minimax rate whenever $\epsilon^2 \lesssim \sigma^{\frac{1}{2s+1}} m^{\frac{1}{4s+1}} n^{\frac{1/2-2s}{4s+1}}$. This means that if one distributes $N=mn$ observations across $m$ machines, the task becomes more challenging as the $N$ observations are spread over a greater number of machines, rather than having many observations on a smaller number of machines. This phenomenon is also observed in the comparable estimation setting of nonparametric regression, as recently investigated in \cite{cai2023private}. The phase transitions, however, are not observed in its estimation counterpart. We provide a detailed interpretation of the phase transitions in Section \ref{sec:main_results}. 

Our analysis also reveals that the minimax rate obtained in Theorem \ref{thm:rate_theorem_shared_randomness} becomes worse without access to shared randomness. This is revealed by Theorem \ref{thm:rate_theorem_shared_randomness} in Section \ref{sec:main_results}. For certain values of $\epsilon$, we show that the performance is strictly worse for methods that use only local randomness, and we exhibit optimal local and shared randomness methods for these regimes, respectively, in Sections \ref{ssec:procedure_II} and \ref{ssec:procedure_III}. The shared source of randomness does not reveal any information on the identity of the individuals comprising the sample, even when it is publicly available. This means that the shared randomness does not violate the privacy constraints, and the improvement in the rate is a direct result of the shared randomness. We comment on this further in Section \ref{sec:main_results}.

In many practical applications, the regularity parameter $s$ is unknown. In Section \ref{sec:adaptive-tests}, we extend the methods of Section \ref{sec:construction_of_tests} that attain the minimax rates, up to additional logarithmic factors, when the regularity is unknown. That is, establish that adaptive testing is possible under DP constraints with minimal additional cost in terms of the separation rate. 

The lower bound relies on several technical innovations, which are summarized in Section \ref{sec:lower-bounds}. There, we provide a sketch of the proof and highlight the key innovations. The optimal methods proposed in this paper rely on carefully tailored private test statistics that extend to the adaptive setting with minimal cost due to exponential concentration bounds. We describe the construction of these tests in Sections \ref{sec:construction_of_tests} and \ref{sec:adaptive-tests}.

\subsection{Related Work}\label{ssec:related_work}

The literature on the theoretical properties of DP can be mostly divided into those studying \emph{local} DP or \emph{central} DP. In local DP, the privacy protection is applied at the level of individual data entries or observations, which corresponds to $n=1$ in our setting. This is a stringent form of DP because each item of data is independently given privacy protection. In the other extreme, central DP, only the inference output needs to satisfy the DP constraint (i.e. $m=1$ in our setting), meaning that if the output is a test, only the final decision needs to satisfy a DP constraint. 

Locally differential private estimation has been studied in the context of the many-normal-means model, discrete distributions and parametric models in \cite{duchi2013local,duchi2018minimax,pmlr-v119-acharya20a_context_aware_LDP,min_barg}. The problem of density estimation under local DP constraints has been considered by \cite{duchi2018minimax,sart_density_LDP,kroll_density_at_a_point_LDP,butucea_LDP_adaptation}, of which the latter three works consider adaptation. In the context of hypothesis testing, \cite{pmlr-v48-rogers16,sheffet2018locally,acharya2018differentially,pmlr-v89-acharya19b,berrett2020locally,acharya_IEEE_identity_testing_part_I,acharya_IEEE_III} study testing under local DP for discrete distributions. Nonparametric goodness-of-fit testing under local DP is considered in \cite{dubois:hal-04426780,lam2022minimax}, where in \cite{lam2022minimax}, the authors consider adaptation as well. In \cite{butucea:hal-04425360}, the authors consider estimation of a quadratic functional under local differential privacy constraints, which has connections to goodness-of-fit testing. 

Settings in which the full data is assumed to be on a single server (i.e. $m=1$), where a single privacy constraint applies to all the observations, have also been studied for various parametric high-dimensional problems \cite{smith2011privacy,dwork2014analyze,bassily2014private,kamath2019privately,kamath2020private,cai2021cost,narayanan2022private,pmlr-v195-brown23a,cai2024optimal}. In \cite{lalanne2023about}, nonparametric density estimation with known smoothness is considered. When it comes to hypothesis testing under central DP, \cite{canonne2019structure} study simple hypothesis testing. \cite{NEURIPS2018_7de32147} considers uniformity and independence testing in the multinomial model and \cite{NEURIPS2020_private_identity_canonne,pmlr-v178-narayanan22a} study signal detection in the many-normal-means model. 
In \cite{NEURIPS2022_5bc3356e}, hypothesis testing in a  linear regression setting is considered.

Investigations into the more general federated setting have been much more limited, with estimation being considered in \cite{liu2020learning,pmlr-v206-acharya23a_user_level_LDP}, which study estimation for discrete distributions and \cite{levy2021learning,narayanan2022tight} which study mean estimation, \cite{cai2023private} which study nonparametric regression, and \cite{li2024federated} which consider sparse linear regression. In the paper \cite{canonne2023private}, the authors consider discrete distribution testing in a two server setting ($m=2$) with differing DP constraints.

\subsection{Organization of the paper}

The rest of the paper is organized as follows. In Section \ref{sec:main_results}, we present the main results of the paper, for the known smoothness case and the adaptive setting. Next, Section \ref{sec:construction_of_tests} presents the methods that achieve the optimal rates derived in Section \ref{sec:main_results}. In Section \ref{sec:adaptive-tests}, we extend these methods to be adaptive in the case that the smoothness is unknown. In Section \ref{sec:lower-bounds}, we present the lower bound theorems for the testing problem and give a sketch of its proof. Further proofs are deferred to the supplemental material to the article, \cite{Cai2024FL-NP-Testing-Supplement}. 

\subsection{Notation, definitions and assumptions}\label{ssec:notations}

Throughout the paper, we shall write $N := mn$. For two positive sequences $a_k$, $b_k$ we write $a_k\lesssim b_k$ if the inequality $a_k \leq Cb_k$ holds for some universal positive constant $C$. Similarly, we write $a_k\asymp b_k$ if $a_k\lesssim b_k$ and $b_k\lesssim a_k$ hold simultaneously and let $a_k\ll b_k$ denote that $a_k/b_k = o(1)$. 

We use the notations $a\vee b$ and $a\wedge b$ for the maximum and minimum, respectively, between $a$ and $b$. For $k \in \N$, $[k]$ shall denote the set $\{1,\dots,k\}$. Throughout the paper $c$ and $C$ denote universal constants whose value can differ from line to line. The Euclidean norm of a vector $v \in \mathbb{R}^d$ is denoted by $\|v\|_2$. For a matrix $M \in \R^{d \times d}$, the norm $M \mapsto \| M \|$ is the spectral norm and $\text{Tr}(M)$ is its trace. Furthermore, we let $I_d$ denote the $d \times d$ identity matrix. 

Throughout the paper, $\mathrm{d}_H$ is the Hamming distance on $\cX^n$ is defined as $\mathrm{d}_H(x,\breve{x}) := \sum_{i=1}^n \mathbbm{1}\left\{ x_i \neq \breve{x}_i \right\}$ for $x = (x_i)_{i =1}^n, \breve{x} = (\breve{x}_i)_{i =1}^n  \in \cX^n$. Furthermore, for a vector space $\cX$ and $x = (x_i)_{i \in [n]}\in \cX^n$, we shall write $\overline{x}$ for the average $n^{-1} \sum^n_{i=1} x_i$.

\section{Minimax optimal testing rates under privacy constraints}
\label{sec:main_results}

In this section, we discuss the main results in detail. We start the discussion with results for the oracle case where the regularity parameter is known in Section \ref{ssec:known_smoothness_main_results}. Section \ref{ssec:adaptation_main_results} describes the main results for when the regularity is not known.

\subsection{Description of the minimax separation rate}\label{ssec:known_smoothness_main_results}


We first give a precise statement concerning the minimax separation rate shown in \eqref{eq:shared_randomness_rate_single_line}.

\begin{theorem}\label{thm:rate_theorem_shared_randomness}
    Let ${s},R > 0$ be given and consider any sequences of natural numbers $m \equiv m_N$ and $n := N/m$ such that $N = mn \to \infty$, $1/N \ll \sigma \equiv \sigma_N = O(1)$, $\epsilon \equiv \epsilon_N$ in $(N^{-1},1]$ and $\delta \equiv \delta_N \lesssim N^{-(1 + \omega)}$ for any constant $\omega > 0$. Let $\rho$ a sequence of positive numbers satisfying \eqref{eq:shared_randomness_rate_single_line}.

    Then,
    \begin{equation*}
    \underset{T \in \mathscr{T}^{(\epsilon,\delta)}_{\texttt{SHR}}}{\inf} \; \; \cR( H_{\rho M_N}^{{s},R} , T) \to \begin{cases}
    0 \; \text{ for any } M_N^2 \gg \log \log (N) \log^{3/2} (N) \log(1/\delta), \\
    1 \; \text{ for any } M_N \to 0.
    \end{cases} 
    \end{equation*}
\end{theorem}
 
The proof of the theorem is given in Section \ref{sec:proofs_main_results} of the Supplementary Material \cite{Cai2024FL-NP-Testing-Supplement}. It is based on a combination of upper and lower bounds, where the lower bound is established in Section \ref{sec:lower-bounds}. The upper bound is given in Section \ref{sec:construction_of_tests}, where we present an $(\epsilon,\delta)$-DP distributed testing protocol that attains the rate in Theorem \ref{thm:rate_theorem_shared_randomness}. These upper and lower bounds are in fact non-asymptotic, meaning that they do not require the assumption that $N \to \infty$.

Theorem \ref{thm:rate_theorem_shared_randomness} shows multiple regime changes, where the distributed testing problem under privacy constraints undergoes a change in the minimax separation rate. Later on in this section, we highlight the different regimes and give an interpretation to each of them.

Theorem \ref{thm:rate_theorem_shared_randomness} considers the minimax rate for the class of distributed protocols with access to shared randomness, $\mathscr{T}^{(\epsilon,\delta)}_{\texttt{SHR}}$. Theorem \ref{thm:rate_theorem_local_randomness} below considers the minimax rate for the (strictly smaller) class of distributed protocols without access to shared randomness, $\mathscr{T}^{(\epsilon,\delta)}_{\texttt{LR}}$. Here, transcripts depend \emph{only} on their local data and possibly a local source of randomness. 



\begin{theorem}\label{thm:rate_theorem_local_randomness}
        Let ${s},R > 0$ be given and consider any sequences of natural numbers $m \equiv m_N$ and $n := N/m$ such that $N = mn \to \infty$, $1/N \ll \sigma \equiv \sigma_N = O(1)$ and $\epsilon \equiv \epsilon_N$ in $(N^{-1},1]$ and $\delta \equiv \delta_N \lesssim N^{-(1 + \omega)}$ for any constant $\omega > 0$. Let $\rho \equiv \rho_N$ a sequence of positive numbers satisfying \begin{equation}\label{eq:local_randomness_rate_single_line}
            \rho^2 \asymp  \left(\frac{\sigma^{2}}{mn}\right)^{\frac{2{s}}{2{s}+1/2}} + \left(\frac{\sigma^2}{m n^{2} \epsilon^2} \right)^{\frac{2{s}}{2{s}+3/2}} \wedge \left( \left(\frac{\sigma^2}{\sqrt{m} n^{} \sqrt{1 \wedge n \epsilon^2}} \right)^{\frac{2{s}}{2{s}+1/2}} + \left( \frac{\sigma^2}{mn^2 \epsilon^2} \right) \right).
            \end{equation}
        Then,
        \begin{equation*}
        \underset{T \in \mathscr{T}^{(\epsilon,\delta)}_{\texttt{LR}}}{\inf} \; \; \cR( H_{\rho M_N}^{{s},R} , T) \to \begin{cases}
        0 \; \text{ for any } M_N^2 \gg \log \log (N)  \log^{3/2} (N) \log(1/\delta), \\
        1 \; \text{ for any } M_N \to 0.
        \end{cases} 
        \end{equation*}
        \end{theorem}

The proof of Theorem \ref{thm:rate_theorem_shared_randomness} is given in Section \ref{sec:proofs_main_results} of the supplemental material \cite{Cai2024FL-NP-Testing-Supplement}. The theorem shows that, depending on the value of $\epsilon$, the minimax rate for protocols that do not have access to shared randomness is strictly worse than those for protocols that do have access to shared randomness. 

To more easily compare the two theorems, we provide a table in Table \ref{tab:rate_table} below, where we separate six ``regimes'' to aid interpretability below. Each of the regimes correspond to the dominating term in the minimax separation rates of Theorems \ref{thm:rate_theorem_shared_randomness} and \ref{thm:rate_theorem_local_randomness}. Which term dominates depends on the value of $\epsilon$, in comparison to $n,m,\sigma$, $s$ and the availability of shared randomness. 

\begin{table}[h!]
    \noindent\resizebox{\textwidth}{!}{
        \begin{tabular}{lcccccc}
      
        \textbf{} & \textbf{Regime 1} & \textbf{Regime 2} & \textbf{Regime 3} & \textbf{Regime 4} & \textbf{Regime 5} & \textbf{Regime 6} \\
        \hline 
        Shared $U$ & \(\left(\frac{\sigma^2}{mn} \right)^{\frac{2s}{2s+1/2}}\) & \(\left( \frac{\sigma^2}{m n^{3/2} \epsilon}  \right)^{\frac{2s}{2s+1}}\) & \(\left( \frac{\sigma^2}{m n^{2} \epsilon^2} \right)^{\frac{2s}{2s+1}}\) & \(\left( \frac{\sigma^2}{\sqrt{m} n^{}} \right)^{\frac{2s}{2s+1/2}}\) & \(\left( \frac{\sigma^2}{\sqrt{m} n^{3/2} \epsilon}  \right)^{\frac{2s}{2s+1/2}}\) & \( \frac{\sigma^2}{mn^2 \epsilon^2} \) \\
        Local only& \(\left(\frac{\sigma^2}{mn}\right)^{\frac{2s}{2s+1/2}}\) & \(\left(\frac{\sigma^2}{m n^{2} \epsilon^2} \right)^{\frac{2s}{2s+3/2}}\) & \(\left( \frac{\sigma^2}{m n^{2} \epsilon^2} \right)^{\frac{2s}{2s+3/2}}\) & \(\left( \frac{\sigma^2}{\sqrt{m} n^{}} \right)^{\frac{2s}{2s+1/2}}\) & \(\left( \frac{\sigma^2}{\sqrt{m} n^{3/2} \epsilon}  \right)^{\frac{2s}{2s+1/2}}\) & \( \frac{\sigma^2}{mn^2 \epsilon^2} \) 
 
        \end{tabular}
        }
    \caption{The minimax separation rates for the testing problem under privacy constraints, for both the local randomness and shared randomness settings. The rates are given up to logarithmic factors. The regimes are defined by the values of $\epsilon$ and the model characteristics $m,n,\sigma,s$. }\label{tab:rate_table}    
  \end{table}
  
The rates in Regimes 4, 5 and 6 are the same for both types of protocols, and these rates are attained by the same testing protocol for each of the classes, which does not require shared randomness. We shall refer to Regime 4, 5 and 6 as the ``low privacy-budget'' regimes, as these rates occur for relatively small values of $\epsilon$. 

We shall refer to Regimes 1, 2 and 3 as the ``high privacy-budget'' regimes, as these rates occur for relatively large values of $\epsilon$. These rates are achieved by a different protocol, for the classes of protocols with and without shared randomness, respectively. These protocols are given in Sections \ref{ssec:procedure_II} and \ref{ssec:procedure_III}. Proving the tighter lower bound in case of the class of protocols with access to local randomness only, requires a different technique to that of the class of shared randomness protocols, which we outline in Section \ref{sec:lower-bounds}. 

The improvement in the rate for the shared randomness protocols, compared to the local randomness protocols, are visible for Regime 2 and Regime 3, but also the values of $\epsilon$ for which the different regimes occur are different for the two types of protocols. The improvement in the rate is loosely speaking a consequence of the improved coordination between the servers made possible by shared randomization. We exhibit a distributed $(\epsilon,\delta)$-DP shared randomness protocol attaining the above rate in Section \ref{ssec:procedure_III}. 

In case of access to shared randomness, the high privacy-budget regime occurs whenever $ \epsilon \gtrsim \sigma^{-\frac{2}{4s+1}} m^{-\frac{2{s}}{4{s}+1}} n^{\frac{1/2-2{s}}{4{s}+1}}$ and $\epsilon \geq n^{-1/2}$, or $\sigma^{-\frac{1}{2s}} m^{-\frac{1}{2}} n^{\frac{1-2{s}}{4{s}}} \leq \epsilon < n^{-1/2}$. In the case of local randomness only, the high privacy-budget regimes occur for larger values of $\epsilon$, namely $\epsilon \gtrsim \sigma^{-\frac{2}{4s+1}} m^{\frac{1}{4s+1}} n^{\frac{1/2-2s}{4s+1}}$ if $\epsilon \geq n^{-1/2}$ or $ \epsilon \gtrsim \sigma^{-\frac{4}{4s-1}} m^{-\frac{1}{2}} n^{\frac{5/2-2s}{4s-1}}$ for $\epsilon \lesssim n^{-1/2}$, whenever $s > 1/4$. When $s \leq 1/4$ and no shared randomness is available, we are always in the low privacy-budget regime for the range of $N^{-1} \lesssim \sigma \lesssim 1$ and $N^{-1} < \epsilon \lesssim 1$ considered.

The testing protocols that attains the rate in the high privacy-budget regimes are given in Section \ref{ssec:procedure_II} and \ref{ssec:procedure_III}. These protocols bear some resemblance with the estimation strategy in \cite{cai2023private}, where transcripts constitute noisy, lower dimensional approximations of the original data.  
As $\epsilon$ increases, the dimensionality of these approximations increases and the rate improves. In Regime 2 and 3, the minimax rate is at least a polynomial factor of $\epsilon$ larger than the unconstrained, non-private minimax separation rate, which is attained in Regime 1 (up to a logarithmic factor). 

The observation that shared randomness can improve performance in our problem has been noted in other contexts involving distributed privacy and communication constraints, see for example \cite{acharya_IEEE_identity_testing_part_I,acharya_IEEE_identity_testing_part_II,acharya_IEEE_III,szabo2022optimal_IEEE,szabo2023distributedtesting}.  \cite{dubois:hal-04426780,butucea:hal-04425360} studies interactive versus non-interactive protocols and finds a difference in terms of minimax performance between the two in the local differential privacy setting. Interestingly, when $n=1$ (i.e. in the local differential privacy setting), we find the similar minimax rates for nonparametric goodness-of-fit testing in the high privacy-budget regimes, for the shared randomness and local randomness protocols, as they do for interactive and non-interactive protocols, whenever $\epsilon$ is in the high-budget regime. Although they study a different model, observations from smooth densities; it is interesting to see that the same rates seem to be attainable without sequential interaction, by using shared randomness instead. We note here that, when sequential- or interactive protocols are allowed, shared randomness can be employed in particular. In real applications without interaction, one should always use shared randomness if at all possible.

In the low privacy-budget case, i.e. Regimes 4, 5 and 6, the minimax rate for both local and shared randomness protocols coincides. We note that the regimes occur at different values of $\epsilon$ for the two types of protocols, however. 
Within the low privacy-budget range, we find essentially three different regimes. When $n^{-1/2} \leq \epsilon < \sigma^{-\frac{2}{4s-1}} m^{\frac{1/4-s}{4s+1}} n^{\frac{1/2-2s}{4s+1}}$, the rate is given by $(\frac{\sigma^2}{\sqrt{m} n^{}} )^{\frac{2{s}}{2{s}+1/2}}$. What is remarkable here, is that whilst the rate is polynomially worse in $m$ than the unconstraint rate, the rate is otherwise independent of $\epsilon$. This regime essentially corresponds to a setting where, even though a high privacy-budget strategy is not feasible, the desired level of privacy is achieved ``for free" with the locally optimal test statistic.

We exhibit an $(\epsilon,\delta)$-DP distributed testing protocol that attains this rate in Section \ref{ssec:procedure_I}. This strategy can roughly be described as first computing a locally optimal private test statistic -- a test statistic that would result in the optimal private test using just the local data -- and then averaging these private test statistics; essentially combining the power of the local tests. That this strategy performs well when the privacy constraint is sufficiently stringent can intuitively be explained as that it is easier to retain privacy when only (a private version of) a single real valued local test statistic is shared, rather than a (private approximation of) the original data.

When $\epsilon \lesssim n^{-1/2}$ within the low privacy-budget range, $\epsilon$ affects the rate polynomially. For the smallest values of $\epsilon$, i.e. $\epsilon \lesssim \sigma^{\frac{1}{2s+1}} m^{-\frac{1}{2}} n^{-\frac{1+s}{2s+1}}$, the rate is given by $\frac{\sigma^2}{mn^2 \epsilon^2}$. Strikingly, the regularity parameter does not appear in the rate in this regime. This phenomenon has the following explanation: for such small values of $\epsilon$, signals of size $\frac{\sigma^2}{mn^2 \epsilon^2}$ are of larger order than the local estimation rate of $(\frac{\sigma^2}{n})^{\frac{2{s}}{2{s}+1/2}}$. Consequently, signal can locally be estimated with high accuracy, and the bottleneck is purely the privacy constraint, not the high-dimensional nature of the problem.

In the case of central DP (i.e. $m=1$), only the low privacy-budget regime is observed. In this case, our results show that the non-private rate is attainable (up to logarithmic factors) for $\epsilon \gtrsim 1/ \sqrt{N}$. This is in contrast to the local DP setting, where both the high- and low privacy budget regimes are observed (depending on the values of $\sigma$ and $s$). Whenever $m$ is larger than say polynomial in $N$ in the low privacy-budget range, (i.e. $m \asymp N^\omega$ for some $\omega > 0$) the unconstrained minimax rate cannot be reached. 

To further illustrate the difference between these classes of protocols, we provide two plots in Figure \ref{fig:nonpar_testing_large_n}. The plots show the relationship between the minimax testing rate $\rho$ and $\epsilon$ for fixed values of $m,n,\sigma$ and four different choices for the regularity $s$. The regimes correspond to the six regimes in Table \ref{tab:rate_table}. 
The plots show that the shared randomness setting strictly improves the rate for certain values of $\epsilon$, and that the values of $\epsilon$ for which the different regimes occur are different for the two types of protocols.
A full case-wise breakdown of when each of the regimes occur is given in Section \ref{sec:supp:additional_rates} of the Supplementary Material \cite{Cai2024FL-NP-Testing-Supplement}. Below, we give an interpretation for each of the regimes.

\begin{figure}[ht]
    \centering
    \includegraphics[width=0.8\textwidth]{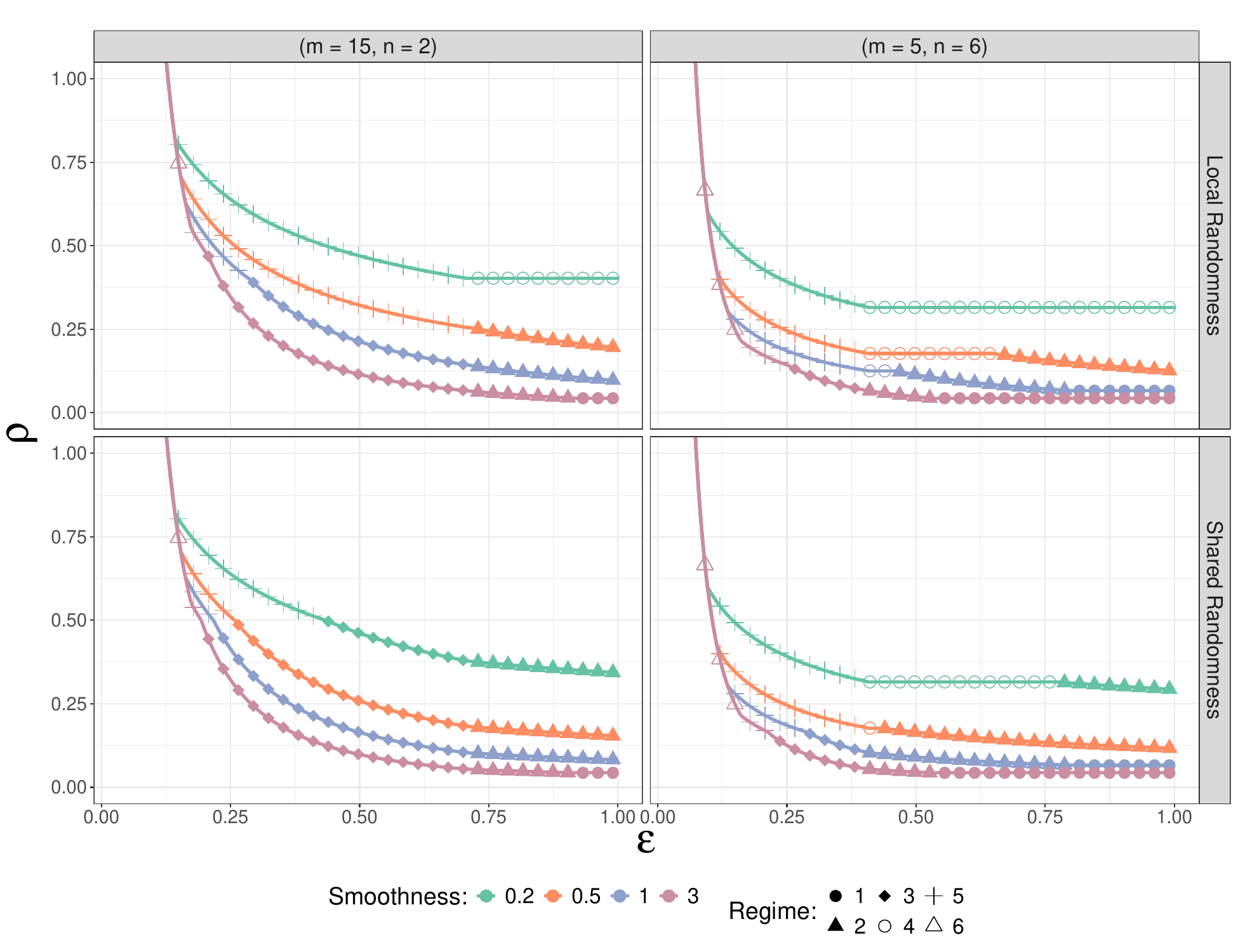}
    \caption{The relationship of the minimax testing rate $\rho$ and $\epsilon$, given by \eqref{eq:local_randomness_rate_single_line} and \eqref{eq:shared_randomness_rate_single_line}, for $(n,m)=(5,5)$ in the left column and $(n,m)=(2,15)$ in the right column, $\sigma=1$ and smoothness levels $s=1/5$, $s=1/2$, $s=1$ and $s=3$. The panels on the first row correspond to distributed $(\epsilon,\delta)$-DP (local randomness only) protocols (i.e. \eqref{eq:local_randomness_rate_single_line}), the bottom row corresponds to distributed $(\epsilon,\delta)$-DP protocols with shared randomness (i.e. \eqref{eq:shared_randomness_rate_single_line}). The regimes correspond to the six regimes (e.g. different rates) in Table \ref{tab:rate_table}.}\label{fig:nonpar_testing_large_n}
\end{figure}

What constitute ``moderate'' or ``large'' values, depends on the size of $m$ relative to $n$, as can be seen when comparing Figure \ref{fig:nonpar_testing_large_n}, which compares $N=30$ observations distributed either between $m=15$ servers with $2$ observations each, and $m=5$ servers with $n=6$ observations each. It can be seen that, as the local sample $n$ is larger compared to the number of times the total number of data points $N$ is divided $m$, the cost of privacy is less. This underlines the idea that, in large samples, it is easier to retain privacy.

When $\epsilon$ becomes ``very small'' (smaller than a threshold depending on $s$, $m$ and $n$), the smoothness starts to matter less and less, up to the point where the difficulty of the problem is no different for (very) different regularity levels. These scenarios correspond to settings where the privacy requirement underlying the problem is so stringent, that it effectively becomes the bottleneck of the testing problem.

\subsection{Adaptation}\label{ssec:adaptation_main_results}

In the previous section, we derived the minimax separation rate for the nonparametric distributed testing problem. However, the proposed tests constructed in Section \ref{sec:construction_of_tests} require knowledge of the regularity parameter $s$ of the underlying $f$. Typically, the regularity of the function is unknown in practice, necessitating the use of data-driven methods to find the best adaptive testing strategies.

Given that the regularity of the underlying signal class is unknown, it makes sense to consider the minimax testing risk 
\begin{equation*}
\underset{s \in [s_{\min} , s_{\max}]}{\sup}  \cR \left(  H^{s,R}_{ M_{N,s} \rho_s}, T \right),
\end{equation*}
for certain predetermined values $0 < s_{\min} < s_{\max} < \infty$. Here, we consider separation rates $\rho_s$ depending on the underlying smoothness. In the case that the true underlying smoothness is $s = s_{\min}$, the separation rate is relatively larger than when (for example) $s = s_{\max}$. In the case that the true smoothness $s$ is larger than $s_{\min}$, we would like to attain the smaller of the two rates $\rho_{s}$. 

In the non-privacy constraint setting, adaptation for the above risk can be achieved with only a minor additional cost in the separation rate (a $\log\log N$ factor). See for example Theorem 2.3 in \cite{spokoiny_adaptive_1996} or Section 7 in \cite{ingster_nonparametric_2003}. Theorem \ref{thm:adaptation_rate} below shows that also under privacy constraint, the optimal private rate can be attained by a protocol that is adaptive to the regularity parameter $s$, with minimal additional cost; at most a logarithmic factor. 

\begin{theorem}\label{thm:adaptation_rate}
    Let $0< s_{\min} < s_{\max} < \infty $, $R > 0$ be given and consider any sequences of natural numbers $m \equiv m_N$ and $n := N/m$ such that $N = mn \to \infty$, $1/N \ll \sigma \equiv \sigma_N = O(1)$, $\epsilon \equiv \epsilon_N$ in $(N^{-1},1]$ and $\delta \equiv \delta_N \lesssim N^{-(1 + \omega)}$ for any constant $\omega > 0$. 

    If $\rho$ a sequence of positive numbers satisfies \eqref{eq:shared_randomness_rate_single_line}, there exists a sequence of distributed $(\epsilon,\delta)$-DP testing protocols $T_N$ such that
    \begin{equation*}
        \underset{s \in [s_{\min} , s_{\max}]}{\sup}  \; \; \cR( H_{\rho M_N}^{{s},R} , T_N) \to \begin{cases}
        0 \; \text{ for any } M_N^2 \gg \log \log (N) \log^{5/2} (N) \log(1/\delta) \\
        1 \; \text{ for any } M_N \to 0.
        \end{cases} 
        \end{equation*}
    Furthermore, whenever $\rho$ satisfies \eqref{eq:local_randomness_rate_single_line}, there exists a sequence of distributed $(\epsilon,\delta)$-DP testing protocols $T_N$ using only local randomness such that the above display holds as well.
\end{theorem}

We construct such adaptive distributed $(\epsilon,\delta)$-DP testing protocols in Section \ref{sec:adaptive-tests} and their resulting performance proofs the above theorem. The adaptive methods can be seen as extensions of the methods exhibited in Section \ref{sec:construction_of_tests} for when the smoothness is known. The adaptive methods can essentially be seen as a multiple testing extension of the known smoothness methods, testing along a grid of smoothness levels between $s_{\min}$ and $s_{\max}$. The strain on the privacy budget stemming from conducting multiple testing procedures is limited, due to the fact that the cardinality of this grid is order $\log(N)$. The Type I error control is assured by a Bonferroni correction, which leverages the exponential bounds on the Type I error of the individual ``known smoothness tests''.

\section{Optimal differentially private testing procedures}\label{sec:construction_of_tests}

In this section, we construct distributed $(\epsilon,\delta)$-DP testing procedures that attain the minimax separation rates derived in Section \ref{sec:main_results}. 

The testing procedures are constructed in three steps. First, in Section \ref{ssec:procedure_I}, we construct a distributed $(\epsilon,\delta)$-DP testing procedure that uses only local randomness and that is optimal in the low privacy-budget regime described in the previous section. We refer to this procedure as $T_{\text{I}}$. Second, we construct two distributed $(\epsilon,\delta)$-DP testing procedures that use local randomness and shared randomness, respectively, and that are optimal in their respective high privacy-budget regimes. We refer to these procedures as $T_{\text{II}}$ and $T_{\text{III}}$ and describe them in Sections \ref{ssec:procedure_II} and \ref{ssec:procedure_III}, respectively.

The testing procedures differ in terms of the testing strategy. In the low privacy-budget case where $T_{\text{I}}$ is optimal, the testing strategy can be seen to consist of first computing a locally optimal private test statistic in each machine; that is, a test statistic that would result in the optimal private test using just the local data. The locally optimal test statistic is based on the squared Euclidean norm of the truncated observation. To deal with the nonlinearity of the Euclidean norm, the strategy appropriately restricts the domain of the clipped locally optimal test statistic, after which we employ a Lipschitz-extension to obtain a test statistic that is well-defined on the sample space and more robust to outliers than the Euclidean norm itself. The noisy version of this test statistic is locally optimal under privacy constraints, in the sense that a corresponding (strict) p-value test attains the lower bound rate (up to a logarithmic factor) as established by Theorem \ref{thm:rate_theorem_local_randomness} for the case where $m=1$. When $m>1$, the final test statistic is obtained by averaging the locally optimal private test statistics.

In the large $\epsilon$ regime, instead of computing a locally optimal test statistic, both $T^{}_{\text{II}}$ and $T^{}_{\text{III}}$ are based on truncated, clipped and noisy versions of the local observations. The key difference between the two is that the latter uses the same random rotation of the local observations, which is made possible by the availability of shared randomness.

Together, the methods prove Theorem \ref{thm:attainment_nonadaptive} below, which forms the ``upper bound'' part of the minimax separation rate described by Theorems \ref{thm:rate_theorem_local_randomness} and \ref{thm:rate_theorem_shared_randomness}. Unlike the formulation of the latter theorems, we note that the result is not asymptotic.

\begin{theorem}\label{thm:attainment_nonadaptive}
    Let ${s},R > 0$ be given. For all $\alpha \in (0,1)$, there exists a constant $C_\alpha > 0$ such that if
    \begin{equation}\label{eq:attainment_nonadaptive_local_randomness}
        \rho^2 \geq C_\alpha \left(\frac{\sigma^{2}}{mn}\right)^{\frac{2{s}}{2{s}+1/2}} + \left(\frac{\sigma^2}{m n^{2} \epsilon^2} \right)^{\frac{2{s}}{2{s}+3/2}} \wedge \left( \left(\frac{\sigma^2}{\sqrt{m} n^{} \sqrt{1 \wedge n \epsilon^2}} \right)^{\frac{2{s}}{2{s}+1/2}} + \left( \frac{\sigma^2}{mn^2 \epsilon^2} \right) \right),
    \end{equation}
    there exists a distributed $(\epsilon,\delta)$-DP testing protocol $T \equiv T_{m,n,s,\sigma}$ such that
    \begin{equation}
     \; \; \cR( H_{\rho M_N}^{{s},R} , T) \leq \alpha,
    \end{equation}
    for all natural numbers $m,N$ and $n = N/m$, $\sigma \in [1/N,\sigma_{\max}]$, $\epsilon \in (N^{-1},1]$, $\delta \leq N^{-(1 + \omega)}$ for any constant $\omega > 0$, $\sigma_{\max} > 0$ and a nonnegative sequence $M_N^2 \gtrsim \log \log (N) \log^{3/2} (N) \log(1/\delta)$.

    Similarly, for any $\alpha \in (0,1)$, there exists a constant $C_\alpha > 0$ such that if
    \begin{equation}\label{eq:attainment_nonadaptive_shared_randomness}
        \rho^2 \geq C_\alpha \left(\frac{\sigma^{2}}{mn}\right)^{\frac{2{s}}{2{s}+1/2}} + \left(\frac{\sigma^2}{m n^{3/2} \epsilon \sqrt{1 \wedge n \epsilon^2} } \right)^{\frac{2{s}}{2{s}+1}} \wedge \left( \left(\frac{\sigma^2}{\sqrt{m} n^{}} \right)^{\frac{2{s}}{2{s}+1/2}} + \left( \frac{\sigma^2}{mn^2 \epsilon^2} \right) \right),
    \end{equation}
    we have that there exists a distributed $(\epsilon,\delta)$-DP shared randomness testing protocol $T \equiv T_{m,n,s,\sigma}$ such that 
    \begin{equation}
     \; \; \cR( H_{\rho M_N}^{{s},R} , T) \leq \alpha,
    \end{equation}
    for all natural numbers $m,N$ and $n = N/m$, $\sigma \in [1/N,\sigma_{\max}]$, $\epsilon \in (N^{-1},1]$, $\delta \leq N^{-(1 + \omega)}$ for any constant $\omega > 0$ and a nonnegative sequence $M_N^2 \gtrsim \log \log (N) \log^{3/2} (N) \log(1/\delta)$.
\end{theorem}

The proof of the theorem follows directly from the guarantees proven for each of the three testing protocols; we defer it to Section \ref{sec:supp:proofs_optimal_testing_strategies} in the supplement. Before giving the detailed construction of the three tests, we introduce some common notation. Let $\Pi_L$ denote the projection of elements $\R^\N$ onto the first $d_L := \sum_{l = 1}^{L} 2^l$ coordinates, where the elements as ordered and indexed as follows;
\begin{equation*}
\Pi_L x = \left( x_{11},\dots,x_{12}, x_{21},\dots,x_{14},\dots,  x_{L1},\dots,x_{L2^L},0,0,0,\dots\right).
\end{equation*}
We shall also use the notation $d_L := \sum_{l=1}^L 2^l$ and let $X^{(j)}_{L;i}$ denote vector in $\R^{d_L}$ formed by the first $d_L$ coordinates of $\Pi_L X^{(j)}_{i}$ and let $X^{(j)}_L = (X^{(j)}_{L;i})_{i \in [n]}$. Furthermore, we recall that for $v = (v_1,\dots,v_n) \in \cX^n$ for a vector space $\cX$, $\bar{v}$ denotes the vector space average $n^{-1} \sum_{i=1}^n v_i$. 

In order to obtain statistics with (uniformly) bounded sensitivity it is useful to bound quantities between certain thresholds. Formally, for $a,b,x \in \R$ with $a<b$, let $[ x ]_a^b$ denote \emph{$x$ clipped between $a$ and $b$}, that is
\begin{equation}\label{eq:def_clipping}
 [ x ]_{a}^b := \begin{cases}
 b \; \text{ if } x > b, \\ 
 x \; \text{ if } a \leq x \leq b \\ 
 a \; \text{ otherwise.}
 \end{cases}
\end{equation}
The distributed privacy protocols under consideration in this paper can be seen as noisy versions of statistics of the data. Roughly put, the ``amount'' of noise added depends on the \emph{sensitivity} of the statistics. This brings us to the concept of sensitivity. Formally, consider a metric $\mathrm{d}$ on a set $\cY$. Given $n$ elements $x=(x_1,\dots,x_n)$ in a sample space $\cX$, the $\mathrm{d}$-\emph{sensitivity at $x$} of a map $S: \cX^n \to \cY$ is
\begin{equation*}
\Delta_S(x) := \underset{\breve{x} \in \cX^n : \mathrm{d}_H(x,\breve{x}) \leq 1}{\sup} \mathrm{d} \left( S(x), S(\breve{x})\right),
\end{equation*}
where $\mathrm{d}_H$ is the Hamming distance on $\cX^n$ (see Section \ref{ssec:notations} for a definition). The $\mathrm{d}$-\emph{sensitivity} of $S$ is defined as $\Delta_S := \sup_x \Delta_S(x)$. In this paper, the main noise mechanism is the \emph{Gaussian mechanism}. The Gaussian mechanism yields $(\epsilon,\delta)$-differentially private transcripts for statistics that have bounded $L_2$-sensitivity, with the noise variance scaling with the $L_2$-sensitivity. See \cite{dwork2014algorithmic} for a thorough treatment.
We remark that for the rates in Regime 3 up until 6 in Table \ref{tab:rate_table}, $(\epsilon,0)$-DP can be attained by employing a Laplace mechanism instead. That is, for the values of $\epsilon$ for which Regime 3 up until 6 in Table \ref{tab:rate_table} are optimal, the test statistics in the sections have matching $L_1$- and $L_2$-sensitivity, so the Gaussian mechanism can be replaced by the Laplace mechanism instead in these regimes.

\subsection{Private testing procedure I: low privacy-budget strategy}\label{ssec:procedure_I}

In the classical setting without privacy constraints (and $m=1$), a rate optimal test for the hypotheses of \eqref{eq:alternative_hypothesis} is given by
\begin{equation}\label{eq:classically_optimal_test}
\mathbbm{1}\left\{ S^{(j)}_{L_{{}}}  > \kappa_{\alpha} \right\}, \text{ where } S^{(j)}_{L_{{}}} := \frac{1}{\sqrt{ d_{L_{}}}} \left( \left\| \sigma^{-1} \sqrt{n} \overline{X^{(j)}_{L_{}}} \right\|_2^2 - d_{L_{}} \right), 
\end{equation}
where $d_{L_{}} := \sum_{l=1}^{L_{}} 2^l$ and the rate optimal choice of $L$ is $L_{*} = \left \lceil \frac{1}{2s + 1/2} \log_2 (N) \right \rceil$. Under the null hypothesis, $S^{(j)}_{L_{*}}$ is Chi-square distributed degrees of freedom. Under the alternative hypothesis, the test statistic picks up a positive ``bias'' as $\| \sigma^{-1} \sqrt{n} \overline{X^{(j)}_{L_{*}}} \|^2 \sim \chi^2_{L_{*}}(\|\Pi_{L_{*}} f\|_2^2)$ under $\P_f$, which could surpass the critical value $\kappa_\alpha$ if $\sigma^{-2} n \|\Pi_{L_{*}} f\|_2^2$ is large enough.  Consequently, the level of the test is controlled by setting $\kappa_{\alpha}$ appropriately large. For a proof of its rate optimality, see e.g. \cite{gine_mathematical_2016}.

As is commonly the case for superlinear functions, the test statistic $S_{L;\tau}^{(j)}$ has poor sensitivity uniformly over the sample space, meaning that a change in just one datum can result in a large change in the test statistic. This means that it forms a poor candidate to base a privacy preserving transcript on. For example, one would need to add a substantial amount of noise guarantee DP for the statistic. To remedy this, we follow a similar strategy as proposed in \cite{NEURIPS2020_private_identity_canonne} and improved upon by \cite{pmlr-v178-narayanan22a}. We construct a clipped and symmetrized version of the test statistic above, which has small sensitivity on a set $\cC_{L;\tau}$, in which $X^{(j)}$ takes values with high probability. We define the test statistic explicitly on $\cC_{L;\tau}$ only. By a version of the McShane--Whitney--Extension Theorem, we obtain a test statistic with the same sensitivity that is defined on the entire sample space. 

Consider for $\tau > 0$, $L \in \N$, $d_L := \sum^L_{l=1} 2^l$ and $V^{(j)}_{L;\tau} \sim \chi^2_{d_L}$ independent of $X^{(j)}$ the random map from $(\R^{d_L})^n$ to $\R$ defined by
\begin{equation}\label{eq:clipped_Lips_stat_rescaled}
\tilde{S}_{L;\tau}^{(j)}(x) = \left[\frac{1}{\sqrt{d_L}} \left(\left\| \sigma^{-1} \sqrt{n} \overline{ x} \right\|^2_2 - {V^{(j)}_{L;\tau}} \right)\right]_{-\tau}^{\tau}.
\end{equation}  
For any $\tau$, this test statistic $\tilde{S}_{L;\tau}^{(j)}(X^{(j)}_{L})$ can be seen to have mean zero and bounded variance under the null hypothesis, by similar reasoning as for the test statistic in \eqref{eq:classically_optimal_test} (see the proof of Lemma \ref{lem:typeI_typeII_error_control_Lipschitz_ext_test_privacy} for details). 

Loosely speaking, the test statistic $\tilde{S}_{L;\tau}^{(j)}(X^{(j)}_L)$ retains the signal as long as $\tau > 0$ is chosen appropriately in comparison to the signal size (i.e. $\|\Pi_{L} f\|_2^2$) and has good sensitivity for ``likely'' values of $X^{(j)}$ under $\P_f$, but not uniformly over the sample space. We make the latter statement precise as follows.

Let $K_\tau = \lceil 2\tau D^{-1}_\tau \rceil$ and consider the set $\cC_{L;\tau} = \cA_{L;\tau} \cap \cB_{L;\tau}$, where
\begin{align}\label{eq:def_C1_C2}
\cA_{L;\tau} &= \left\{ (x_i) \in (\R^\infty)^n :  \bigg|\| \sigma^{-1} \textstyle \sum_{i \in \cJ} \Pi_L x_i \|_2^2 - kd_L \bigg| \leq \frac{1}{8} k D_\tau n \sqrt{d_L} \;\; \forall \cJ \subset [n], |\cJ| = k \leq K_\tau \right\}, \\
\cB_{L;\tau} &= \left\{ (x_i) \in (\R^\infty)^n : \left| \langle \sigma^{-1} \Pi_L x_i, \sigma^{-1} \textstyle \sum_{k \neq i} \Pi_L x_k \rangle \right| \leq  \frac{1}{8} k D_\tau n \sqrt{d_L},  \;\; \forall i=1,\dots,n \right\}. \nonumber
\end{align}
Lemma \ref{lem:concentration_on_C} in the supplement shows that $X^{(j)}$ concentrates on $\cC_{L;\tau}$ when the underlying signal is, roughly speaking, not too large compared to $\tau$ (in particular under the null hypothesis).

It can be shown that, on the set $\cC_{L;\tau}$, $x \mapsto S^{(j)}(x)$ is $D_\tau$-Lipschitz with respect to the Hamming distance, see Lemma \ref{lem:Lipschitz_constant_on_C} in the supplement. Lemma \ref{lem:Lipschitz_extension_theorem} in the supplement shows that there exists a measurable function ${S}_{L;\tau}^{(j)} : (\R^{d_L})^n \to \R$, $D_\tau$-Lipschitz with respect to the Hamming distance, such that ${S}_{L;\tau}^{(j)}(X^{(j)}_L) = \tilde{S}_{L;\tau}^{(j)}(X^{(j)}_L)$ whenever $X^{(j)} \in \cC_{L;\tau}$. Lemma \ref{lem:Lipschitz_extension_theorem} is essentially the construction of McShane \cite{mcshane1934extension} for obtaining a Lipschitz extension with respect to the Hamming distance, but our lemma verifies in addition the Borel measurability of the resulting map.

The Lipschitz constant upper bounds the sensitivity of a test statistic that is Lipschitz continuous with respect to the Hamming distance. Specifically, we have that 
\begin{equation*}
\Delta_{S^{(j)}} = \underset{x,\breve{x} \in \ell_2(\N)^n : \mathrm{d}_H(x,\breve{x}) \leq 1}{\sup} \left| S^{(j)}(x) -  S^{(j)}(\breve{x})\right| \leq D_\tau.
\end{equation*}
Using the Gaussian mechanism, the transcripts
\begin{equation}\label{eq:transcript_private_Euclidean_norm_lips_ext_for_one_clipping}
Y^{(j)}_{L;\tau} = \gamma_\tau \breve{S}_{L;\tau}^{(j)}(X^{(j)}_L) + W^{(j)}_\tau, \quad \text{where } W^{(j)}_\tau \sim N(0,1) \text{ independent for } j \in [m],
\end{equation}
$\gamma_\tau = {\epsilon}/({D_\tau \sqrt{2 \mathfrak{c} \log (2 / \delta)}})$ and $\tau > 0$, are $(\epsilon/\sqrt{\mathfrak{c}},\delta)$-differentially private for any $\epsilon > 0$ (see e.g. \cite{dwork2014algorithmic}). These transcripts are mean zero and have bounded variance under the null hypothesis, so a test of the form
\begin{equation}\label{eq:Lipschitz_ext_test_privacy_for_one_clipping}
\varphi_\tau := \mathbbm{1} \left\{ \frac{1}{\sqrt{m}} \underset{j=1}{\overset{m}{\sum}} Y^{(j)}_{L;\tau} \geq \kappa (\gamma_\tau \vee 1)  \right\}
\end{equation}
has an arbitrarily small level for large enough $\kappa > 0$ (see Lemma \ref{lem:typeI_typeII_error_control_Lipschitz_ext_test_privacy_more_general} in the supplement). Furthermore, the lemma below shows that, if the signal size is large enough in the $\sum_{l=1}^{L} 2^l$ first coordinates, the above test enjoys a small Type II error probability as well.

\begin{lemma}\label{lem:typeI_typeII_error_control_Lipschitz_ext_test_privacy}
Consider the test $\varphi_\tau$ as defined by \eqref{eq:Lipschitz_ext_test_privacy_for_one_clipping}. If 
\begin{equation}\label{eq:tau_sandwich}
    \tau/4 \leq \frac{n\| f_L \|_2^2}{\log(N) \sqrt{2 {\mathfrak{c}} \log(2/\delta)} \sigma^2 \sqrt{d}}  \leq \tau/2 
\end{equation}
and
\begin{equation}\label{eq:lem:tau_specific_minimal_signal_size}
 \| \Pi_L f \|_2^2 \geq C_\alpha  \kappa \log^{} (N) \sqrt{{\mathfrak{c}} \log(1/\delta)} \left( \frac{\sqrt{2^L}}{\sigma^2 \sqrt{N} \sqrt{n} (\sqrt{n} \epsilon \wedge 1)} \right) \bigvee \left( \frac{1}{\sigma^2 N n\epsilon^2} \right)
\end{equation}
for $C_\alpha > 0$ large enough, it holds that $\P_f(1-\varphi_\tau) \leq \alpha$. 
\end{lemma}

A proof of the above lemma is given in Section \ref{ssec:supp:procedure_I} of the supplement. The above test is calibrated for the detection of signals size between $\tau/4$ and $\tau/2$. In order to detect signals of any size larger than the right-hand side of \eqref{eq:lem:tau_specific_minimal_signal_size}, we follow what is essentially a multiple testing procedure. For large signals, we need a larger clipping to detect them, as well as a larger set $\cC_{L;\tau}$ to assure that the data is in $\cC_{L;\tau}$ with high probability, as larger signals increase the probability of ``outliers'' from the perspective of the sensitivity of the $L_2$-norm. 

It turns out that a sufficient range of clipping thresholds to consider (for detecting the signals $f \in \cB^{{s},R}_{p,q}$ under consideration in Lemma \ref{lem:lips_test_up_rate}) is given by
\begin{equation}\label{eq:collection_of_threshold_lips_test}
\tau \in \mathrm{T}_L := \left\{ 2^{-k+2} \frac{ n (1-2^{-s})^{2 - 2/q} R^2}{\sigma^{2} \sqrt{2^{L}}}  : k=1,\dots,\lceil1+2\log_2( N R / \sigma )\rceil \right\}.
\end{equation}
The $(\epsilon, \delta)$-differentially private testing procedure $T_{\text{I}}$ is now constructed as follows. For each $\tau \in \mathrm{T}_L$, the machine transfers \eqref{eq:transcript_private_Euclidean_norm_lips_ext_for_one_clipping} with $\mathfrak{c} = {|\mathrm{T}_L|}$. By the independence of the Gaussian noise added in \eqref{eq:transcript_private_Euclidean_norm_lips_ext_for_one_clipping} for each $\tau \in \mathrm{T}_L$, the transcript $Y^{(j)} = \{ Y^{(j)}_{L;\tau} : \tau \in \mathrm{T}_L \}$ is $(\epsilon,\delta)$-differentially private (see e.g. Theorem A.1 in \cite{dwork2014algorithmic}). 

The test
\begin{equation}\label{eq:definition_full_test_lips_ext_privacy}
T_{\text{I}} :=  \mathbbm{1} \left\{ \underset{\tau \in \mathrm{T}_L}{\max} \, \frac{1}{\sqrt{m}} \underset{j=1}{\overset{m}{\sum}} Y^{(j)}_{L;\tau} \geq \kappa_\alpha \left( \frac{\epsilon}{D_\tau \sqrt{2 |\mathrm{T}_L| \log (2 / \delta)}} \vee 1 \right) \sqrt{\log |\mathrm{T}_L|}   \right\}
\end{equation}
then satisfies $\P_0 T_{\text{I}} \leq \alpha$ via a union bound and sub-exponential tail bound, we defer the reader to the proof of Lemma \ref{lem:lips_test_up_rate} for details. Furthermore, for $f \in \cB^{{s},R}_{p,q}$, we have $\|\Pi_L f\|_2\leq \|f\|_2 \lesssim R$. If $f$ in addition satisfies \eqref{eq:lem:tau_specific_minimal_signal_size}, there exists $\tau^* \in \mathrm{T}_L$ such that \eqref{eq:tau_sandwich} is satisfied and consequently
\begin{equation*}
\P_f (1 - T_{\text{I}}) \leq \P_f (1 - \varphi_{\tau^*}) \leq \alpha/2.
\end{equation*}
The optimal choice of $L$ depends on the regularity level of the signal $f$, balancing the approximation error $\| f - \Pi_L f\|_2^2$ and the right-hand side of \eqref{eq:lem:tau_specific_minimal_signal_size}, for which we defer the details to Section \ref{ssec:supp:procedure_I} in the supplement. To summarize, we have obtained the following lemma.

\begin{lemma}\label{lem:lips_test_up_rate}
For all $R>0$, $\alpha \in (0,1)$ there exists $\kappa_\alpha > 0$ and $C_\alpha > 0$ such that the test $T_{\text{I}}$ defined in \eqref{eq:definition_full_test_lips_ext_privacy} satisfies $\P_0 T_{\text{I}} \leq \alpha$. Furthermore, if $f \in \cB^{{s},R}_{p,q}$ is such that for some $L$ and $M_{N,\delta,\tau} =  \log^{} (N)  \sqrt{ \log \log(N R/\sigma)\log(N R/\sigma) \log(1/\delta)}$,
\begin{equation*}
\| \Pi_L f \|_2^2 \geq C_\alpha M_{N,\delta,\tau} \left( \frac{\sqrt{2^L}}{\sigma^2 \sqrt{N} \sqrt{n} (\sqrt{n} \epsilon \wedge 1)} \right) \bigvee \left( \frac{1}{\sigma^2 N n\epsilon^2} \right),
\end{equation*}
we have that $\P_f (1 - T_{\text{I}}) \leq \alpha$.
\end{lemma}

\subsection{Private testing procedure II: high privacy-budget strategy}\label{ssec:procedure_II}

In the high-privacy budget regime, we construct a testing procedure that consists essentially of two steps. In the first step, the data is truncated, clipped and averaged over the coordinates, after which Gaussian noise is added to obtain a private summary of the original data. Then, as a second step, the transcripts are averaged, and based on this average, a test statistic that is reminiscent of a chi-square test is computed in the central server. This is in contrast to the strategy of the previous section, where each server computes a (private version of) a chi-square test statistic.

The approach taken here is to divide the servers equally over the first $d_L$ coordinates (i.e. as uniformly as possible), where we recall the notation $d_L := \sum_{l=1}^L 2^l$. That is to say, for $L,K_L \in \N$, we partition the coordinates $\{ 1,\ldots,d_L \}$ into approximately $d_L/K_L$ sets of size $K_L$. The servers are then equally divided over each of these partitions and communicate the sum of the clipped $X^{(j)}_{L;i}$'s coefficients corresponding to their partition, were we also recall that the notation $X^{(j)}_{L;i}$ denotes the vector in $\R^{d_L}$ formed by the first $d_L$ coordinates of $\Pi_L X^{(j)}_{i}$.

More formally, take $K_L =  \lceil n \epsilon^2 \wedge d_L \rceil$ and consider sets $\cJ_{lk;L} \subset [m]$ for indexes $(l,k) \in \{ l=1,\dots, L, k = 1,\dots,2^l \} =:I_L$, such that $| \cJ_{lk;L} | = \lceil \frac{m K_L}{d_L} \rceil$ and each $j \in \{1,\dots,m\}$ is in $\cJ_{lk;L}$ for at least $K_L$ different indexes $k \in \{1,\dots,d_L\}$. For $(l,k) \in I_L$, $j \in \cJ_{lk;L}$, generate the transcripts according to
\begin{equation}\label{eq:transcript_impure_DP_coordinate_wise}
Y^{(j)}_{lk;L}|X^{(j)} \equiv Y^{(j)}_{lk;L}(X^{(j)})  = \gamma_L \, \sum_{i=1}^n \, [ \sigma^{-1} (X_i^{(j)})_{lk} ]_{-\tau}^\tau + W^{(j)}_{lk}
\end{equation}
with $\gamma_L = {\epsilon}/({2\sqrt{2 K_L \log(2/\delta)}\tau})$, $\tau = \tilde{\kappa}_\alpha \sqrt{\log(N/\sigma)}$ and $(W^{(j)}_{lk})_{j \in [m],(l,k) \in I_L}$ i.i.d. standard Gaussian noise. 

Since  $x \mapsto  \sum^n_{i=1}  [ \sigma (x_i^{(j)})_{lk} ]_{-\tau}^\tau$ has sensitivity bounded by $2 \tau$, for $k=1,\dots,K$, releasing 
\begin{equation*}
Y^{(j)}_L(X^{(j)})=(Y^{(j)}_{L,l_1 k_1}(X^{(j)}),\dots,Y^{(j)}_{L,l_{K_L}k_{K_L}}(X^{(j)}))
\end{equation*}
satisfies $(\epsilon,\delta)$-DP, see Lemma \ref{lem:L_2-sensitivity_private_coin_bound} in the supplement for details. 

If the privacy budget were of no concern, submitting the above transcripts with $2^L \asymp N^{1/(2s+1/2)}$ would be sufficient to construct a test statistic that attains the unconstrained rate of $\rho^2 \asymp N^{-2s/(2s+1/2)}$. Under (more stringent) privacy constraints, however, the optimal number of coordinates to be transmitted should depend on the privacy budget. Whenever $\epsilon \lesssim 1/\sqrt{n}$, it turns out that submitting just one coordinate is in fact rate optimal. Sending more than one coordinate leads to worse rates as the noise overpowers the benefit of having a higher dimensional transcript. As $\epsilon$ increases, the optimal number of coordinates to be transmitted increases as well. Whenever $\epsilon \gtrsim \sigma^{-\frac{2}{4s+1}} m^{\frac{1}{4s+1}} n^{\frac{1/2-2s}{4s+1}}$, the optimal number of coordinates to be transmitted is $2^L \asymp N^{1/(2s+1/2)}$.

The test
\begin{equation}\label{eq:test_impure_DP_coordinate_wise_strat}
T_{\text{II}} = \mathbbm{1}\left\{ \frac{1}{\sqrt{d_L}} \underset{(l,k) \in I_L}{\overset{}{\sum}} \left[ \left( \frac{1}{\sqrt{|\cJ_{lk;L}|}}  \underset{j \in \cJ_{lk;L}}{\overset{}{\sum}} Y^{(j)}_{lk;L}  \right)^2 -  \frac{n\epsilon^2}{4{K_L}\tau^2} -  1 \right] \geq \kappa_\alpha \left(\frac{n\epsilon^2}{4{K_L}\tau^2} \vee 1\right) \right\}
\end{equation}
satisfies $\P_0 T_{\text{II}} \leq \alpha$ by Lemma \ref{lem:protocol_II_priv_coin_bound} in the supplement whenever $\tilde{\kappa}_\alpha > 0$ and $\kappa_\alpha > 0$ are chosen large enough. 

The power that the test attains depends on the signal size up until resolution level $L$, i.e. $\|\Pi_L f \|_2$. Specifically, the test Type II error $\P_f (1 - T_{\text{II}}) \leq \alpha$ whenever
\begin{equation}\label{eq:f_condition_1coord}
    \| \Pi_L f \|_2^2 \geq C_\alpha \frac{ \log \log (N) \log(N) \log(1/\delta) 2^{(3/2)L}}{ m n^2 \epsilon^2}.
\end{equation}
The optimal choice of $L$ for is determined by the trade-off between the approximation error $\| f - \Pi_L f \|_2^2$ and the right-hand side of \eqref{eq:f_condition_1coord}. The proof of the following lemma is given in Section \ref{ssec:supp:procedure_II} of the supplement.

\begin{lemma}\label{lem:rate_priv_coin_privacy_pure_DP}
Take $\alpha \in (0,1)$. Suppose $f$ satisfies \eqref{eq:f_condition_1coord} and that $\epsilon \geq \frac{2^{L+1}}{\sqrt{mn}}$ for some $L \in \N$. Then, the distributed $(\epsilon,\delta)$-DP testing protocol $T_{\text{II}}$ of level $\alpha$ has Type II error $\P_f (1 - T) \leq \alpha$ for a large enough constant $C_\alpha > 0$ and $\tilde{\kappa}_\alpha > 0$, depending only on $\alpha$.
\end{lemma}

\subsection{Private testing procedure III: high privacy-budget shared randomness strategy}\label{ssec:procedure_III}

In this section, we construct a testing procedure that is based on the same principles as the one in the previous section, but with the difference that the servers a source of randomness. The transcripts are still based on the clipped and averaged coordinates of the truncated data, but instead of dividing the servers across the coordinates, we apply the same random rotation across the servers.

Next, we describe the testing procedure in detail. Consider for $L \in \N$ the quantities $d_L = \sum_{l=1}^L 2^l$ and $K_L = \lceil n \epsilon^2 \wedge d_L \rceil$ and let $U_L$ denote a random rotation uniformly drawn (i.e. from the Haar measure) on the group of random orthonormal ${d_L \times d_L}$-matrices.  

For $I_L:= \{ (l,k) : l=1,\dots, \lceil \log_2(K_L) \rceil, \, k=1,\dots,2^l \}$, $(l,k) \in I_L$ and $j=1,\dots,m$, generate the transcripts according to
\begin{equation}\label{eq:transcripts_pub_coin_privacy_impure_DP}
 Y^{(j)}_{lk;L} | (X^{(j)},U) =  \gamma_L \, \underset{i=1}{\overset{n}{\sum}}  [ (U X_{L;i}^{(j)})_{lk} ]^\tau_{-\tau} + W^{(j)}_{lk},
 \end{equation}
with $\gamma_L = \frac{\epsilon}{2\sqrt{2K_L \log(2/\delta) \log(N)}\tau}$, $\tau = \tilde{\kappa}_\alpha \sqrt{\log(N/\sigma)}$, $\tilde{\kappa}_\alpha > 0$ and $(W^{(j)}_l)_{j,l}$ i.i.d. centered standard Gaussian noise. By an application of Lemma \ref{lem:transcript_privacy_impure_DP_shared_randomness}, the transcript $Y^{(j)}_L := (Y^{(j)}_{lk;L})_{(l,k)\in I_L}$ is $(\epsilon,\delta)$-differentially private. 

In the shared randomness strategy above, we essentially only send the first $\sum_{l=1}^{\lceil \log_2(K_L) \rceil} 2^l$ coordinates. The random rotation $U_L$ ensures that, roughly speaking, a sufficient amount of the signal is present in these first coordinates, with high probability. 

We then construct the test
\begin{equation}\label{eq:test_pub_coin_privacy_impure_DP}
T_{\text{III}} = \mathbbm{1}\left\{ \frac{1}{\sqrt{K_L}} \underset{(l,k) \in I_L}{\overset{}{\sum}} \left[ \left( \frac{1}{\sqrt{m}}  \underset{j =1}{\overset{m}{\sum}} Y^{(j)}_{lk;L}  \right)^2 -  n\gamma^2_L -  1 \right] \geq \kappa_\alpha \left(n \gamma^2_L \vee 1\right) \right\},
\end{equation}
 which satisfies $\P_0 \varphi \leq \alpha/2$ by Lemma \ref{lem:procedure_III_test_up_rate} in the supplement, for $\kappa_\alpha > 0$ large enough. The lemma below is proven in Section \ref{sec:supp:proofs_optimal_testing_strategies} of the supplement, and yields that the Type II error of the test satisfies $\P_f (1 - T_{\text{III}}) \leq \alpha$ whenever the coordinates up to resolution level $L$ are of sufficient size. The optimal value for $L$ depends on the truncation level $s$, and is chosen by balancing the approximation error $\| f - \Pi_L f \|_2^2$ and the right-hand side of \eqref{eq:f_condition_pub_multicoord}, we defer the reader to Section \ref{ssec:supp:procedure_III} of the supplement for details.

\begin{lemma}\label{lem:rate_pub_coin_privacy_impure_DP} 
    The testing protocol $T_{\text{III}}$, with level $\alpha$ and has corresponding Type II error probability $\P_f (1 - T_{\text{III}}) \leq \alpha$ whenever
    \begin{equation}\label{eq:f_condition_pub_multicoord}
    \| \Pi_L f \|_2^2 \geq C_\alpha\frac{2^L \log(1/\delta) \log(N)}{m n^{} \sqrt{n \epsilon^2 \wedge 2^{L}} \sqrt{n \epsilon^2 \wedge 1}} 
    \end{equation}
    for constant $C_\alpha > 0$ and $\tilde{\kappa}_\alpha >0$ depending only on $\alpha$.
    \end{lemma}

\section{Adaptive tests under DP constraints}
\label{sec:adaptive-tests}

In the previous section we have derived methods that match (up to logarithmic factors) the theoretical lower bound established in Section \ref{sec:main_results}. The proposed tests, however, depend on the regularity parameter $s$ of the functional parameter of interest $f$. 

In this section we derive an distributed $(\epsilon,\delta)$-DP testing protocol that adapts to the regularity when it is unknown. This method attains the optimal rate of Theorem \ref{thm:adaptation_rate}, and consequently proves the aforementioned theorem.

The adaptive procedure builds on the tests constructed in Section \ref{sec:construction_of_tests} and combines them using essentially a multiple testing strategy. Roughly speaking, the method consists of taking approximately a $1/\log N$-mesh-size grid in the regularity interval $[s_{\min},s_{\max}]$, constructing optimal tests for each of the grid points and combining them using a type of Bonferroni's correction. By design, the tests constructed in Section \ref{sec:construction_of_tests} are based on sub-exponential private test statistics, which allows a combination of the test statistics with a Bonferroni correction of the order of $\log \log N$. 

Combining $\log N$ many $(\epsilon',\delta)$-differentially private transcripts using Gaussian mechanisms, results in a $(\epsilon,\delta)$-differentially private protocol, with $\epsilon = \epsilon' \sqrt{\log N}$. This means that the erosion of the privacy budget by conducting a test for each grid-point is limited to a logarithmic factor, means the method greatly improves over the potentially polynomially worse rate of a non-adaptive method.

The detailed adaptive testing procedures are given as follows. Let $\rho_s$ equal the right-hand side of \eqref{eq:local_randomness_rate_single_line} in case there is access to local randomness only, or the right-hand side of \eqref{eq:shared_randomness_rate_single_line} in case shared randomness is available. Let ${L_{s}}=\lfloor s^{-1}\log_2 (1/\rho_s)\rfloor \vee 1$ and define furthermore $\mathcal{S}:=\{L_{s_{\min}},\dots,L_{s_{\max}}\}$ such that $L_s\in \cS$ for all $s\in[s_{\min},s_{\max}]$. Furthermore, we note that the resulting ``collection of resolution levels'' satisfies $|\mathcal{S}| \leq C_{s_{\max}} \log N$ for some constant $C_{s_{\max}} > 0$ depending only on $s_{\max}$. 

Consider the first the case without access to shared randomness. We partition the collection of resolution levels $\mathcal{S}$, depending on the model characteristics, as follows. 
\begin{equation}\label{eq:partition_of_resolution_levels_local_randomness}
    \cS^{\texttt{LOW}}_{\texttt{LR}} = 
        \left\{ L \in \cS : 2^L \leq \epsilon \sqrt{mn} (1  + \sqrt{n}\mathbbm{1}_{\{\sqrt{n} \epsilon > 1\}})  \right\}, \quad \cS^{\texttt{HIGH}}_{\texttt{LR}} = \cS \setminus \cS^{\texttt{LOW}}_{\texttt{LR}}.
\end{equation}
If the true regularity $s_0$ is such that $L_{s_0} \in \cS^{\texttt{LOW}}$, the low privacy-budget test of Section \ref{ssec:procedure_I} (with $L = L_{s_0}$) is a rate optimal strategy. If $L_{s_0} \in \cS^{\texttt{HIGH}}$, the high privacy-budget test of Section \ref{ssec:procedure_II} is rate optimal.

For the case of shared randomness, the phase transitions occur for different values of $s\in[s_{\min},s_{\max}]$, or their respective resolution levels $L_s$. So in this case, we partition the collection of resolution levels as 
\begin{equation}\label{eq:partition_of_resolution_levels_shared_randomness}
    \cS^{\texttt{LOW}}_{\texttt{SHR}} = 
        \left\{ L \in \cS : 2^L \leq \epsilon^2 {mn}  \right\}, \quad \cS^{\texttt{HIGH}}_{\texttt{SHR}} = \cS \setminus \cS^{\texttt{LOW}}_{\texttt{SHR}}.
\end{equation}
Consider some $\cS' \subset \cS$. The ``adaptive version'' of the low privacy-budget test defined in \eqref{eq:definition_full_test_lips_ext_privacy} takes the form
\begin{equation}\label{eq:definition_full_test_lips_ext_privacy_adaptive}
T_{\text{I}}^S :=  \mathbbm{1} \left\{ \underset{L \in \cS',\, \tau \in \mathrm{T}_L}{\max} \; \; \frac{1}{\sqrt{m} \left( \gamma_L \vee 1 \right) \sqrt{\log |\mathrm{T}_L||\cS'|}} \; \underset{j=1}{\overset{m}{\sum}} Y^{(j)}_{L;\tau} \geq \kappa_\alpha    \right\},
\end{equation}
where $\mathrm{T}_L$ is as defined in \eqref{eq:collection_of_threshold_lips_test} and $Y^{(j)}_L = \{ Y^{(j)}_{L;\tau} : \tau \in \mathrm{T}_L \}$ is generated according to \eqref{eq:transcript_private_Euclidean_norm_lips_ext_for_one_clipping} for $L \in S$  with \begin{equation*}
    \gamma_\tau = \frac{\epsilon}{2 D_\tau \sqrt{ |\mathrm{T}_L| |\cS'| \log (4 / \delta)}}.
    \end{equation*}
The above choice of $\epsilon$ yields that $(Y^{(j)}_L)_{L \in S}$ is $(\epsilon/2,\delta/2)$-DP due to the Gaussian mechanism. The enlargement of the critical region, which is now effectively rescaled by $\sqrt{\log |\mathrm{T}|_L|\cS'|}$ instead of $\sqrt{\log |\mathrm{T}|_L}$, accounts for the potentially larger set of test statistics over which the maximum is taken. In the case of having access only to local sources of randomness, we set $\cS' = \cS^{\texttt{LOW}}_{\texttt{LR}}$. If $\cS^{\texttt{LOW}}_{\texttt{LR}}$ is empty, we set $T_{\text{I}} = 0$ instead, which forms an $(0,0)$-differentially private protocol.

In the case of having access to local sources of randomness only; if $\cS^{\texttt{HIGH}}_{\texttt{LR}}$ is non-empty, the adaptive version of the high privacy-budget test defined in \eqref{eq:test_impure_DP_coordinate_wise_strat} is given by
\begin{equation}\label{eq:test_impure_DP_coordinate_wise_strat_adaptive}
    T_{\text{II}} = \mathbbm{1}\left\{ \underset{L \in \cS^{\texttt{HIGH}}_{\texttt{LR}}}{\max} \; \frac{1}{\sqrt{d_L} \left(\eta_L \vee 1\right)} \underset{(l,k) \in I_L}{\overset{}{\sum}} \left[ \left( \frac{1}{\sqrt{|\cJ_{lk;L}|}}  \underset{j \in \cJ_{lk;L}}{\overset{}{\sum}} Y^{(j)}_{lk;L}  \right)^2 -  \eta_L -  1 \right] \geq \kappa_\alpha  \sqrt{\log |\cS^{\texttt{HIGH}}|} \right\},
\end{equation}
where the transcripts are generated according to \eqref{eq:transcript_impure_DP_coordinate_wise} for $L \in \cS^{\texttt{HIGH}}$, with $\gamma_L = {\epsilon}/({4 \sqrt{ |\cS^{\texttt{HIGH}}| K_L \log(4/\delta)}\tau})$, $\eta_L = \frac{n\epsilon^2}{4{K_L}\tau^2}$, $\tau = \tilde{\kappa}_\alpha \sqrt{\log(N/\sigma)}$. Due to the Gaussian mechanism, the transcripts satisfy an $(\epsilon/2,\delta/2)$-DP constraint. As before, if $\cS^{\texttt{HIGH}}$ is empty, set $T_{\text{II}} = 0$ instead.

In the case of having access to local randomness only, the adaptive testing procedure then consists of computing the tests $T_{\text{I}}^{\cS^{\texttt{LOW}}}$ and $T_{\text{II}}^{}$, for which the released transcripts satisfy $(\epsilon,\delta)$-DP. The final test is then given by  
\begin{equation}\label{eq:final_adaptive_test_local_randomness}
    T = T_{\text{I}}^{\cS^{\texttt{LOW}}} \vee T_{\text{II}}.
\end{equation}
In Section \ref{sec:supp:proofs_optimal_testing_strategies} in the supplement, it is shown that this test is adaptive and rate optimal (up to logarithmic factors), proving the first part of Theorem \ref{thm:adaptation_rate}.

In case of shared randomness, the adaptive version of the high privacy-budget test defined in \eqref{eq:test_pub_coin_privacy_impure_DP} is given by 
\begin{equation}\label{eq:test_pub_coin_privacy_impure_DP_adaptive}
    T_{\text{III}} = \mathbbm{1}\left\{ \underset{L \in S^{\texttt{HIGH}}_{\texttt{SHR}}}{\max} \; \frac{1}{\sqrt{K_L} \left(n \gamma_{L}^2 \vee 1\right)} \underset{(l,k) \in I_L}{\overset{}{\sum}} \left[ \left( \frac{1}{\sqrt{m}}  \underset{j =1}{\overset{m}{\sum}} Y^{(j)}_{lk;L}  \right)^2 -  n\gamma^2_L -  1 \right] \geq \kappa_\alpha \sqrt{\log |\cS|}  \right\},
\end{equation}
where the transcripts are generated according to \eqref{eq:transcripts_pub_coin_privacy_impure_DP} for $L \in S^{\texttt{HIGH}}_{\texttt{SHR}}$, $\gamma_L = \frac{\epsilon}{4\sqrt{K_L |S^{\texttt{HIGH}}_{\texttt{SHR}}| \log(4/\delta) \log(N)}\tau}$, $\tau = \tilde{\kappa}_\alpha \sqrt{\log(N/\sigma)}$. By similar reasoning as earlier, the transcripts $\{Y^{(j)}_L : L \in \S^{\texttt{HIGH}}_{\texttt{SHR}} \}$ are $(\epsilon/2,\delta/2)$-DP. If $S^{\texttt{HIGH}}_{\texttt{SHR}}$ is empty, we set $T_{\text{III}} = 0$ instead. 

The adaptive testing procedure in the case of shared randomness then consists of computing the tests $T_{\text{I}}^{S^{\texttt{LOW}}}$ and $T_{\text{III}}$, for which the released transcripts satisfy $(\epsilon,\delta)$-DP. The final test is then given by
\begin{equation}\label{eq:final_adaptive_test_shared_randomness}
    T = T_{\text{I}}^{S^{\texttt{LOW}}_{\texttt{SHR}}} \vee T_{\text{III}}.
\end{equation}

In the supplement's Section \ref{sec:supp:proofs_optimal_testing_strategies}, we prove that this test is adaptive, attaining the optimal rate for shared randomness protocols (up to logarithmic factors), giving us the second statement Theorem \ref{thm:adaptation_rate}.


\section{The minimax private testing lower bound}
\label{sec:lower-bounds}

In this section, we present a single theorem outlining the lower bound for the detection threshold for distributed testing protocols that adhere to DP constraints, with and without the use of shared randomness. The theorem directly yields the ``lower bound part'' of Theorems \ref{thm:rate_theorem_local_randomness} and \ref{thm:rate_theorem_shared_randomness} presented in Section \ref{sec:main_results}. In conjunction with Theorem \ref{thm:attainment_nonadaptive}, the theorem shows that the tests constructed in Section \ref{sec:construction_of_tests} are rate optimal up to logarithmic factors.

\begin{theorem}\label{thm:nonasymptotic_testing_lower_bound}
    Let ${s},R > 0$ be given. For all $\alpha \in (0,1)$, there exists a constant $c_\alpha > 0$ such that if
    \begin{equation}\label{eq:nonasymptotic_testing_lower_bound_local_randomness}
        \rho^2 \leq c_\alpha \left(\frac{\sigma^{2}}{mn}\right)^{\frac{2{s}}{2{s}+1/2}} + \left(\frac{\sigma^2}{m n^{2} \epsilon^2} \right)^{\frac{2{s}}{2{s}+3/2}} \wedge \left( \left(\frac{\sigma^2}{\sqrt{m} n^{} \sqrt{1 \wedge n \epsilon^2}} \right)^{\frac{2{s}}{2{s}+1/2}} + \left( \frac{\sigma^2}{mn^2 \epsilon^2} \right) \right),
    \end{equation}
    it holds that
    \begin{equation} 
     \underset{T \in \mathscr{T}^{(\epsilon,\delta)} }{\inf} \;\; \cR( H_{\rho}^{{s},R} , T) > 1 - \alpha,
    \end{equation}
    for all natural numbers $m,N$ and $n = N/m$, $\sigma > 0$, $\epsilon \in (N^{-1},1]$ and $\delta \leq N^{-(1 + \omega)}$ for any constant $\omega > 0$.

    Similarly, for any $\alpha \in (0,1)$, there exists a constant $c_\alpha > 0$ such that if
    \begin{equation}\label{eq:nonasymptotic_testing_lower_bound_shared_randomness}
        \rho^2 \leq c_\alpha \left(\frac{\sigma^{2}}{mn}\right)^{\frac{2{s}}{2{s}+1/2}} + \left(\frac{\sigma^2}{m n^{3/2} \epsilon \sqrt{1 \wedge n \epsilon^2} } \right)^{\frac{2{s}}{2{s}+1}} \wedge \left( \left(\frac{\sigma^2}{\sqrt{m} n^{}} \right)^{\frac{2{s}}{2{s}+1/2}} + \left( \frac{\sigma^2}{mn^2 \epsilon^2} \right) \right),
    \end{equation}
    we have that there exists a distributed $(\epsilon,\delta)$-DP shared randomness testing protocol $T \equiv T_{m,n,s,\sigma}$ such that 
    \begin{equation}\label{eq:nonasymptotic_testing_lower_bound} 
        \underset{T \in \mathscr{T}^{(\epsilon,\delta)}_{\texttt{SHR}} }{\inf} \; \; \cR( H_{\rho}^{{s},R} , T) > 1 - \alpha,
    \end{equation}
    for all natural numbers $m,N$ and $n = N/m$, $\sigma > 0$, $\epsilon \in (N^{-1},1]$ and $\delta \leq N^{-(1 + \omega)}$ for any constant $\omega > 0$.
\end{theorem}

The theorem states that, whenever the signal-to-noise ratio $\rho$ is below a certain threshold times the minimax separation rate, no distributed testing protocol can achieve a combined Type I and Type II error rate below $\alpha$. Its proof is lengthy and involves a combination of various techniques. We defer the full details of the proof to Section \ref{sec:proofs_lower_bound} of the supplement, but provide an overview of the main steps below.

For Steps 1, 2 and 3, there is no distinction between local and shared randomness. We use the same notation for distributed protocols in these steps, but simply assume $U$ is degenerate in the case of local randomness.
\begin{enumerate}[Step 1:]
    \item 
    The first step is standard in minimax testing analysis: we lower bound the testing risk by a Bayes risk,
    \begin{equation}\label{eq:lb_proof_many-normal_bayes_risk_lb}
        \underset{T \in \mathscr{T}}{\inf} \, \cR(H_\rho,T) \geq \underset{T \in \mathscr{T}}{\inf} \, \underset{\pi}{\sup} \, \left( \P_0(T(Y) = 1) + \int \P_{f} ( T(Y) = 0) d\pi(f) - \pi(H_\rho^c) \right),
    \end{equation}
    where $\mathscr{T}$ denotes either the class of local randomness or shared randomness $(\epsilon,\delta)$-DP protocols. This inequality allows the prior $\pi$ to be chosen adversarially to the distribution of the transcripts. This turns out to be crucial in the context of local randomness protocols, as is further highlighted in Step 4. The specific prior distribution is chosen to be a centered Gaussian distribution, with a finite rank covariance, where the rank is of the order $2^{L}$, for some $L \in \N$. This covariance is constructed in a way that it puts most of its mass in the dimensions in which the privacy protocol is the least informative, whilst at the same time it assures that the probability mass outside of the alternative hypothesis $\pi(H_\rho^c)$ is small. The particular choice for a Gaussian prior (instead of e.g. the two point prior in \cite{ingster2003nonparametric}) is motivated by Step 3. 
   \item
   In this step, we approximate the distribution of the transcripts with another distribution that results in approximately the same testing risk, but has two particular favorable properties for our purposes. 
    \begin{itemize}
        \item Whenever the distribution of a transcript satisfies an $(\epsilon,\delta)$-DP constraint with $\delta > 0$, the transcript's density can be unbounded on a set of small probability mass (proportional to $\delta$). Consequently,  the local likelihoods of the transcripts can have erratic behavior in the tails. To remedy this, we consider approximations to the transcripts that have bounded likelihoods, that differ only on a set that is negligible in terms of its impact on the testing risk. Furthermore, these approximating transcripts satisfy a $(\epsilon,2\delta)$-DP privacy constraint. These bounded likelihoods enable the argument of Step 5.
        \item Similarly, whenever $\delta > 0$, the distribution of the data conditionally on the transcript potentially has an unbounded density. For the argument employed in Step 3, we require a uniform abound on the density of the distribution of $X|Y$. We mitigate this by approximating the distribution of the transcripts by a distribution that induces a bounded density for the data conditionally on the transcript. This approximation of the original transcript satisfies a $(\epsilon,3\delta)$-DP constraint.
    \end{itemize}    
   Furthermore, we show that both approximations can be done in a way that the approximating transcript distribution is $(\epsilon,6\delta)$-DP. 

   \item By standard arguments, on can further lower bound the testing risk in \eqref{eq:lb_proof_many-normal_bayes_risk_lb} for a particular transcript distribution $\P^{Y|X,U} = \bigotimes^m_{j=1} \P^{Y^{(j)}|X^{(j)},U}$ and prior distribution $\pi$ by a quantity depending on the chi-square divergence between $\P_\pi^{Y|U=u}$ and $\P_0^{Y|U=u}$;
   \begin{equation}\label{eq:chi-sq-div-lb}
    1 -   \left( \sqrt{(1/2) \int \E_0^{Y|U=u}  \left(\left( \frac{d\P_\pi^{Y|U=u}}{d\P_0^{Y|U=u}} \right)^2 - 1 \right)^2 d\P^U(u) } + \pi(H_\rho^c)  \right).
    \end{equation}
    The likelihood ratio of the transcripts depends on the privacy protocol, and is difficult to analyze directly. We employ the technique developed in \cite{szabo2023distributedtesting}. Specifically, Lemma 10.1 in \cite{szabo2023distributedtesting}, which states, roughly speaking, that the inequality
   \begin{equation}\label{eq:brascamplieb-ineq}
    \E_0^{{Y}|U=u} \left( \frac{d\P_\pi^{Y|U=u}}{d\P_0^{Y|U=u}} \right)^2  \leq {G} \underset{j=1}{\overset{m}{\Pi}} \E_0^{{Y^{(j)}}|U=u} \left( \frac{d\P_\pi^{Y^{(j)}|U=u}}{d\P_0^{Y^{(j)}|U=u}} \right)^2 
   \end{equation}
    holds for a finite constant $0<{G}<\infty$ and equality with the smallest possible ${G}$ is attained whenever the conditional distribution of the data given the transcripts is Gaussian in an appropriate sense (we defer the details here to Section \ref{ssec:supp:step3} in the supplement). This result is a type of Brascamp-Lieb inequality \cite{brascamp1976best,lieb_gaussian_1990}. There is an existing literature on Brascamp-Lieb inequality in relation to information theoretical problems, in relation to mutual information \cite{carlen_subadditivity_2008,liu_brascamp-lieb_2016,liu_smoothing_2016}, in addition to the communication constraint testing problem in \cite{szabo2023distributedtesting}. That \eqref{eq:brascamplieb-ineq} has a ``Gaussian maximizer'' allows tractable analysis of the chi-square divergence in \eqref{eq:chi-sq-div-lb}, yielding that the latter display is further lower bounded by
    \begin{equation}
        1 - \sqrt{(1/2) \int \left( \mathrm{A}_u^\pi \mathrm{B}_u^\pi - 1 \right) d\P^U(u) } + \pi(H_\rho^c),
    \end{equation}
    where 
    \begin{equation}\label{eq:Au_Bu_defined}
        \mathrm{A}_u^\pi := \int e^{ f^\top \sum_{j=1}^{m} \Xi^j_u g} d(\pi \times  \pi)(f,g), \quad \mathrm{B}_u^\pi := \underset{j=1}{\overset{m}{\Pi}} \E_0^{{Y^{(j)}}|U=u} \left( \frac{d\P_\pi^{Y^{(j)}|U=u}}{d\P_0^{Y^{(j)}|U=u}} \right),
    \end{equation}
    where $\Xi^j_u$ denotes the covariance of (a subset of) the data $X^{(j)}_L$ (defined as in \eqref{eq:def_Xi_uj}) conditionally on the transcript $Y^{(j)}$ and $U=u$;
    \begin{equation}\label{eq:def_Xi_uj}
        \Xi_u^{j} := \E_0^{Y^{(j)}|U=u} \E_0\left[  \underset{i=1}{\overset{n}{\sum}} \sigma^{-1} {X}^{(j)}_{L;i} \bigg| Y^{(j)}, U=u \right] \E_0\left[  \underset{i=1}{\overset{n}{\sum}} \sigma^{-1} {X}^{(j)}_{L;i} \bigg| Y^{(j)}, U=u \right]^\top.
    \end{equation}
    Whilst the quantities $\mathrm{A}_u$ and $\mathrm{B}_u^\pi$ are still not fully tractable, sharp bounds for both are possible and form the content of Steps 4 and 5, respectively.

   \item As remarked earlier, class of local randomness protocols is a strictly smaller class. To attain the sharper (i.e. larger) lower bound for local randomness protocols, we exploit the fact that Step 2 allows us to choose the prior adversarially to the distribution of the transcripts. In particular, since $U$ is degenerate in the case of local randomness only, this means that the covariance of $\pi$ to be more diffuse in the directions in which $\Xi^j_u$ is the smallest. When considering shared randomness protocols, $U$ is not degenerate, and the lower bound follows by taking the covariance of $\pi$ to be an order $2^L$-rank approximation of the identity map on $\ell_2(\N)$. The bounds for $\mathrm{A}_u^\pi$ are
   \begin{equation}
    \mathrm{A}_u^\pi = \exp\left({ C  \frac{\rho^4}{c_\alpha 2^{3L}}  \text{Tr} \left( \Xi_u  \right)^2 }\right) \quad \text{ and } \quad \mathrm{A}_u^\pi \leq \exp\left({ C  \frac{\rho^4}{c_\alpha 2^{2L}} \|\Xi_u\| \text{Tr} \left( \Xi_u \right) }\right)
   \end{equation}
   for local randomness protocols and shared randomness protocols, respectively. 
   \item So far, Steps 1-4 have not used the fact that the transcripts are necessarily less informative than the original data, as a consequence of the transcripts being $(\epsilon,\delta)$-DP. In this step, we exploit the privacy constraint to argue that $\mathrm{A}_u^\pi$ and $\mathrm{B}_u^\pi$ are small at the detection boundary for $\rho$. 
   
   In order to capture the information loss due to privacy in $\mathrm{A}_u^\pi$, it suffices to bound the trace and operator norm of $\Xi_u$. The quantity $\Xi_u$ can be seen as the Fisher information of the finite dimensional submodel spanned by the covariance of $\pi$. This quantity, loosely speaking, captures how much information the transcript contains on the original data. In order to analyze $\Xi_u$, we rely on a ``score attack'' type of technique, as employed in \cite{cai2023score,cai2023private}. 
   
   The quantity $\mathrm{B}_u^\pi$ corresponds to the (product of) the local likelihoods of the transcripts. Whenever $\epsilon \geq 1/\sqrt{n}$, it suffices to consider the trivial bound 
   \begin{equation}
    \E_0^{{Y^{(j)}}|U=u} \left( \frac{d\P_\pi^{Y^{(j)}|U=u}}{d\P_0^{Y^{(j)}|U=u}} \right) \leq \E_0^{{X^{(j)}}|U=u} \left( \frac{d\P_\pi^{X^{(j)}}}{d\P_0^{X^{(j)}}} \right),
   \end{equation}
   and further bounding the right-hand side without privacy specific arguments. Whenever $\epsilon < 1/\sqrt{n}$, more sophisticated methods are need to capture the effect of privacy. Our argument uses a coupling method, which, combined with the fact that the likelihoods of the transcripts are bounded in our construction, allows us to obtain a sharp bound for $\mathrm{B}_u^\pi$. After obtaining the bounds in terms of the rank $2^L$ and $\rho$, the proof is finished by choosing $L$ such that the second and third term in \eqref{eq:chi-sq-div-lb} are balanced (minimizing their sum).
\end{enumerate}

\section{Discussion}\label{sec:discussion}

The findings in this paper highlight the trade-off between statistical accuracy and privacy in federated goodness-of-fit testing under differential privacy (DP) constraints. We characterize the problem in terms of the minimax separation rate, which quantifies the difficulty of the testing problem based on the regularity of the underlying function, the sample size, the degree of data distribution, and the stringency of the DP constraint. The minimax separation rate varies depending on whether the distributed testing protocol has access to local or shared randomness. Furthermore, we construct data-driven adaptive testing procedures that achieve the same optimal performance, up to logarithmic factors, even when the regularity of the functional parameter is unknown.

One possible extension of this work is to consider a more general distribution of the privacy budget across the servers. Our current analysis supports differing budgets to the extent that $\epsilon_j \asymp \epsilon_k$, $\delta_j \asymp \delta_k$, and $n_j \asymp n_k$. However, one could explore more heterogeneous settings where severs differ significantly in their differential privacy constraints and number of observations. Although this would complicate the presentation of results, the techniques developed in this paper could, in principle, be extended to such settings.

Another interesting direction is to consider multiple testing problems, where the goal is to test multiple hypotheses simultaneously. We anticipate that the framework, insights, and theoretical results provided in the current paper will serve as valuable resources for future studies in this domain.

Regarding adaptation, not much is known about the cost of privacy outside the local DP setting (i.e., one observation per server; $n=1$ in our context). Interestingly, the cost of adaptation is minimal in the privacy setting considered in this paper. It remains an open question whether this minimal cost is a general phenomenon, whether it can be characterized exactly, or whether the cost of adaptation is more severe in other settings. We leave these questions for future research.

\begin{supplement}
    \stitle{Supplementary Material to ``Federated Nonparametric Hypothesis Testing with Differential Privacy Constraints: Optimal Rates and Adaptive Tests''}
    \sdescription{In this supplement, we present the detailed proofs for the  main results in the paper ``Federated Nonparametric Hypothesis Testing with Differential Privacy Constraints: Optimal Rates and Adaptive Tests''.}
    \end{supplement}

\begin{funding}
The research was supported in part by NIH grants R01-GM123056 and R01-GM129781.    
\end{funding}

\bibliographystyle{imsart-number} 
\bibliography{references}       


\newpage

\pagebreak 

\begin{frontmatter}
    
\title{Supplementary Material to ``Federated Nonparametric Hypothesis Testing with Differential Privacy Constraints: Optimal Rates and Adaptive Tests''}
\runtitle{Federated Nonparametric Private Testing}

\begin{aug}
\author[A]{\fnms{T. Tony}~\snm{Cai}\ead[label=e1]{tcai@wharton.upenn.edu}},
\author[A]{\fnms{Abhinav}~\snm{Chakraborty}\ead[label=e2]{abch@wharton.upenn.edu}}
\and
\author[A]{\fnms{Lasse}~\snm{Vuursteen}\ead[label=e3]{lassev@wharton.upenn.edu}}
\address[A]{Department of Statistics and Data Science,
University of Pennsylvania\printead[presep={,\ }]{e1,e2,e3}}
\end{aug}

\begin{abstract}

In this supplement, we present the detailed proofs for the  main results in the paper ``Federated Nonparametric Hypothesis Testing with Differential Privacy Constraints: Optimal Rates and Adaptive Tests''.

\end{abstract}

\begin{keyword}[class=MSC2020]
    \kwd[Primary ]{62G10}
    \kwd{62C20}
    \kwd[; secondary ]{68P27}
\end{keyword}

    \begin{keyword}
    \kwd{Federated learning}
    \kwd{Differential privacy}
    \kwd{Nonparametric goodness-of-fit testing}
    \kwd{Distributed inference}
    \end{keyword}

\end{frontmatter}

\setcounter{page}{1}
\setcounter{section}{0}
\setcounter{equation}{0}
\renewcommand{\theequation}{S.\arabic{equation}}
\renewcommand{\thesection}{\Alph{section}}



    \section{Proof of the lower bound Theorem \ref{thm:nonasymptotic_testing_lower_bound}}\label{sec:proofs_lower_bound}

    In this section, we provide a proof for Theorem \ref{thm:nonasymptotic_testing_lower_bound}. The proof is divided into several steps, following the outline provided in Section \ref{sec:lower-bounds}. 
    
    It is convenient to introduce the following notation. In the notation that follows, we shall consider a local randomness protocol to consist of a triplet \[( T, \{\left(\P^{Y^{(j)}|X^{(j)} = x, U=u}\right)_{x \in \cX^n,u\in \cU} \}_{j=1}^m, (\cU,\mathscr{U},\P^U)),\] where the probability space $(\cU,\mathscr{U},\P^U)$ has the trivial sigma-algebra; $\mathscr{U} = \{ \emptyset, \cU\}$. This allows streamlining the argument for both lower bounds of Theorem \ref{thm:nonasymptotic_testing_lower_bound} into a single argument. Given an $(\epsilon,\delta)$-DP distributed protocol triplet \[( T, \{\left(\P^{Y^{(j)}|X^{(j)} = x, U=u}\right)_{x \in \cX^n,u\in \cU} \}_{j=1}^m, (\cU,\mathscr{U},\P^U)),\] we shall use the notation $P_f = \P_f^{X^{(j)}}$. For the Markov kernel \[(A,(x,u)) \mapsto \P^{Y^{(j)}|(X^{(j)},U)=(x,u)}(A),\] we shall use the shorthand $(A,(x,u)) \mapsto K^j(A|(x,u))$. If the transcript $Y^{(j)}$ is $(\epsilon,\delta)$-DP, the Markov kernel satisfies
    \begin{align*}\label{eq:privacy_constraint}
     K^j(A&|x_1,\dots,x_i,\dots,x_n,u) \leq e^{\epsilon}  K^j(A|x_1,\dots,x_i',\dots,x_n,u) + \delta \\
     &\text{ for all } A \in \mathscr{Y}, \; x_i',x_1,\dots,x_i,\dots,x_n \in \cX, \; i \in \{1,\dots,n\}. \nonumber
    \end{align*}
    
    When the parameter underlying the true distribution of the data $F$ is drawn from a prior distribution on $\ell_2(\N)$, we obtain that the distributed testing protocol satisfies the Markov chain 
    \begin{equation} \label{def: dist:test:MC}
         { \quad F} \qquad 
         \begin{matrix}
        \put(-10,-4){\vector(3,1){20}} &\quad(X^{(1)},U) &\put(0,5){\vector(1,0){20}}  &\qquad Y^{(1)}\quad&\put(-10,4){\vector(3,-1){20}} \\
        \put(-10,3){\vector(1,0){20}}&\quad\vdots &\put(0,4){\vector(1,0){20}}  &\qquad\vdots  \quad&\put(-10,4){\vector(1,0){20}} \\
        \put(-10,12){\vector(3,-1){20}}&\quad(X^{(m)},U) & \put(0,5){\vector(1,0){20}}  &\qquad Y^{(m)}\quad&\put(-10,5){\vector(3,1){20}}
        \end{matrix}  \qquad
         T,
        \end{equation}
    with $\P^{Y^{(j)}|U=u}_f = P_f K^j(\cdot|X^{(j)},u)$ for all $j \in [m]$. Let $K_u = \bigotimes_{j=1}^m K^j(\cdot|\cdot,u)$ denote the product conditional distribution with $U=u$, such that distribution of the collection of transcripts conditionally on $U=u$ then satisfies $\P^{Y|U=u}_f = P^m_f K_u()$.
    
    If the transcript $Y^{(j)}$ is $(\epsilon,\delta)$-DP and generated from $K^j_u, u \in \cU$, this translates to
    \begin{align}\label{eq:privacy_constraint}
     K^j(A&|x_1,\dots,x_i,\dots,x_n,u) \leq e^{\epsilon}  K^j(A|x_1,\dots,x_i',\dots,x_n,u) + \delta \\
     &\text{ for all } A \in \mathscr{Y}, \; x_i',x_1,\dots,x_i,\dots,x_n \in \cX, \; i \in \{1,\dots,n\}. \nonumber
    \end{align}
    
    \subsection{Step 1: Lower bounding the testing risk by the Bayes risk}
    
    As a first step, we lower bound the testing risk by the Bayes risk. Following the notation introduced above, let $T\equiv ( T, \{{K}^j \}_{j=1}^m, (\cU,\mathscr{U},\P^U))$ be an $(\epsilon,\delta)$-DP distributed testing protocol and let $
    {K}_u = \bigotimes_{j=1}^m {K}^j(\cdot|\cdot,u)$. The testing risk for $T$ can be written as 
    \begin{equation}
        \cR( H_{\rho}^{{s},R} , T) = \int P_0^m K_u T d\P^U(u) + \underset{f \in H_{\rho}^{s,R}}{\sup} \int P_{f}^m K_u(1-T) d\P^U(u).
    \end{equation}
    Consider $L \in \N$, $d_L = \sum_{l=1}^L 2^l$ and consider $\pi = N(0,c_\alpha^{-1/2} d^{-1}_L \rho^2 \bar{\Gamma})$ for a symmetric idempotent matrix $\bar{\Gamma} \in \R^{d_L \times d_L}$. Consider also the linear operator $\Psi_L : \R^{d_L} \to \ell_2(\N)$ defined by $\Psi_L \tilde{f} = f$ for
    \begin{equation}\label{eq : inverse wavelet transform}
    f_{lk} = \begin{cases}
        \tilde{f}_{lk} & \text{if } l \leq L, \\
        0 & \text{if } l > L,
    \end{cases}
    \end{equation}
    for $\tilde{f}=(\tilde{f}_{11},\dots,\tilde{f}_{L 2^L}) \in \R^{d_L}$. Since $\Psi_L$ is measurable, any probability distribution $\pi_L$ on $\R^{2^L}$, $\pi_L \circ \Psi^{-1}_L$ defines a probability measure on the Borel sigma algebra of $\ell_2(\N)$.
    
    Using that $0 \leq T \leq 1$, the above testing risk is lower bounded by the Bayes risk 
    \begin{equation}\label{eq:privacy_bayes_risk_breve-pre-pi-bound}
    \int \left( P_0^{m}{K}_u T  + \int  P_{f}^{m} {K}_u ( 1 - T) d\pi \circ \Psi^{-1}_L(f) \right) d\P^U(u) - \pi \circ \Psi^{-1}_L \left( (H_{\rho}^{s,R})^c \right).
    \end{equation}
    We highlight here that $\bar{\Gamma}$ can depend on $\{ {K}^j \}_{j=1}^m$. 
    
    To lower bound the testing risk further, it suffices to show that the prior $\pi$ has little mass outside of $H_\rho^{s,R}$. This follows by a standard Gaussian concentration argument, provided in Lemma \ref{lem:limited_gaussian_mass} below, from which it follows that, for $c_\alpha > 0$ small enough, it holds that $\pi \circ \Psi^{-1} ((H_\rho^{s,R})^c) \leq \alpha/4$. This yields that \eqref{eq:privacy_bayes_risk_breve-pre-pi-bound} is lower bounded by
    \begin{equation}\label{eq:privacy_bayes_risk_breve}
        \int \left( P_0^{m}{K}_u T  + \int  P_{f}^{m} {K}_u ( 1 - T) d\pi \circ \Psi^{-1}_L(f) \right) d\P^U(u) - \alpha/4.
        \end{equation}
    
    \begin{lemma}\label{lem:limited_gaussian_mass}
        Suppose $\rho \lesssim c_\alpha 2^{-Ls}$. Then, for any $c_\alpha > 0$ small enough, it holds that $\pi \circ \Psi^{-1}_L((H_\rho^{s,R})) \leq \alpha/4$.
    \end{lemma}
    \begin{proof}
        Let $f \in \ell_2(\N)$. It holds that $f \in H_\rho^{s,R}$ if and only if $\|f\|_{2}^2 \geq \rho^2$ and $\|f\|_{\cB^s_{p,q}} \leq R$. For the first of these events, we have that
        \begin{equation}\label{eq:concentration_condition_many_normal_means}
            \pi \circ \Psi^{-1}_L(f : \|f\|_2^2 \geq \rho^2 ) = \text{Pr}\left(  Z^\top \bar{\Gamma} Z  \geq \sqrt{c_\alpha} d_L \right),
        \end{equation}
        where $Z \sim N(0,I_{d_L})$. Using that $\bar{\Gamma}$ is idempotent, the right-hand side of the above display can be made arbitrarily small for small enough choice of $c_\alpha > 0$, by Lemma \ref{lem : Chernoff-Hoeffding bound chisq}. 
    
        To assure that $f \sim \pi \circ \Psi^{-1}_L$ concentrates on a Besov ball, we recall the definition of the Besov norm as given in \eqref{eq:besov_norm}. We have that 
        \begin{align*}
            \pi \circ \Psi^{-1}_L(f : \| f\|_{\cB^{s}_{p,q}} \leq R )  &= \pi \circ \Psi^{-1}_L \left(   2^{L({s}+1/2-1/p)} \left\| (f_{Lk})_{k=1}^{2^l} \right\|_p \leq R \right) \\
            &= \text{Pr}\left(   2^{-L/p} \left\| Z \right\|_p \leq C R / c_{\alpha}^{3/4} \right)
        \end{align*}
    where $Z \sim N(0,I_{2^L})$ and $C > 0$ a constant. The right-hand side of the above display can be made arbitrarily small for small enough choice of $c_\alpha > 0$, following from the fact that $Z$ is a standard normal vector $\E \|Z\|_p^p \lesssim 2^{L}$, where the constant depends on $p$ (see e.g. Proposition 2.5.2 in \cite{vershynin_high-dimensional_2018}) and Markov's inequality. 
    \end{proof}
    
    The larger $L$, the larger the effective dimension of the signal under the alternative hypothesis. Setting $L$ such that $2^L \asymp \rho^{-1/s}$ means that the requirement of the above lemma are satisfied and consequently the Gaussian prior most of it its mass in the alternative hypothesis. Recalling that $d_L \asymp 2^L$, the condition \eqref{eq:nonasymptotic_testing_lower_bound_local_randomness} in the case of local randomness protocols can be written as
    \begin{equation}\label{eq:nonasymptotic_testing_lower_bound_local_randomness_dL}
        \rho^2  \lesssim c_\alpha \sigma^2 \left( \frac{d_L^{3/2}}{ m n  (n\epsilon^2 \wedge d_L)} \bigwedge \left( \frac{\sqrt{d_L}}{\sqrt{m}n \sqrt{n \epsilon^2 \wedge 1}} \bigvee \frac{1}{mn^2 \epsilon^2} \right) \right),
        \end{equation}
    and in the case of shared randomness protocols, 
    \eqref{eq:nonasymptotic_testing_lower_bound_shared_randomness} can be written as
    \begin{equation}\label{eq:nonasymptotic_testing_lower_bound_shared_randomness_dL}
    \rho^2 \lesssim  c_\alpha  \sigma^2 \left( \frac{d_L}{ m n \sqrt{n \epsilon^2 \wedge 1} \sqrt{n\epsilon^2 \wedge d_L }} \bigwedge \left( \frac{\sqrt{d_L}}{\sqrt{m}n \sqrt{n \epsilon^2 \wedge 1} } \bigvee \frac{1}{mn^2 \epsilon^2} \right) \right).
    \end{equation}
    
    \subsection{Step 2: Approximating the distribution of the transcripts}
    
    In Step 3, we aim to use the Brascamp-Lieb type inequality Lemma 10.1 of \cite{szabo2023distributedtesting}, which we restate as Lemma \ref{lem:brascamp_lower_bound_summarization} below. However, the \emph{forward-backward channel} corresponding to ${K}_u$, $(x_1,x_2) \mapsto {q}_u(x_1,x_2)$, defined as 
    \begin{equation}\label{eq:forward_backward_kernel_def}
    {q}_u(x_1,x_2) := \int \frac{d {K}(\cdot|x_1,u) }{ d\P_0^{Y|U=u} } (y) \frac{d {K}(\cdot|x_2,u) }{ d\P_0^{Y|U=u}} (y) d\P^{Y|U=u}_0(y),
    \end{equation}
    is possibly unbounded when $\delta > 0$. To overcome this, we use Lemma \ref{lem:obtaining_bounded_kernel_densities} below to construct $(\epsilon,3\delta)$-DP Markov kernels $\{ \tilde{K}^j \}_{j=1}^m$ such that the corresponding forward-backward channel $\tilde{q}_u(x_1,x_2)$ is bounded. The construction of $\{ \tilde{K}^j \}_{j=1}^m$ is such that the total variation distance between $\tilde{K}^j$ and $\tilde{K}^j$ is small.
    
    \begin{lemma}\label{lem:obtaining_bounded_kernel_densities}
        For any $\alpha \in (0,1)$ and $(\epsilon,\delta)$-DP collection of Markov kernels $\{K^j\}_{j=1}^m$, there exists a collection of $(\epsilon,3\delta)$-DP kernels $\{\tilde{K}^j\}_{j=1}^m$ such that for some fixed constant $C > 0$,
        \begin{equation}\label{eq:uniform_bound_kernel_density}
        \underset{x \in \cX^{n} }{\sup} \; \frac{d\tilde{K}^j(\cdot|x)}{dP_0 \tilde{K}^j(\cdot|X^{(j)})} (y) < C, \; P_0 \tilde{K}^j(\cdot|X^{(j)})\text{-almost surely},
        \end{equation}
        whilst
        \begin{equation*}
        \| P_f (K^j(\cdot|X^{(j)}) - \tilde{K}^j(\cdot,X^{(j)})) \|_{\mathrm{TV}} \leq \frac{\alpha}{2m}.
        \end{equation*}
        \end{lemma}
        \begin{proof}
        For any $x \in \cX^n$ and set $A \in \mathcal{\mathscr{Y}}^{(j)}$, we have that
        \begin{equation*}
        K^j(A|x) = \int_A \frac{dK^j(\cdot|x)}{dP_0 K^j(\cdot|X^{(j)})} (y) d P_0 K^j(y|X^{(j)}) \leq 1,
        \end{equation*}
        where we note that by Lemma \ref{lem:bracamp_densities_exist}, the density of the integrand exists. So, by Markov's inequality, there exists a set $A_x^M \in \mathcal{\mathscr{Y}}^{(j)}$ such that
        \begin{equation*}
        \frac{dK^j(\cdot|x)}{dP_0 K^j(\cdot|X^{(j)})} (y) \leq M \; \text{ on } A_x^M,
        \end{equation*}
        whilst 
        \begin{equation}\label{eq:refer_once123}
        K^j\left((A_x^M)^c|x\right) \leq 1/M. 
        \end{equation}
        Define for all $x \in \cX$,
        \begin{equation*}
            \tilde{K}^j(B | x) := \frac{ K^j( B \cap A_{x}^M  |x)}{ K^j( A_{x}^M  | x )},
            \end{equation*}
        which can be rewritten as
        \begin{equation}\label{eq:refer_once123123}
        K^j \left(B \cap A_x^M | x \right) + K^j \left( (A_x^M)^c | x \right) \frac{ K^j( B \cap A_{x}^M  |x)}{ K^j( A_{x}^M  | x )}.
        \end{equation}
        Then, $\tilde{K}^j$ is $(\epsilon, 3\delta)$-DP whenever $M>4\delta^{-1}$; for any $x, x' \in \cX^n$ that are Hamming distance $1$-apart and $B \in \mathscr{Y}^{(j)}$,
        \begin{align*}
        \tilde{K}^j(B | x) &\leq K^j \left( B | x \right) + K^j \left( (A_x^M)^c | x \right) \frac{ K^j \left( B \cap A_{x}^M  | x \right)}{ K^j \left( A_{x}^M  | x \right)} \\
        &= K^j \left( B \cap A_{x'}^M | x \right) + K^j \left( B \cap (A_{x'}^M)^c | x \right) + K^j \left( (A_x^M)^c | x \right) \frac{ K^j \left( B \cap A_{x}^M  | x \right)}{ K^j \left( A_{x}^M  | x \right)} \\ 
        &\leq e^\epsilon K^j \left( B \cap A_{x'}^M | x' \right) + e^\epsilon K^j \left( B \cap (A_{x'}^M)^c | x' \right) + 2\delta + \frac{1}{M} \\
        &\leq e^\epsilon \tilde{K}^j \left( B | x' \right) + (1+e^\epsilon)M^{-1} + 2\delta,
        \end{align*}
        where the second to last inequality follows by \eqref{eq:refer_once123} and the last inequality follows by simply adding the nonnegative second term in \eqref{eq:refer_once123123}. Furthermore, its Radon-Nikodym derivative satisfies 
        \begin{align*}
        \frac{d\tilde{K}^j(\cdot|x)}{dP_0 \tilde{K}^j(\cdot|X^{(j)})} (y) &\leq 2 \mathbbm{1}_{A_x^M} \frac{dK^j(\cdot|x)}{dP_0 {K}^j(\cdot|X^{(j)})} (y) \leq 2M
        \end{align*}
        $P_0 \tilde{K}^j(\cdot|X^{(j)})$-almost surely. Moreover, it holds for any $f \in \ell_2(\N)$ that
        \begin{align*}
        \| P_f^n (K^j(\cdot|X^{(j)}) - \tilde{K}^j(\cdot,X^{(j)})) \|_{\mathrm{TV}} &\leq \int \| K^j(\cdot|x) - \tilde{K}^j(\cdot|x) \|_{\mathrm{TV}} dP_f^n(x) \\
        &\leq 2 \int | K^j\left((A_x^M)^c|x \right) | dP_f^n(x) \leq \frac{2}{M}.
        \end{align*}
        Since a choice $M > \delta^{-1} \vee 2m/\alpha$ yields the bound uniformly in $x \in \cX^n$, the result follows.
        \end{proof}
    
    By standard arguments (see also Lemma \ref{lem:privacy_testing_risk_approximation}), the first term in \eqref{eq:privacy_bayes_risk_breve} can be lower bounded as follows;
    \begin{align*}
        P_0^{m}{K}_u T &= P_0^{m}\tilde{K}_u T + P_0^{m}{K}_u T - P_0^{m}\tilde{K}_u T  \\
        &\geq  P_0^{m}\tilde{K}_u T - \sum_{j=1}^m \| P_0 \tilde{K}^j(\cdot|X^{(j)}) - P_0 {K}^j(\cdot|X^{(j)}) \|_{\mathrm{TV}}.
    \end{align*}
    The same argument on the second term yields that, we obtain that \eqref{eq:privacy_bayes_risk_breve} is lower bounded by
        \begin{equation}\label{eq:sup:approx-bayes-risk-I}
            \int \left( P_0^{m}\tilde{K}_u T  + \int  P_{f}^{m} \tilde{K}_u ( 1 - T) d\pi \circ \Psi^{-1}_L(f) \right) d\P^U(u) - 3\alpha/8,
    \end{equation}
    for an $(\epsilon,3\delta)$-DP distributed testing protocol $\tilde{T} = ( T, \{\tilde{K}^j \}_{j=1}^m, (\cU,\mathscr{U},\P^U))$.
    
    Another issue suffered by $(\epsilon,\delta)$-DP Markov kernels with $\delta > 0$, is that one has very poor control over the higher moments of the local likelihoods 
    \begin{equation*}
    {\cL}^{j}_{\pi,u}(y) := \frac{dP_\pi \tilde{K}^j(\cdot|X^{(j)},u)}{dP_0 \tilde{K}^j(\cdot|X^{(j)},u)} (y),
    \end{equation*}
    where $P_\pi(A) := \int P_f(A) d\pi \circ \Psi^{-1}_L(f)$, which are required to sufficiently bound the corresponding quantity $\mathrm{B}_u$ defined in \eqref{eq:sup:Bu_defined} in Step 5. Using similar ideas as in the proof of Lemma \ref{lem:obtaining_bounded_kernel_densities}, we can construct approximating kernels $\{ \breve{K}^j \}_{j=1}^m$ such that the likelihoods $\cL^{j}_{\pi,u}$ are bounded. This is the content of the following lemma.
    
    \begin{lemma}\label{lem:portmanteau_lemma}
        Let $\alpha \in (0,1)$, $\pi = N(0, c_\alpha^{-1/2} d^{-1/2}_L \rho^2 \bar{\Gamma})$ for an arbitrary positive semidefinite $\bar{\Gamma}$ and let $\{K^j\}_{j=1}^m$ correspond to a $(\epsilon,\delta)$-DP distributed protocol $T$ (i.e. $K^j$ satisfies \eqref{eq:privacy_constraint}). Furthermore, assume that $\epsilon \leq 1/\sqrt{n}$ and define for $j=1,\dots,m$ the events
        \begin{equation*}
        A_{j,u} := \left\{ y : |\cL_{\pi,u}^j (y) - 1| \leq \frac{4 m^{1/2}}{{\alpha}} \right\}
        \end{equation*}
        and define
       \begin{equation*}
       \tilde{K}^j(B |x,u) := K^j(B \cap A_{j,u} | x,u) + K^j( A_{j,u}^c | x,u) \frac{P_0K^j( B \cap A_{j,u}| X^{(j)},u) }{P_0K^j( A_{j,u}| X^{(j)},u)}.
       \end{equation*}
       Suppose in addition that $\delta \leq c_\alpha / m$. Then,
       \begin{enumerate}[label=(\alph*)]
           \item The collection $\{\tilde{K}^j\}_{j=1}^m$ are $(\epsilon,2\delta)$-DP Markov kernels.
           \item It holds $P_0\tilde{K}^j(\cdot|X^{(j)},u)$-a.s. that
           \begin{equation*}
           \tilde{\cL}_{\pi,u}^{j}(y)  := \int \frac{d\tilde{K}^j(y|x,u)}{dP_0 \tilde{K}^j(y|X^{(j)},u)} dP_\pi^n (x)
           \end{equation*}
           satisfies
       \begin{equation}\label{eq:privacy_portmanteau_lemma_toshow_b}
        | \tilde{\cL}_{\pi,u}^j(y) - 1| \leq \frac{5 m^{1/2}}{{\alpha}}.
       \end{equation}
           \item If $\rho$ satisfies \eqref{eq:nonasymptotic_testing_lower_bound_local_randomness} or \eqref{eq:nonasymptotic_testing_lower_bound_shared_randomness} with $c_\alpha>0$ small enough, it holds that 
           \begin{align*}
           P_0^{m}K T  + \int P_{f}^{m} K ( 1 - T) d\pi \circ \Psi^{-1}_L(f) \geq  P_0^{m}\tilde{K}T + \int P_{f}^{m} \tilde{K} ( 1 - T) d\pi \circ \Psi^{-1}_L(f) - \alpha,
           \end{align*}
           where $\tilde{K}$ is the product kernel corresponding to $\{\tilde{K}^j\}_{j=1}^m$.
           \item If $K^j$ satisfies \eqref{eq:uniform_bound_kernel_density}, then $\tilde{K}^j$ satisfies the same bound for some constant $C > 0$.
       \end{enumerate}
       \end{lemma}
    
    We prove this lemma at the end of this section. The lemma above, in combination with the same argument as before (e.g. Lemma \ref{lem:privacy_testing_risk_approximation}) allows us to replace the $(\epsilon,3\delta)$-DP Markov kernel $\tilde{K}^j$ of \eqref{eq:sup:approx-bayes-risk-I} with an $(\epsilon,6\delta)$-DP Markov kernel $\breve{K}^j$, whose transcripts have bounded likelihoods in the sense of \eqref{eq:privacy_portmanteau_lemma_toshow_b}. Furthermore, by part \emph{(c)} of the lemma, the Bayes risk of \eqref{eq:sup:approx-bayes-risk-I} is further lower bounded by
    \begin{equation}\label{eq:sup:approx-bayes-risk-II}
        \int \left( P_0^{m}\breve{K}_u T  + \int  P_{f}^{m} \breve{K}_u ( 1 - T) d\pi \circ \Psi^{-1}_L(f) \right) d\P^U(u) - \alpha/2,
    \end{equation}
    where the $\breve{K}^j$'s satisfy \eqref{eq:uniform_bound_kernel_density}.
    
    \subsubsection{Proof of Lemma \ref{lem:portmanteau_lemma}}
    \begin{proof}
    The first statement follows by Lemma \ref{lem:privacy_markov_kernel_construction_lemma} below. For the second statement, we first note that by Lemma \ref{lem:grand_coupling} proven in Section \ref{ssec:step5}, it holds that
    \begin{align*}
    \frac{c_\alpha^{1/4}}{m^{1/2}}  + \delta + \frac{c_\alpha}{m^{3/2}} &\geq (P_\pi-P_0) K^j\left( \{ |\cL^{j}_{\pi,u} - 1| \geq 4m^{1/2}/\alpha \} | X^{(j)},u\right)  \\
    &= P_0K^j\left((\cL^j_{\pi,u} - 1) \mathbbm{1}\{ |\cL^{j}_{\pi,u} - 1| \geq 4m^{1/2}/\alpha \} | X^{(j)},u \right) \\
    &\geq  4\frac{m^{1/2}}{\alpha} P_0K^j\left(  |\cL^{j}_{\pi,u} - 1| \geq 4m^{1/2}/\alpha  | X^{(j)},u \right),
    \end{align*}
    where the second inequality follows from the fact that $m^{1/2}/\alpha \geq 1$ and $\cL^{j}_{\pi,u} \geq 0$. Using that $\delta \leq c_\alpha / m$, we obtain that
    \begin{equation}\label{eq:privacy_good_tail_kernel_portmanteau_lemma_Aj_small_support}
    P_0K^j( A_{j,u}^c| X^{(j)},u) \leq  \frac{{\alpha(c_\alpha^{1/4} + c_\alpha(1+m^{-1}))}}{4{m}^{}} := \eta_\alpha.
    \end{equation}
    Since $K^j(B | x,u) \leq \tilde{K}^j(B |x,u)$ for all measurable $B \subset A_{j,u}$ and $P_0 \tilde{K}^j(\cdot |X^{(j)},u)$ has no support outside of $A_{j,u}$, it holds that 
    \begin{equation*}
    \frac{d{K}^j(\cdot|x,u)}{dP_0 \tilde{K}^j(\cdot|X^{(j)},u)} (y) \leq \frac{d{K}^j(\cdot|x,u)}{dP_0{K}^j(\cdot|X^{(j)},u)} (y),
    \end{equation*}
    for all $y \in A_{j,u}$ (and hence $P_0\tilde{K}^j(\cdot|X^{(j)},u)$-a.s.). Similarly, we have for $P_\pi$-a.s. all $x$'s that
    \begin{equation*}
    \frac{K^j( A_{j,u}^c | x,u)}{P_0K^j( A_{j,u}| X^{(j)},u)} \frac{dP_0K^j( \cdot \cap A_{j,u}| X^{(j)},u)}{dP_0 \tilde{K}^j(\cdot|X^{(j)},u)} (y) \leq \frac{K^j( A_{j,u}^c | x,u)}{P_0K^j( A_{j,u}| X^{(j)},u)}  \leq \frac{1}{1 - \eta_\alpha},
    \end{equation*}
    using that $K^j \leq 1$ and $P_0K^j( A_{j,u}| X^{(j)},u) \geq 1 - \eta_\alpha$. By standard arguments and the above two statements, it follows that 
    \begin{align*}
    \int \frac{d\tilde{K}^j(y|x,u)}{dP_0 \tilde{K}^j(y|X^{(j)},u)} dP_\pi^n (x) &\leq \mathbbm{1}_{A_{j,u}}(y) \int \frac{d{K}^j(y|x,u)}{dP_0 {K}^j(y|X^{(j)},u)} dP_\pi^n (x) + \frac{1}{1 - \eta_\alpha} \\
    &= \mathbbm{1}_{A_{j,u}}(y) {\cL}_{\pi,u}^j(y) + \frac{1}{1 - \eta_\alpha}.
    \end{align*}
    Applying the definition of the event $A_{j,u}$ and using that $\alpha \leq 1$, we obtain that for $c_\alpha > 0$ small enough
    \begin{align*}
    \tilde{\cL}_{\pi,u}^j - 1 \leq \frac{4 m^{1/2}}{{\alpha}} + \frac{1}{1 - \eta_\alpha} - 1 \leq \frac{5 m^{1/2}}{{\alpha}}.
    \end{align*}
    Using that $\tilde{\cL}_{\pi,u}^j - 1 \geq - 1$, we obtain \eqref{eq:privacy_portmanteau_lemma_toshow_b}, proving statement \emph{(b)}.
    
    For the third statement, we will aim to apply Lemma \ref{lem:privacy_testing_risk_approximation}. By the construction of $\tilde{K}^j$ and the triangle inequality,
    \begin{align*}
    \| P_0 K^j(\cdot|X^{(j)},u) - P_0 \tilde{K}^j(\cdot|X^{(j)},u) \|_{\mathrm{TV}} &\leq 2 \left\| P_0 K^j(\cdot \cap A_{j,u}^c |X^{(j)},u) \right\|_{\mathrm{TV}}.
    \end{align*}
    The latter is bounded by ${{\alpha}}/({2{m}^{}})$ (see \eqref{eq:privacy_good_tail_kernel_portmanteau_lemma_Aj_small_support}). Similarly,
    \begin{align*}
    \| P_\pi K^j(\cdot|X^{(j)},u) - P_\pi \tilde{K}^j(\cdot|X^{(j)}) \|_{\mathrm{TV}} &\leq 2 P_\pi K^j(A_{j,u}^c |X^{(j)},u).
    \end{align*}
    By Lemma \ref{lem:grand_coupling}, 
    \begin{align*}
    P_\pi K^j(A_{j,u}^c |X^{(j)},u) \leq \left(1+c_\alpha^{1/4} m^{-1/2}\right) P_0 K^j(A_j^c |X^{(j)},u)  + \delta  + \frac{c_\alpha}{m^{3/2}}.
    \end{align*}
    Again using \eqref{eq:privacy_good_tail_kernel_portmanteau_lemma_Aj_small_support} and the fact that $\delta \leq c_\alpha / m$ yield that the latter is also bounded by $\alpha/4m$ for $c_\alpha > 0$ small enough. The condition  and small enough choice of $c_\alpha > 0$ yields that the conditions of Lemma \ref{lem:privacy_testing_risk_approximation} and the conclusion of (c) follows. Finally, if $K^j$ satisfies \eqref{eq:uniform_bound_kernel_density}, the last statement follows directly by the construction of $\tilde{K}^j$.
    \end{proof}
    
    We finish the section by providing the two technical lemmas mentioned in the earlier proofs above. We omit the presence of the shared randomness $U$ in the statement of the lemmas, as it is of no consequence to the arguments below.
    
    \begin{lemma}\label{lem:privacy_markov_kernel_construction_lemma}
        Let $K$ be a Markov kernel from $(\cX,\mathscr{X})^n$ to $(\cY,\mathscr{Y})$ satisfying an $(\epsilon,\delta)$-DP constraint (i.e. \eqref{eq:privacy_constraint}) and define for a $A \in \mathscr{Y}$ and a probability measure $\mu$ on $\mathscr{Y}$ 
        \begin{equation*}
        \tilde{K}(B|x) := K(B \cap A | x) + K(A^c | x) \mu(B), \text{ for } x \in \cX, \, B \in \mathscr{Y}.
        \end{equation*}
        Then, $\tilde{K}$ is a Markov kernel $(\cX,\mathscr{X})$ to $(\cY,\mathscr{Y})$ satisfying an $(\epsilon,2\delta)$-DP constraint.
       \end{lemma}
       \begin{proof}
       First, $\tilde{K}$ can be seen to be a Markov kernel, as the necessary measurability assumptions hold by construction and 
       \begin{equation*}
       \tilde{K}(\cY|x) = K(\cY \cap A | x) + K(A^c | x) = 1,
       \end{equation*}
       where it is used that $\mu$ is a probability measure. Furthermore, for arbitrary $B$ and $x,x' \in \cX^n$ such that $d_H(x,x') \leq 1$, it holds that
       \begin{align*}
       \tilde{K}(B|x) &\leq e^\epsilon K(B \cap A | x') + \delta + e^\epsilon K(A^c | x') \mu(B) + \mu(B) \delta \\
       &\leq e^\epsilon \tilde{K}(B | x') + 2 \delta.
       \end{align*}
    \end{proof}
       
    \begin{lemma}\label{lem:privacy_testing_risk_approximation}
        Let $\alpha \in (0,1)$ be given. Let $( T, \{K^j \}_{j=1}^m)$ be a distributed testing protocol and suppose that there exist kernels $\{ \tilde{K}^j\}_{j=1}^m$ such that for $j=1,\dots,m$,
        \begin{equation*}
        \| P_0 (K^j(\cdot|X^{(j)},u) - \tilde{K}^j(\cdot|X^{(j)},u)) \|_{\mathrm{TV}} \leq \frac{\alpha}{2m} \;\; \P^U\text{-a.s}
        \end{equation*}
        and
        \begin{equation*}
        \| P_\pi (K^j(\cdot|X^{(j)},u) - \tilde{K}^j(\cdot|X^{(j)},u)) \|_{\mathrm{TV}} \leq \frac{\alpha}{2m}, \;\; \P^U\text{-a.s}
        \end{equation*}
        Then, 
        \begin{align*}
        \P^U P_0^{m}&K( T(Y) | X,U)  + \P^U \int P_{f}^{m} K ( 1 - T(Y)| X,U)) d\pi(f) \geq \\ &\P^U P_0^{m}\tilde{K}( T(Y) | X,U)  + \P^U \int P_{f}^{m} \tilde{K} ( 1 - T(Y)| X,U)) d\pi(f) - \alpha,
        \end{align*}
        for the same collection of distributions.
        \end{lemma}
        \begin{proof}
        We omit the dependence of $u$ in the proof, as it is of no consequence to the arguments below. Using standard arguments, 
        \begin{align*}
        P_0^{m}K(T(Y) = 1|X) + \int P_{f}^{m} K ( T(Y) = 0|X) d\pi(f) &\geq \\
        P_0^{m}\tilde{K}(T(Y) = 1|X) + \int P_{f}^{m} \tilde{K} ( T(Y) = 0|X) d\pi(f) - &\| P_0^{m} (K(\cdot|X) - \tilde{K}(\cdot|X)) \|_{\mathrm{TV}} \\ &- \| P_\pi^{m} (K(\cdot|X) - \tilde{K}(\cdot|X)) \|_{\mathrm{TV}}.   
        \end{align*}
        Standard arguments (see e.g. Lemma \ref{lem:TV_distance_product_measures} in this supplement),
        \begin{equation*}
        \| P_\pi^{m} (K(\cdot|X) - \tilde{K}(\cdot|X)) \|_{\mathrm{TV}} \leq \underset{j=1}{\overset{m}{\sum}}  \| P_\pi (K^j(\cdot|X^{(j)}) - \tilde{K}^j(\cdot|X^{(j)})) \|_{\mathrm{TV}}.
        \end{equation*}
        By applying the same lemma to $\| P_0^{m} (K(\cdot|X) - \tilde{K}(\cdot|X)) \|_{\mathrm{TV}}$, combined with what is assumed in this lemma, we obtain the result.
        \end{proof}

    \subsection{Step 3: Bounding the Chi-square divergence using the Brascamp-Lieb inequality}\label{ssec:supp:step3}
    
    We proceed with lower bounding \eqref{eq:sup:approx-bayes-risk-II}, where in a slight abuse of notation, we shall denote $\breve{K}^j$ by $K^j$ for the remainder of the section. 
    
    To lower bound the Bayes risk, in light of the Neyman-Pearson lemma, it should suffice to show that $\cL^{Y|U=u}_\pi(Y)$ is close to $1$ with high probability. This is made precise below by showing that the right-hand side of \eqref{eq:sup:approx-bayes-risk-II} is further bounded from below by
    \begin{equation}\label{eq : sup inf public coin divergence lb}
    1 -  \left( \sqrt{(1/2) \int \E_0^{Y|U=u}  \left(\cL^{Y|U=u}_\pi(Y) - 1 \right)^2 d\P^U(u) } - \alpha/2  \right).
    \end{equation}
    To see this, note that any $T : \cY^m \to \{0,1\}$, we can write $A_T = T^{-1}( \{ 0 \} )$ and note that
    \begin{align*}
        P_0^{m}\tilde{K}_u T(Y)  + P_\pi^{m}\tilde{K}_u ( 1 - T(Y) ) &= 1 - \left(P_0^{m}\tilde{K}_u \left(Y \in A_T \right) - P_\pi^{m}\tilde{K}_u \left(Y \in A_T \right)\right).
    \end{align*}
    We obtain that
    \begin{align*}\label{bound_lecam}
        \int \left( P_0^{m}\tilde{K}_u T  + \int  P_{f}^{m} \tilde{K}_u ( 1 - T) d\pi \circ \Psi^{-1}(f) \right) d\P^U(u)  & \\
    \geq 1 - \underset{A}{\sup} \, | \int P_0^{m}\tilde{K}_u (A) - P_\pi^{m}\tilde{K}_u (A) d\P^U(u) |.&
    \end{align*}
    We find by Jensen's inequality that
    \begin{equation*}
     \|\P_{0}^Y - \P_{\pi}^Y \|_{\mathrm{TV}} \leq \int \|\P_{0}^{Y|U=u} - \P_{\pi}^{Y|U=u} \|_{\mathrm{TV}} d\P^U(u).
    \end{equation*}
    Combining the above with Pinsker's second inequality a standard bound for the Kullback-Leibler divergence (see Lemma \ref{lem:chi_sq_div_bounds_KL}), we obtain \eqref{eq : sup inf public coin divergence lb}.

    This brings us to a crucial part of the proof; the application of Lemma 10.1 in \cite{szabo2023distributedtesting}, which we restate as Lemma \ref{lem:brascamp_lieb_consequence} below. We first introduce some notation.  
    
    We note that for $f \in \R^{d_L}$, it holds that 
    \begin{equation*}
        \frac{dP_{\Psi f}}{dP_0}(X^{(j)}_i) \overset{d}{=} \frac{dN\left({f}, \sigma^2 I_{d_L}\right)}{dN\left(0, \sigma^2 I_{d_L}\right)} =: \mathscr{L}_f^{ji}(X^{(j)}_i),
    \end{equation*}
    where the equality in distribution is true under $\P_0^{X}$.
    
    Denote the ``local'' and ``global'' likelihoods of the data as
    \begin{equation*}
    \mathscr{L}_f^j(X^{(j)})= {\prod}_{i=1}^n \mathscr{L}_f^{ji}(X^{(j)}_i), \; \; \; \mathscr{L}_f(X) := {\prod}_{j=1}^m \mathscr{L}_f^j(X^{(j)}),
    \end{equation*}
    and the mixture likelihoods as  
    \begin{equation*}
    \mathscr{L}_\pi^j(X) = \int \mathscr{L}_f(X^{(j)}) d \pi (f) \; \text{ and } \; \mathscr{L}_\pi(X) = \int \mathscr{L}_f(X) d \pi (f).
    \end{equation*}

    In view of the Markov chain structure, the probability measure $d\P_\pi(x,u,y)$ disintegrates as $d\P^{Y|(X,U)=(x,u)}d\P^X_f(x)d\P^U(u)d\pi(f)$. Using this, $\E_0^{Y|U=u}  \left(L^{Y|U=u}_\pi(Y)\right)^2$ can be seen to equal
    \begin{equation}\label{eq : conditional exp. form of chi-sq div}
      \E_0^{Y|U=u} \E_0\left[ \mathscr{L}_\pi(X) \bigg| Y, U=u \right]^2 = \int \left(  \int \mathscr{L}_\pi(x) \frac{dK(\cdot|x,u)}{d\P^{Y|U=u}_0} (y) d\P_0^X(x) \right)^2 d\P^{Y|U=u}_0(y),
    \end{equation}
    where it is used that $K(\cdot | x, u) \ll \P^{Y|U=u}_0(\cdot)$, $\P^{(X,U)}_f$-almost surely (Lemma \ref{lem:bracamp_densities_exist}). Using Fubini's theorem (``decoupling'' in $X$), we can write the above display as
    \begin{equation}\label{eq : numerator key ratio q notation}
    \int \mathscr{L}_\pi(x_1) \mathscr{L}_\pi(x_2) q_u(x_1,x_2) d(\P_0^X \times \P_0^X) (x_1,x_2),
    \end{equation}
    where 
    \begin{equation}\label{eq:forward_backward_kernel_def}
    q_u(x_1,x_2) := \int \frac{d K(\cdot|x_1,u) }{ d\P_0^{Y|U=u} } (y) \frac{d K(\cdot|x_2,u) }{ d\P_0^{Y|U=u}} (y) d\P^{Y|U=u}_0(y).
    \end{equation}
    Since $K(\cdot|x,u)$ and $\P^{Y|U=u}_0$ are product measures on $\cY =  \cY^m$, we can write $q_u(x_1,x_2) = \Pi_{j=1}^m q_u^j(x_1,x_2)$ where 
    \begin{equation}\label{eq:marginal_forward_backward_kernel_def}
    q_u^j(x_1,x_2) =  \int   \frac{K^j(y^{j}|x_1^{j},u) K^j(y^{j}|x_2^{j},u) }{ \P_0^{Y^{j}|U=u}(y^{j}) } d\P^{Y^{(j)}|U=u}_0(y).
    \end{equation}
    The map $(x_1,x_2)\mapsto q_u(x_1,x_2)$ can be seen as capturing the dependence between the original data $X$ and a random variable $X'$ with conditional distribution
    \begin{equation}\label{eq:backward_forward_channel_Xprime}
     X'|X=x \sim \int d\P^{X|(Y,U)=(y,u)}_0 d\P^{Y|(X,U)=(x,u)},
    \end{equation}
    which is sometimes referred to as the ``forward-backward channel'', stemming from the fact that $X \to Y \to X'$ forms a Markov chain. An easy computation using the law of total expectation shows that the covariance of $q_u(x_1,x_2) d(P_0 \times P_0) (x_1,x_2)$,
    \begin{equation}\label{eq : q cov condition}
    \int \left( \begin{matrix}
    x_1 \\
    x_2
    \end{matrix}\right)
     \left( \begin{matrix}
    x_1^\top & x_2^\top
    \end{matrix}\right)
     q_u(x_1,x_2) d(P_0 \times P_0) (x_1,x_2)  \in \R^{2mnd_L \times 2mnd_L},
    \end{equation}
      is equal to $\Sigma_u := \text{Diag}\left(\Sigma^{11}_u,\dots,\Sigma^{1n}_u,\dots,\Sigma^{m1}_u,\dots,\Sigma^{mn}_u\right) \in \R^{2mnd_L \times 2mnd_L}$ for
    \begin{equation*}
    \Sigma^{ji} := \sigma^{2} \left(\begin{matrix}
         I_{d_L} & \Xi_u^{ji} \\
    \Xi_u^{ji} &   I_{d_L}
    \end{matrix}\right),
    \end{equation*}
    with
    \begin{equation*}
    \Xi_u^{ji} := \E_0^{Y^{(j)}|U=u} \E_0\left[  \sigma^{-1} {X}^{(j)}_{L;i} \bigg| Y^{(j)}, U=u \right] \E_0\left[ \sigma^{-1} {X}^{(j)}_{L;i} \bigg| Y^{(j)}, U=u \right]^\top.
    \end{equation*} 
    Define also
    \begin{equation}\label{eq:def_Xi_u_j}
    \Xi_u^{j} := \E_0^{Y^{(j)}|U=u} \E_0\left[ \sigma^{-1} \underset{i=1}{\overset{n}{\sum}}  {X}^{(j)}_{L;i} \bigg| Y^{(j)}, U=u \right] \E_0\left[ \sigma^{-1} \underset{i=1}{\overset{n}{\sum}}  {X}^{(j)}_{L;i} \bigg| Y^{(j)}, U=u \right]^\top.
    \end{equation} 
    We are now ready to state the lemma that forms the crux of our distributed testing lower bound proof.
    
    \begin{lemma}\label{lem:brascamp_lieb_consequence}
    Suppose that $(x_1,x_2) \mapsto q_u(x_1,x_2)$ is bounded and that $\pi$ is a centered Gaussian distribution on $\R^d$. Then,
        \begin{equation}\label{eq : gaussian upper bound ratio lb prf}
    \frac{\int \mathscr{L}_\pi(x_1) \mathscr{L}_\pi(x_2) q_u(x_1,x_2) d(\P_0^X \times \P_0^X) (x_1,x_2)}{\underset{j=1}{\overset{m}{\Pi}} \int \mathscr{L}^{j}_\pi(x_1^{j}) \mathscr{L}^{j}_\pi(x_2^{j}) q_u^j(x_1^{j},x_2^{j}) d(\P_0^{X^{(j)}} \times \P_0^{X^{(j)}}) (x_1^{j},x_2^{j})}
    \end{equation}
    is bounded above by
    \begin{equation*}
     \frac{\int \mathscr{L}_\pi(x_1) \mathscr{L}_\pi(x_2) dN(0,\Sigma)(x_1,x_2)}{\underset{j=1}{\overset{m}{\Pi}} \int \mathscr{L}^{j}_\pi(x_1^{j}) \mathscr{L}^{j}_\pi(x_2^{j})  dN(0,\Sigma^j) (x_1^{j},x_2^{j})}.
    \end{equation*}
    \end{lemma}
    
    The lemma has the following interpretation: the ratio of the second moment of the Bayes factor of the ``global Bayesian hypothesis test'' that of the product of second moments of the ``local Bayes factors'', is maximized over the class of forward-backward channel with covariance $\Sigma$ when the forward-backward channel is Gaussian. 
    
    There is an existing literature on Brascamp-Lieb inequality in relation to information theoretical problems, in relation to mutual information \cite{carlen_subadditivity_2008,liu_brascamp-lieb_2016,liu_smoothing_2016}. For a proof, we refer to Lemma 10.1 in \cite{szabo2023distributedtesting}, which uses that the prior $\pi$ is Gaussian, exploiting the conjugacy between the prior and the model which enables the use of techniques from \cite{lieb_gaussian_1990}. 
    
    The main consequence of the above lemma is analytic expressions, which can be used to upper bound the chi-square divergence in \eqref{eq : sup inf public coin divergence lb} which is the content of Lemma \ref{lem:testing_problem_AB_breakdown} below. 
    
    \begin{lemma}\label{lem:testing_problem_AB_breakdown}
    Define
    \begin{equation}\label{eq:Au_defined}
    \mathrm{A}_u^\pi := \int e^{ f^\top \sum_{j=1}^{m} \Xi^j_u g} d(\pi \times  \pi)(f,g)
    \end{equation}
    and
    \begin{equation}\label{eq:sup:Bu_defined}
    \mathrm{B}_u^\pi := \underset{j=1}{\overset{m}{\Pi}} \E_0^{{Y^{(j)}}|U=u} \E_0 \left[ \mathscr{L}_\pi \left({X}^{{(j)}}\right) \bigg| Y^{{(j)}}, U=u \right]^2.
    \end{equation}
    If $(x_1,x_2) \mapsto q_u(x_1,x_2)$ is bounded and if $\pi$ is a centered Gaussian distribution on $\R^{d_L}$, it holds that
    \begin{equation*}
     \E_0^{Y|U=u}  \left(L^{Y|U=u}_\pi(Y)\right)^2 \leq  \mathrm{A}_u^\pi \cdot \mathrm{B}_u^\pi.
    \end{equation*}
    \end{lemma}
    
    The above lemma describes how the variance of the Bayes factor given $U$ is bounded by two factors. One factor depends on the Fisher information of the transcript's likelihood at $f=0$ given $U=u$; $\Xi_u := \sum_{j=1}^{m} \Xi^j_u$. In this sense, $\mathrm{A}_u^\pi$ captures how well the transcript allows for ``estimation'' of $f$. The second factor can be seen as the $m$-fold product of the local Bayes factors, capturing essentially the power of combining the locally most powerful test statistics; the likelihood ratios.
    \begin{proof}[Proof of Lemma \ref{lem:testing_problem_AB_breakdown}]
    We start by noting that $\mathrm{B}_u^\pi$ is equal to the denominator of \eqref{eq : gaussian upper bound ratio lb prf}. By Lemma \ref{lem:brascamp_lieb_consequence},
    \begin{equation*}
     \E_0^{Y|U=u}  \left(L^{Y|U=u}_\pi(Y)\right)^2 \leq \frac{\int \mathscr{L}_\pi(x_1) \mathscr{L}_\pi(x_2) dN(0,\Sigma)(x_1,x_2)}{\underset{j=1}{\overset{m}{\Pi}} \int \mathscr{L}^{j}_\pi(x_1^{j}) \mathscr{L}^{j}_\pi(x_2^{j})  dN(0,\Sigma^j) (x_1^{j},x_2^{j})} \cdot \mathrm{B}_u^\pi.
    \end{equation*}
    By the block diagonal matrix structure of $\Sigma$, the denominator in the first factor of the right-hand side satisfies
    \begin{align*}
    \underset{j=1}{\overset{m}{\Pi}} \int \mathscr{L}^{j}_\pi(x_1^{j}) \mathscr{L}^{j}_\pi(x_2^{j}) dN(0,\Sigma^j) (x_1^j,x_2^j)
    &= \underset{j=1}{\overset{m}{\Pi}}  \int e^{  \frac{1}{2} \left(\|\sqrt{\Sigma^j} \left({f},  {g}\right) \|_2^2-  \|\left({f},  {g}\right)\|_2^2 \right)} d (\pi \times\pi)\left({f}, {g}\right)\\
    &= \underset{j=1}{\overset{m}{\Pi}} \int e^{  f^\top  \Xi^j_u g} d(\pi \times  \pi)(f,g)\\
    &\geq \underset{j=1}{\overset{m}{\Pi}}e^{  \int f^\top  \Xi^j_u g  \,d(\pi \times  \pi)(f,g)}=1.
    \end{align*}
    Through the expression for the moment generating function of the Gaussian, the numerator of $\mathrm{A}_u^\pi$ is equal to
    \begin{equation*}
    \int \mathscr{L}_\pi(x_1) \mathscr{L}_\pi(x_2) dN(0,\Sigma) (x_1,x_2)=\int e^{ f^\top \sum_{j=1}^{m} \Xi^j_u g} d(\pi \times  \pi)(f,g).
    \end{equation*}
    \end{proof}
    
    \subsection{Step 4: Adversarially choosing the prior based on shared or local randomness}
    
    Suppose that for some constant $c > 0$,
    \begin{equation}\label{eq : to check before 6.2.2}
     { \varrho^2} \| \sqrt{\bar{\Gamma}}^\top \Xi_u \sqrt{\bar{\Gamma}} \| \leq c.
    \end{equation}
    If $\bar{\Gamma} \in \R^{d_L \times d_L}$ is symmetric, idempotent with rank proportional to $d_L$ and $\pi = N(0,\varrho^2 \bar{\Gamma})$, standard results for the Gaussian chaos, e.g. Lemma 6.2.2 in \cite{vershynin_high-dimensional_2018} combined with \eqref{eq : to check before 6.2.2} and the fact that $\|\sqrt{\bar{\Gamma}}\| \leq 1$, yield that
    \begin{equation*}
    \mathrm{A}_u^\pi \leq \exp \left( { C  \sigma^{-2} { \varrho^4}  \text{Tr} \left((\sqrt{\bar{\Gamma}}^\top \Xi_u \sqrt{\bar{\Gamma}})^2 \right) } \right),
    \end{equation*}
    for a constant $C>0$ depending only on $c$. As a final step of the testing risk lower bound technique, we use essentially a geometric argument to sharpen this bound in case the distributed protocol does not enjoy shared randomness. The $d_L \times d_L$ matrix $\Xi_u := \sum_{j=1}^m \Xi_u^j$ geometrically captures how well $Y$ allows to ``reconstruct'' the compressed sample $X$. When $U$ is degenerate, $\Xi_u$ is ``known'' to the prior, and $\bar{\Gamma}$ can be chosen to exploit ``direction'' in which $\Xi_u$ contains the least information. We finish this section by proving the lemma below, which makes this notion precise.
    
    \begin{lemma}\label{lem:brascamp_lower_bound_summarization}
    Let $\alpha \in (0,1)$ and suppose that the map $(x_1,x_2) \mapsto q_u(x_1,x_2)$ defined in \eqref{eq:forward_backward_kernel_def} is bounded for all distributed testing protocols in $\mathscr{T}$. Let $\pi = N(0,\varrho^2 \bar{\Gamma})$, with $\varrho :=  \frac{\rho}{c_\alpha^{1/4} d^{1/2}_L}$ and $\bar{\Gamma} \in \R^{d_L \times d_L}$ is symmetric, idempotent with rank proportional to $d_L$. Assume that $\rho$ is such that $\varrho^2 \| \Xi_u \| \leq c$ $\P^U$-a.s. for some constant $c>0$. It then holds that
    \begin{equation}\label{eq:summarization_lemma_public_coin}
    \mathrm{A}_u^\pi \leq \exp\left({ C \sigma^{-2} \frac{\rho^4}{c_\alpha d^2_L} \|\Xi_u\| \text{Tr} \left( \Xi_u \right) }\right),
    \end{equation}
    for some fixed constant $C>0$ depending only on $c>0$. Furthermore, if $U$ is degenerate and $ \frac{2\rho^2}{\sqrt{c_\alpha} d^2} \text{Tr} ( \Xi ) \leq c$, it holds that 
    \begin{equation}\label{eq:summarization_lemma_private_coin}
    \mathrm{A}_u^\pi \leq \exp\left({ C  \sigma^{-2} \frac{\rho^4}{c_\alpha d^3_L}  \text{Tr} \left( \Xi_u  \right)^2 }\right),
    \end{equation}
    for some fixed constant $C>0$ depending only on $c>0$. 
    \end{lemma}
    \begin{proof}
    In case of shared randomness (i.e. $U$ not being degenerate), simply taking $\bar{\Gamma} = I_{d_L}$, noting that $\text{Tr}( \Xi^2_u) = \|\Xi_u\| \, \text{Tr} (\Xi_u)$ and following the argument earlier in the section yields \eqref{eq:summarization_lemma_public_coin}. 
    
    Now assume $U$ is degenerate and write $\Xi_u = \Xi$. The matrix $\Xi$ is positive definite and symmetric, therefore it possesses a spectral decomposition $V^\top \text{Diag}(\xi_1,\dots,\xi_{d_L}) V$. Without loss of generality, assume that $\xi_1 \geq \xi_2 \geq \cdots \geq \xi_{d_L}$ with corresponding eigenvectors $V = \left(\begin{matrix} v_1 & \cdots & v_{d_L} \end{matrix}\right)$. Let $\check{V}$ denote the ${d_L} \times \lceil {d_L}/2 \rceil$ matrix $\left(\begin{matrix} v_{\lfloor {d_L}/2 \rfloor+1} & \cdots & v_{d_L} \end{matrix}\right)$. The choice of prior may depend on $\Xi$, to see this, note the order of the supremum and infimum in \eqref{eq : sup inf public coin divergence lb} and the fact that $\Xi$ solely depends on the choice of kernel. To that extent, set $\bar{\Gamma} = \check{V} \check{V}^\top$. It holds that
    \begin{align*}
        \text{Tr}(\check{V} \check{V}^\top) &= \underset{i=1}{\overset{d_L}{\sum}}  \underset{k=\lfloor {d_L}/2 \rfloor+1}{\overset{d_L}{\sum}} (v_k)_i^2 = \lceil {d_L}/2 \rceil.
    \end{align*}
    The choice $\bar{\Gamma}$ is thus seen to satisfy the conditions of symmetry and positive definiteness and is idempotent with rank $\lceil {d_L}/2 \rceil$. 
    
    Since the eigenvalues are decreasingly ordered,
    \begin{equation*}
    \xi_{\lfloor {d_L}/2 \rfloor} \leq \frac{2}{{d_L}} \underset{i=1}{\overset{\lfloor {d_L}/2 \rfloor}{\sum}} \xi_i \leq \frac{2}{{d_L}} \text{Tr} ( \Xi ).
    \end{equation*}
    By orthogonality of the columns of $V$, $\check{V}^\top \Xi \check{V} = \text{Diag}(\xi_{\lfloor d/2 \rfloor+1},\dots,\xi_{d_L})$. The condition of \eqref{eq : to check before 6.2.2} reduces to
    \begin{align*}
    \varrho^2 \|\sqrt{\bar{\Gamma}}^\top \Xi_u \sqrt{\bar{\Gamma}} \| &\leq  \varrho^2 \xi_{\lfloor {d_L}/2 \rfloor} \leq 2 \frac{\rho^2}{\sqrt{c_\alpha} {d_L}^2} \text{Tr} ( \Xi ).
    \end{align*}
    
    Furthermore,
    \begin{align*}
    \text{Tr} \big( (\sqrt{\bar{\Gamma}}^\top \Xi_u \sqrt{\bar{\Gamma}} )^2 \big) =\text{Tr} \left( (\check{V}^\top \Xi \check{V})^2 \right) = \underset{i=\lfloor {d_L}/2 \rfloor+1}{\overset{{d_L}}{\sum}} \xi_i^2 \leq {d_L} \xi_{\lfloor {d_L}/2 \rfloor}^2 \leq \frac{4}{{d_L}} \text{Tr} ( \Xi )^2,
    \end{align*}
    which implies in turn that
    \begin{equation*}
    \varrho^4 \text{Tr} \left( (\check{V}^\top \Xi \check{V})^2 \right) \leq 4 \frac{\rho^4}{c_\alpha {d_L}^3} \text{Tr} ( \Xi )^2.
     \end{equation*}
    \end{proof}

    \subsection{Step 5: Capturing the cost of privacy in trace of $\Xi_u$ and the local Bayes factors}\label{ssec:step5}
    
    The cost of privacy is captured through bounds on $\mathrm{A}_u^\pi$ and $\mathrm{B}_u^\pi$. These bounds specifically use the fact that the Markov kernels that underlie these quantities are $(\epsilon,6\delta)$-differentially private. 
    
    We start with the bound on $\mathrm{A}_u^\pi$, for which we proceed by a data processing argument for the matrix $\Xi_u$ under the $(\epsilon,6\delta)$-DP constraint. This comes in the guise of Lemma \ref{lem:data_processing_impure_DP_Fisher_info} below. Its proof is deferred to the end of the section.
    
    \begin{lemma}\label{lem:data_processing_impure_DP_Fisher_info}
    Let $0 < \epsilon \leq 1$ and let $Y^{(j)}$ be a transcript generated by an $(\epsilon,\delta)$-DP constraint distributed protocol, with $0 < \epsilon \leq 1$ and $0 \leq \delta \leq  \left( \left({n}{{d_L}}^{-1} \wedge n^{1/2}{d_L}^{-1/2}  \right) \epsilon^2 \right)^{1+\omega}$ for some $\omega>0$.  The matrix $\Xi_u^j$ as defined in \eqref{eq:def_Xi_u_j} satisfies
    \begin{equation*}
    \text{Tr}\left( \Xi_u^j \right) \leq (C n^2 \epsilon^2) \wedge (nd_L)
    \end{equation*}
    for a fixed constant $C > 0$. Furthermore, it holds that $\Xi_u^j \leq  n I_{d_L}$.
    \end{lemma}
    
    The lemma implies in particular that $\| \Xi_u^j \| \leq (C n^2 \epsilon^2) \wedge n$, as $\Xi_u^j$ is symmetric and positive definite. Combining this with \eqref{eq:nonasymptotic_testing_lower_bound_shared_randomness_dL} and the triangle inequality, we obtain
    \begin{equation}\label{eq : to check before 6.2.2 public coin}
     { \varrho^2} \|  \Xi_u  \| \leq { \varrho^2} \underset{j=1}{\overset{m}{\sum}}  \|  \Xi_u^j  \| \leq \frac{m\left((C n^2 \epsilon^2) \wedge n\right) \rho^2}{\sigma^{2} \sqrt{c_\alpha} d_L} \leq C\sqrt{c_\alpha}.
    \end{equation}
    Similarly, \eqref{eq:nonasymptotic_testing_lower_bound_local_randomness_dL} yields
    \begin{equation}\label{eq : to check before 6.2.2 private coin}
        \frac{2 \rho^2}{\sqrt{c_\alpha} d^2_L} \text{Tr}\left( \Xi_u^j \right) \leq \frac{m\left((C n^2 \epsilon^2) \wedge (n d_L) \right) \rho^2}{\sigma^{2}\sqrt{c_\alpha} d^2_L} \leq C\sqrt{c_\alpha} / \sqrt{d_L}.
    \end{equation}
    The last two displays together finish the verification of the conditions of Lemma \ref{lem:brascamp_lower_bound_summarization}. The above data processing inequalities for $\Xi_u^j$ and bounds on $\rho^2$ also yield a bound on $\mathrm{A}_u^\pi$ as defined in Lemma \ref{lem:brascamp_lower_bound_summarization}. In case of shared randomness protocols, using \eqref{eq:summarization_lemma_public_coin}, \eqref{eq : to check before 6.2.2 public coin}, Lemma \ref{lem:data_processing_impure_DP_Fisher_info} and \eqref{eq:nonasymptotic_testing_lower_bound_shared_randomness_dL}, we obtain
    \begin{equation*}
    \mathrm{A}_u^\pi \leq \exp \left( C^2 c_\alpha \right).
    \end{equation*}
    In case of local randomness protocols, combining \eqref{eq:summarization_lemma_private_coin} with \eqref{eq : to check before 6.2.2 private coin} and \eqref{eq:nonasymptotic_testing_lower_bound_local_randomness_dL} yields the above bound on $\mathrm{A}_u^\pi$.
    
    Next, we turn to $\mathrm{B}_u^\pi$. The first bound, given in Lemma \ref{lem:Bu_bound_bandwidth_constraints}, does not use privacy at all. This bound is only tight whenever $\epsilon \gtrsim 1/\sqrt{n}$ (as a bound on $\mathrm{B}_u^\pi$), in which case it corresponds to the regime where majority voting bares no privacy cost.
    
    \begin{lemma}\label{lem:Bu_bound_bandwidth_constraints}
    Consider $\mathrm{B}_u^\pi$ as in \eqref{eq:sup:Bu_defined} with $\pi = N(0,\rho^2 c_\alpha^{-1/2} d^{-1}_L \bar{\Gamma})$ with $\bar{\Gamma}$ idempotent. It holds that
    \begin{equation*}
    \mathrm{B}_u^\pi \leq \exp \left(C  \frac{mn^2 \rho^{4}}{c_\alpha \sigma^2 d_L}  \right).
    \end{equation*}
    \end{lemma}
    \begin{proof}
    Since conditional expectation contracts the $L_2(\P_0^Y)$-norm, 
    \begin{equation*}
    \underset{j=1}{\overset{m}{\Pi}} \E_0^{{Y^{(j)}}|U=u} \E_0 \left[ \mathscr{L}_\pi \left({X}^{{(j)}}\right) \bigg| Y^{{(j)}}, U=u \right]^2 \leq \underset{j=1}{\overset{m}{\Pi}} \E_0^{X^{(j)}} \left[ \mathscr{L}_\pi \left({X}^{(j)}\right)^2 \right].
    \end{equation*}
    We now proceed to bound the first factor {in the product} on the right-hand side of the display above, which for a positive semi-definite choice of $\bar{\Gamma}$ equals
    \begin{equation*}
     \int \E_0^{X^{(j)}}  \exp \left( (\sqrt{\bar{\Gamma}}(f+g))^\top \sum_{i=1}^n  {X}^{{(j)}}_{L;i} - \frac{n \sigma^2}{2} \left\|\sqrt{\bar{\Gamma}}f\right\|_2^2 - \frac{n\sigma^2}{2} \left\|\sqrt{\bar{\Gamma}}g\right\|_2^2 \right) dN\left(0, \frac{\rho^2}{\sqrt{c_\alpha}d_L}  I_{2d_L}\right) (f,g).
    \end{equation*}
    By direct computation involving the moment generating function of the normal distribution, the latter display equals
    \begin{equation*}
     \int \exp \left( \frac{n \rho^2}{\sigma^2 \sqrt{c_\alpha}d_L} z^\top \bar{\Gamma} z'  \right) dN(0,I_{2d_L}) (z,z'). 
    \end{equation*}
    We aim to apply the moment generating function of the Gaussian chaos, e.g. Lemma 6.2.2 in \cite{vershynin_high-dimensional_2018} to the above display. Using that $\rho^2$ satisfies \eqref{eq:nonasymptotic_testing_lower_bound_local_randomness_dL} or \eqref{eq:nonasymptotic_testing_lower_bound_shared_randomness_dL} and since $\|\bar{\Gamma}\| = 1$ by the fact that $\bar{\Gamma}$ is idempotent, $ \frac{ n \rho^2}{\sigma^2 c_\alpha^{1/2} {d_L}} \leq \sqrt{c_\alpha}$ with $c_\alpha >0$ chosen small enough, the aforementioned result yields that there exists a constant $C>0$
    \begin{equation}\label{eq : factorized final bound}
     \underset{j=1}{\overset{m}{\Pi}} \E_0^{X^{(j)}|U=u} \left[ \mathscr{L}_\pi \left({X}^{(j)}\right)^2 \right] \leq \exp \left(C c_\alpha^{-1} \frac{mn^2 \rho^{4}}{\sigma^4 d_L}  \right),
    \end{equation}
    where $C > 0$ is universal.
    \end{proof}
    
    Whenever $\epsilon \leq n^{-1/2}$, a much more involved data processing argument is needed than the one above. In the argument that follows, we will use the bound on $\cL_{\pi,u}^j$ obtained through Lemma \ref{lem:portmanteau_lemma} in Step 2. The data processing argument leads to the bound of Lemma \ref{lem:product_of_likelihoods_privacy_bound_impure_DP} below. Its proof is based on coupling arguments, where the two different couplings constructed result in the different rates observed in the condition of the theorem.
    
    \begin{lemma}\label{lem:product_of_likelihoods_privacy_bound_impure_DP}
    Let $\pi = N(0,d^{-1}_L \rho^2 \bar{\Gamma})$, with $\bar{\Gamma} \in \R^{d_L \times d_L}$ a symmetric idempotent matrix, 
    \begin{equation*}
     \rho^2 \leq \sigma^{2} c_\alpha ({d^{1/2}_L}/({\sqrt{m}n^{\frac{3}{2}} \epsilon}) \vee {1}/{(mn^2\epsilon^2)})
    \end{equation*}  
    and $\{K^j\}_{j=1}^m$ correspond to a $(\epsilon,\delta)$-DP distributed protocol with transcripts $Y^{(j)}$ such that $0 < \epsilon \leq 1$, $\delta \lesssim c_\alpha (m^{-1} \wedge \epsilon)$ and
    \begin{equation*}
     | {\cL}_{\pi,u}^j(y) - 1| \leq \frac{5 m^{1/2}}{{\alpha}}  \;\;\;P_0 K^j(\cdot|X^{(j)},u)\text{-a.s.}
    \end{equation*}
    Then, there exists a universal constant $C> 0$ such that 
    \begin{equation}\label{eq:Bu_bound_exp_times_const}
    \mathrm{B}^\pi_u \leq e^{C \sqrt{c_\alpha}}.
    \end{equation}
    \end{lemma}
    
    Combining the lemma above with the bound $\mathrm{B}^\pi_u \leq \exp(C \frac{mn^2\rho^4}{\sigma^2 d_L})$ (which follows from Lemma \ref{lem:Bu_bound_bandwidth_constraints}), we obtain that \eqref{eq:Bu_bound_exp_times_const} holds whenever $\rho$ satisfies \eqref{eq:nonasymptotic_testing_lower_bound_local_randomness_dL} or \eqref{eq:nonasymptotic_testing_lower_bound_shared_randomness_dL}. Combining this with the bounds on $\mathrm{A}_u^\pi$ derived earlier and considering \eqref{eq : sup inf public coin divergence lb} lower bounds the testing risk, we obtain that 
    \begin{equation*} 
        \cR( H_{\rho}^{{s},R} , T) > 1 -\alpha
    \end{equation*} 
    for $c_\alpha > 0$ small enough, from which the result of Theorem \ref{thm:nonasymptotic_testing_lower_bound} follows. To complete the proof, we provide the proofs of Lemmas \ref{lem:data_processing_impure_DP_Fisher_info} and \ref{lem:product_of_likelihoods_privacy_bound_impure_DP} in the following subsection.
    
    \subsubsection{Proof of Lemma \ref{lem:product_of_likelihoods_privacy_bound_impure_DP}}\label{ssec:grand_coupling_lemma_proof}
    
    Before providing the proof of Lemma \ref{lem:product_of_likelihoods_privacy_bound_impure_DP}, we first develop additional tools. 
    
    First off is the following general coupling lemma that will be used in the proof of Lemma \ref{lem:product_of_likelihoods_privacy_bound_impure_DP}. The lemma is in essence Lemma 6.1 in \cite{karwa2017finite}, but its proof might be easier to verify and is provided for completeness. 
    
    \begin{lemma}\label{lem:coupling_general}
    Consider random variables $X_1,\ldots,X_n \overset{\text{i.i.d.}}{\sim} P_1$ and $\tilde{X}_1,\ldots,\tilde{X}_n \overset{\text{i.i.d.}}{\sim}  P_2$ defined on the same space. Write $X=(X_1,\ldots,X_n)$, $\tilde{X}=(\tilde{X}_1,\ldots,\tilde{X}_n)$ and let $K$ be a Markov kernel between the sample space of $X$ (equivalently $\tilde{X}$) and an arbitrary target space, satisfying an $(\epsilon,\delta)$-DP constraint (i.e. \eqref{eq:privacy_constraint}) with $\epsilon \leq 1$.
    
    Suppose that there exists a coupling $\P$ of $(\tilde{X},{X})$ such that $\P^{\tilde{X}} = P_1^n$, $\P^{X} = P_2^n$ and
    \begin{equation*}
    D_i := \mathbbm{1}\left\{\tilde{X}_i \neq {X}_i \right\} \sim \text{Ber}\left(p\right), \text{ i.i.d. for } i=1,\dots,n, \; p  \in [0,1]
    \end{equation*}
    under $\P$. 
    
    Then, it holds that
    \begin{equation}\label{eq:forward-reverse-equation}
    P_1^n K\left(A |\tilde{X}\right) \leq e^{4 \epsilon n p_{}} P_2^n K\left(A |{X}\right) + 2 \delta e^{ 4n p \epsilon}.
    \end{equation}
    \end{lemma}
    \begin{proof}
    Let $\E$ denote expectation with respect to $\P_{}$ and write $D = (D_i)_{i \in [n]}$, $S := \sum_{i=1}^n D_i$. We start by noting that
    \begin{equation}\label{eq:ch1_if_Dis0_tildeX_is_X}
    \E \left[ K(A | \tilde{X}) | S = 0 \right] = \E \left[ K(A | {X}) | S = 0 \right].
    \end{equation}
    Next, we show that for all $ k \in [n]$,
    \begin{align}
    e^{-\epsilon}\E \left[ K(A | \tilde{X}) | S = k - 1 \right] - \delta \leq
    \E \left[ K(A | \tilde{X}) | S = k \right] 
     \leq e^{\epsilon}\E \left[ K(A | \tilde{X}) | S = k - 1 \right] + \delta. \label{eq:iterative_equation_conditional_D_privacy}
    \end{align}
    Write $v_{(-i)} = (v_i)_{[n]\setminus \{i\}}$ for a vector $v \in \R^n$. Let $k \in [n]$ be given and let $\cV_k$ denote the set of $v \in \{0,1\}^n$'s such that $\sum_{i=1}^n v_i = k$. Using the definition of DP, the integrand in the conditional expectation satisfies
    \begin{equation}\label{eq:privacy_lb_ub}
    e^{-\epsilon} K(A | \tilde{X}_1,\dots,\breve{X}_{i}, \dots,\tilde{X}_n) - \delta \leq K(A | \tilde{X}) \leq e^\epsilon K(A | \tilde{X}_1,\dots,\breve{X}_{i}, \dots,\tilde{X}_n) + \delta,
    \end{equation}
    for any random variable $\breve{X}_{i}$ taking values in the sample space of $\tilde{X}_i$. In particular, if $v_i = 1$ it holds that
    \begin{align*}
    \E^{ } \left[   K(A | x_1,\dots,{X}_{i}, \dots,x_n) \big| D_{i} = v_i, \tilde{X}_{(-i)}=x_{(-i)}, D_{(-i)}=v_{(-i)} \right] \leq& \\e^\epsilon  \E^{} \left[  K(A | x_1,\dots,{X}_{i}, \dots,x_n) \big|  D_{i} = 0, \tilde{X}_{(-i)}=x_{(-i)}, D_{(-i)}=v_{(-i)} \right] + \delta,&
    \end{align*}
    for all $x$ in the sample space of $\tilde{X}$. It follows by the law of total probability that
    \begin{align*}
    \E \left[ K(A | \tilde{X}) | D = v \right] \leq e^{\epsilon}\E \left[ K(A | \tilde{X}) | D_i = 0,\, D_k = v_k \text{ for } k \in [n]\setminus\{i\} \right] + \delta,
    \end{align*}
    for all $i \in [n]$. The event $\left\{ D = v \right\}$ equals the event $\left\{ D = v, \, S_{} = k  \right\}$, and similarly it holds that
    \begin{equation*}
    \left\{ D_k = v_k \text{ for } k \in [n]\setminus\{i\}, \, D_i = 0 \right\} = \left\{ D_k = v_k \text{ for } k \in [n]\setminus\{i\}, \, D_i = 0, \, S = k-1 \right\}.
    \end{equation*}
    Consider now the sets
    \begin{align*}
    \cV_{k-1}(v) &:= \left\{ v' \in \cV_{k-1} \, : \,  v_k = v_k' \text{ except for one } i \in [n] \right\} \text{ for } v \in \cV_k,\\
    \cV_{k}(v') &:= \left\{ v \in \cV_{k} \, : \,  v_k = v_k' \text{ except for one } i \in [n] \right\} \text{ for } v' \in \cV_{k-1}.
    \end{align*}
    By what we have derived so far, it holds that any $v \in \cV_k$ and $v' \in \cV_{k-1}(v)$,
    \begin{align*}
    \E \left[ K(A | \tilde{X}) | D = v ,\, S = k \right] \leq e^{\epsilon}\E \left[ K(A | \tilde{X}) | D = v', \, S = k - 1 \right] + \delta.
    \end{align*}
    Consider $\{I_k(v) : v \in \cV_k\}$ independent random variables (on a possibly enlarged probability space) taking values in $[n]$ such that $\P( I_k(v) = i ) = 1/k$ whenever $v_i = 1$. Combining the above with the total law of probability we find that
    \begin{align*}
    \E \left[ K(A | \tilde{X}) | S = k \right] &=\\
    \frac{1}{{n\choose k}} \underset{v \in \cV_k}{\overset{}{\sum}}  \E \left[ K(A | \tilde{X}) | D = v,\, S = k \right] &\leq \\
    e^{\epsilon} \frac{1}{{n\choose k}} \underset{v \in \cV_k}{\overset{}{\sum}}  \E \left[ K(A | \tilde{X}) | D_{I(v)} = 0,\, D_{-I(v)} = v_{-I(v)},\, S = k - 1 \right] + \delta &= \\
    e^{\epsilon} \frac{1}{{n\choose k}} \frac{1}{k} \underset{v \in \cV_k}{\overset{}{\sum}} \underset{v' \in \cV_{k-1}(v)}{\overset{}{\sum}}   \E \left[ K(A | \tilde{X}) | D = v' ,\, S = k - 1 \right] + \delta &= \\
    e^{\epsilon} \frac{1}{{n\choose k}} \frac{1}{k}  \underset{v' \in \cV_{k-1}}{\overset{}{\sum}} \underset{v \in \cV_k(v')}{\overset{}{\sum}}   \E \left[ K(A | \tilde{X}) | D = v' ,\, S = k - 1 \right] + \delta &= \\
    e^{\epsilon} \frac{1}{{n\choose k - 1}} \underset{v' \in \cV_{k - 1}}{\overset{}{\sum}} \E \left[ K(A | \tilde{X}) | D = v', \, S = k - 1 \right] + \delta&=\\
    e^{\epsilon}\E \left[ K(A | \tilde{X}) | S = k - 1 \right] + \delta,
    \end{align*}
    where it is used that $|\cV_k| = {n \choose k}$,
    \begin{equation*}
    \P( D_1 = v_1,\dots,D_n=v_n | S = k) = \P( D_1 = \tilde{v}_1,\dots,D_n=\tilde{v}_n | S = k) 
    \end{equation*}
    for all $v=(v_i)_{i \in [n]},\tilde{v}=(\tilde{v}_i)_{i \in [n]} \in \cV_k$ and for any $v' \in \cV_{k-1}$ there are $n- k + 1$ ways to obtain $v \in \cV_k$ such that $v_k = v_k'$ except for one $i \in [n]$.
    
    By applying the privacy lower bound of \eqref{eq:privacy_lb_ub} and repeating the same steps, we also find that
    \begin{equation*}
    e^{-\epsilon}\E \left[ K(A | \tilde{X}) | S = k - 1 \right] - \delta \leq
    \E \left[ K(A | \tilde{X}) |  S = k \right].
    \end{equation*}
    This proves \eqref{eq:iterative_equation_conditional_D_privacy}, which, applying iteratively, results in the bound 
    \begin{align}
    e^{- \epsilon k} \E \left[ K(A | \tilde{X}) | S = 0 \right] - \delta k  \leq
    \E \left[ K(A | \tilde{X}) | S = k \right] \leq e^{\epsilon k} \E \left[ K(A | \tilde{X}) | S = 0 \right] + \delta k e^{\epsilon k}, \label{eq:iteratively_applied_equation_conditional_D_privacy}
    \end{align}
    for $k=0,1,\dots,n$. By symmetry of the argument, the same inequalities hold for ${X}$ in place of $\tilde{X}$. Using the above inequalities, we can bound
    \begin{equation*}
    P_1 K(A |\tilde{X}) = \E K(A |\tilde{X}) = \E^S  \E \left[ K(A | \tilde{X}) | S \right],
    \end{equation*}
    by
    \begin{equation*}
    \E^S e^{S \epsilon} \E \left[ K(A | \tilde{X}) | S = 0 \right] + \delta \E S e^{S \epsilon}. \label{eq:privacy_kernel_iteration_ub}
    \end{equation*}
    Similarly, applying \eqref{eq:iterative_equation_conditional_D_privacy} with ${X}$ in place of $\tilde{X}$, we find
    \begin{align}
    P_2 K(A | {X}) &= \E^S  \E \left[ K(A | {X}) | S \right] \nonumber  \\
    &\geq \E^S e^{-S \epsilon} \E \left[ K(A | {X}) | S = 0 \right] - \E \delta S. \label{eq:privacy_kernel_iteration_lb}
    \end{align}
    Combining the two inequalities with \eqref{eq:ch1_if_Dis0_tildeX_is_X}, we obtain that
    \begin{equation}
    P_1 K(A |\tilde{X}) \leq \frac{\E^S e^{S \epsilon}}{  \E^S e^{-S \epsilon} } \left( P_2 K(A | {X}) + \E^S \delta S\right) + \delta \E S e^{S \epsilon}.
    \end{equation}
    In view of the moment generating function of the binomial distribution,
    \begin{align*}
    \frac{\E^S e^{S \epsilon}}{  \E^S e^{-S \epsilon} } = \left( \frac{1 + p_{} (e^{\epsilon} -1) }{1 + p_{} (e^{-\epsilon} -1)} \right)^n \leq e^{ 4n p_{} \epsilon},
    \end{align*}
    where the inequality follows from $0\leq \epsilon,p_{} \leq 1$, the inequality $ e^x - e^{-x} \leq 3x$ for $0 \leq x \leq 1$ and Taylor expanding $\log(1+x) = x - {x^2}/{2} + \ldots$. By Chebyshev's association inequality (e.g. Theorem 2.14 in \cite{boucheron_concentration_2013}), $\E^S  S \E^S e^{S \epsilon} \leq \E^S  S e^{S \epsilon}$. Consequently, using the nonnegativity of $S$,
    \begin{align*}
    \delta \left( \frac{\E^S e^{S \epsilon}}{  \E^S e^{-S \epsilon} } \E^S  S + \E S e^{S \epsilon} \right) &\leq 2\delta \E S e^{S \epsilon}.   
    \end{align*}
    Lemma \ref{lem:DepowerD_binomial_expectation} (a straightforward calculation) now finishes the proof. 
    \end{proof}
    
    Next, we construct the coupling that will be used in the proof of Lemma \ref{lem:product_of_likelihoods_privacy_bound_impure_DP}, in conjunction with Lemma \ref{lem:coupling_general} above. The couplings use similar ideas to the one constructed in \cite{narayanan2022private}.
    
    \begin{lemma}\label{lem:grand_coupling}
        Let $K^j$ satisfy an $(\epsilon,\delta)$-DP constraint for $0 < \epsilon \leq 1$. 

        Consider $\pi = N(0,c_\alpha^{1/2} d^{-1}_L \rho^2 \bar{\Gamma})$, with $\rho^2$ satisfying \eqref{eq:nonasymptotic_testing_lower_bound_local_randomness_dL} or \eqref{eq:nonasymptotic_testing_lower_bound_shared_randomness_dL}, with $\epsilon \leq 1/\sqrt{n}$ and $\delta \leq c_\alpha (m^{-1} \wedge \epsilon)$.
        
        For all measurable sets $A$ it holds that
        \begin{equation}\label{eq:coupling_every_set_ub}
        P_\pi K^j\left(A | X^{(j)},u\right) \leq \left(1+c_\alpha^{1/4} m^{-1/2}\right) P_0 K^j\left(A | X^{(j)},u\right) + 2\delta + \frac{c_\alpha}{m^{3/2}}
        \end{equation}
        and
        \begin{equation}\label{eq:coupling_every_set_lb}
        P_\pi K^j\left(A | X^{(j)},u\right) \geq \left(1-c_\alpha^{1/4} m^{-1/2}\right) P_0 K^j\left(A | X^{(j)},u\right) - 2\delta - \frac{c_\alpha}{m^{3/2}}
        \end{equation}
        for all $c_\alpha > 0$ small enough.
        \end{lemma}
    
    \begin{proof}[Proof of Lemma \ref{lem:grand_coupling}]
    
    We suppress the dependence of the Markov kernels on the draw of the shared randomness $u$, as it is of no relevance to the proof. Consider $\tilde{X}^{(j)} \sim P_\pi$ and ${X}^{(j)} \sim P_0$. Since ${X}^{(j)}_{L'k} \overset{d}{=} \tilde{X}^{(j)}_{L'k}$ for all $L' > L$, it suffices to construct couplings for $(\tilde{X}^{(j)}_{L}, {X}^{(j)}_{L})$. We construct two couplings, one for each of the different regimes of $\epsilon$. That is, for each of the regimes, we derive a joint distribution of $(\tilde{X}^{(j)}_{L}, {X}^{(j)}_{L})$ called $\P_{\pi,0}$ such that $\tilde{X}^{(j)}_{L} \sim \P_{\pi,0}^{\tilde{X}^{(j)}_{L}} =  P_\pi$ and ${X}^{(j)}_{L} \sim \P_{\pi,0}^{X^{(j)}_{L}} =  P_0$. The specific couplings that we construct aim at assuring $\mathrm{d}_H(\tilde{X}^{(j)}_{L},{X}^{(j)}_{L})$ is small with high probability. When constructing the coupling below, it suffices to consider $\sigma = 1$ without loss of generality (i.e. by rescaling $X^{(j)}$ and $\tilde{X}^{(j)}$). After the construction of both the couplings, the lemma follows by an application of Lemma \ref{lem:coupling_general}. 
    
    \textbf{Case 1:} Consider $ 1/\sqrt{n} \geq  \epsilon \geq {1}/{\sqrt{m{nd_L}}}$. The construction we follow is similar to that used in Theorem D.6 of \cite{pmlr-v178-narayanan22a}, whose dependencies are favorable for our purpose. 
    
    If $n=1$, Pinsker's inequality followed Lemma \ref{lem:chi_sq_div_bounds_KL} and Lemma \ref{lem:Bu_bound_bandwidth_constraints} applied with $m=1$ yield that
    \begin{align*}
    \| P_0 - P_\pi \|_{\mathrm{TV}} &\leq \sqrt{\frac{1}{2} D_{\chi^2}( P_0; P_\pi)} \leq C \frac{\sqrt{c_\alpha} \rho^2}{\sqrt{d_L}}
    \end{align*}
    for a universal constant $C>0$ contingent only on $\|\bar{\Gamma}\|$. By Lemma \ref{lem:coupling_and_TV}, there exists a coupling $\P_{\pi,0}$ such that $\tilde{X}^{(j)} \sim \P_{\pi,0}^{\tilde{X}^{(j)}_L} =  P_\pi$ and ${X}^{(j)} \sim \P_{\pi,0}^{X^{(j)}_L} =  P_0$ and
    \begin{equation*}
    p:=\P\left( \tilde{X}^{(j)}_L \neq {X}^{(j)}_L \right) \leq \left(C \frac{\sqrt{c_\alpha} \rho^2}{\sqrt{d_L}}\right) \wedge 1.
    \end{equation*}
    Applying Lemma \ref{lem:coupling_general}, it follows that
    \begin{align*}
    P_\pi^n K^j(A| \tilde{X}^{(j)}_L) &= \E_{\pi,0} K^j(A| \tilde{X}^{(j)}_L)  \\
    &\leq  e^{4 \epsilon n p} P_0 K^j(A| {X}^{(j)}_L) + 2 \delta n p e^{\epsilon n p}.
    \end{align*}
    By applying condition \eqref{eq:nonasymptotic_testing_lower_bound_local_randomness_dL} or \eqref{eq:nonasymptotic_testing_lower_bound_shared_randomness_dL} and the bound on $p$, we obtain that
    \begin{equation*}
    e^{C\frac{ \sqrt{c_\alpha} \epsilon \rho^{2} }{d^{1/2}_L}} \leq 1 +  C \sqrt{c_\alpha} / \sqrt{m} 
    \end{equation*}
    Similarly, using that $\delta \leq \epsilon/\sqrt{m}$,
    \begin{align*}
    \delta  p e^{\epsilon p} \leq \delta + C' c_\alpha / m^{3/2}.
    \end{align*}
    The first identity we wish to show, i.e. \eqref{eq:coupling_every_set_ub}, now follows for $n =1$ and a sufficiently small enough choice of $c_\alpha > 0$.
    
    For what follows, take $n > 1$. Consider $V$ a uniform draw from the unit sphere in $\R^{d_L}$ and $Z \sim N(0,I_{d_L})$, both independent of the other random variables considered. We have 
    \begin{equation*}
    \overline{X^{(j)}_L} := \frac{1}{n} \underset{i=1}{\overset{n}{\sum}} X^{(j)}_{L;i}  \overset{d}{=} \frac{\|Z\|_2}{\sqrt{n}} V
    \end{equation*}
    for ${X}^{(j)}_L \sim \P_0^{X^{(j)}_L}$ (see e.g. \cite{vershynin_high-dimensional_2018} Exercise 3.3.7). Similarly, 
    \begin{equation*}
    \overline{\tilde{X}^{(j)}_L} \overset{d}{=} \frac{\left\| (I_{d_L} + \frac{nc_\alpha^{1/2} \rho^2}{d_L}\bar{\Gamma})^{1/2} Z \right\|_2}{\sqrt{n}} V.
    \end{equation*}
    Next, we note that for $\eta_1,\dots,\eta_n \sim N(0,I_{d_L})$ independent of $X^{(j)} = ({X}^{(j)}_1,\dots,{X}^{(j)}_n)$, we have
    \begin{equation}\label{eq:average_minus_etas}
     X^{(j)}_L \overset{d}{=} \left(\overline{X^{(j)}} + \eta_i - \frac{1}{n} \underset{i=1}{\overset{n}{\sum}} \eta_i\right)_{1 \leq i \leq n}.
    \end{equation}
    To see this, note that both the left- and right-hand side are mean zero Gaussian and
    \begin{align}\label{eq:average_minus_etas_covariances}
    \E \left(\overline{X^{(j)}_L} + \eta_i - \frac{1}{n} \underset{i=1}{\overset{n}{\sum}} \eta_i\right) \left(\overline{X^{(j)}_L} + \eta_k - \frac{1}{n} \underset{i=1}{\overset{n}{\sum}} \eta_i\right)^\top &= \\
    \frac{1}{n}I_{d_L} + \mathbbm{1}_{i = k} I_d - \frac{2}{n} I_d + \frac{1}{n} I_{d_L} & = \mathbbm{1}_{i = k} I_{d_L}, \nonumber
    \end{align}
    which means that the covariances of the left-hand side and right-hand side of \eqref{eq:average_minus_etas} are equal too. Noting that $\tilde{X}^{(j)}_L \overset{d}{=} (F + X^{(j)}_{L;i})_{i \in [n]}$ and $\overline{\tilde{X}^{(j)}_L} \overset{d}{=} F + \overline{X^{(j)}_L}$, where $F \sim N(0, \sqrt{c_\alpha} d^{-1}_L \rho^2 \bar{\Gamma} )$ is independent of $\overline{X^{(j)}_L}$, it follows that
    \begin{align*}
    \tilde{X}^{(j)}_L \overset{d}{=} \left(\overline{\tilde{X}^{(j)}_L} + \eta_i - \frac{1}{n} \underset{i=1}{\overset{n}{\sum}} \eta_i\right)_{1 \leq i \leq n}
    \end{align*}
    by similar reasoning. Since the matrix $(I - V V^\top)$ is idempotent, we have that
    \begin{equation*}
    \eta_i = V V^\top \eta_i + (I - V V^\top) \eta_i
    \end{equation*}
    where $V V^\top \eta_i$ is independent from $(I - V V^\top) \eta_i$ and $V^\top \eta_i$ is standard normally distributed, both conditionally and unconditionally on $V$. We can write
    \begin{equation*}
    \eta_i - \frac{1}{n} \underset{i=1}{\overset{n}{\sum}} \eta_i = V V^\top \eta_i - \frac{1}{n} \underset{i=1}{\overset{n}{\sum}} V V^\top \eta_i + G_i,
    \end{equation*}
    where 
    \begin{equation*}
    G_i \equiv G_i(\eta_i) := (I - V V^\top) {\eta}_i - \frac{1}{n} \underset{i=1}{\overset{n}{\sum}} (I - V V^\top) {\eta}_i
    \end{equation*}
    and $G_i$ independent of $V V^\top \eta_i - \frac{1}{n} \underset{i=1}{\overset{n}{\sum}} V V^\top \eta_i$. Let $\tilde{\eta}_i$ be identically distributed to $\eta_i$, for $i=1,\dots,n$. Combining the above, we have that
    \begin{align}\label{eq:coupling_decompostion}
    X^{(j)}_L &\overset{d}{=} \left\{ V\left( \frac{\|Z\|_2}{\sqrt{n}} + V^\top \eta_i - \frac{1}{n} \underset{i=1}{\overset{n}{\sum}} V^\top \eta_i \right) + G_i \right\}_{i \in [n]} =: ({C}_i)_{i \in [n]}, \\
    \tilde{X}^{(j)}_L &\overset{d}{=} \left\{ {V} \left( \frac{\left\| (I_d + \frac{nc_\alpha^{1/2} \rho^2}{d}\bar{\Gamma})^{1/2} Z \right\|_2}{\sqrt{n}} + {V}^\top \tilde{\eta}_i - \frac{1}{n} \underset{i=1}{\overset{n}{\sum}} V^\top \tilde{\eta}_i \right) + G_i \right\}_{i \in [ n ]} =: (\tilde{C}_i)_{i \in [n]}. 
    \end{align}
    As further notations, we introduce
    \begin{align*}
    \zeta_i &:= {\|Z\|_2}/{\sqrt{n}} + V^\top \eta_i- \frac{1}{n} \underset{i=1}{\overset{n}{\sum}} V^\top {\eta}_i \\ 
    \tilde{\zeta_i} &:= {\| (I_{d_L} + \frac{nc_\alpha^{1/2} \rho^2}{d_L}\bar{\Gamma})^{1/2} Z \|_2}/{\sqrt{n}} + {V}^\top \tilde{\eta}_i - \frac{1}{n} \underset{i=1}{\overset{n}{\sum}} V^\top \tilde{\eta}_i.
    \end{align*}
    We have that $\zeta_i | Z \sim N\left(\frac{\|Z\|_2}{\sqrt{n}}, \left(1-\frac{1}{n}\right)\right)$ and
    \begin{equation*}
      \tilde{\zeta}_i|Z \sim N\left(\frac{\left\| (I_{d_L} + \frac{nc_\alpha^{1/2}\rho^2}{d}\bar{\Gamma})^{1/2} Z \right\|_2}{\sqrt{n}},\left(1-\frac{1}{n}\right)\right).
    \end{equation*}
    By e.g. Lemma \ref{lem:TV_distance_bound_many-normal-means}, we find that their respective push forward measures $\P^{\zeta_i|Z}$ and $\P^{\tilde{\zeta}_i|Z}$ satisfy 
    \begin{align*}
    \| \P^{\zeta_i|Z} - \P^{\tilde{\zeta}_i|Z} \|_{\mathrm{TV}} &\leq \frac{1}{2 (1-1/n)\sqrt{n}} \left| {\|Z\|_2} - {\left\| (I_{d_L} + \frac{nc_\alpha^{1/2}\rho^2}{d_L}\bar{\Gamma})^{1/2} Z \right\|_2} \right|  \\ 
    &\leq \frac{\sqrt{n} c_\alpha^{1/2}\rho^2}{d_L} \left| \frac{ Z^\top \bar{\Gamma} Z}{\sqrt{Z^\top I_{d_L} Z} + \sqrt{Z^\top (I_{d_L} + \frac{nc_\alpha^{1/2}\rho^2}{d_L} \bar{\Gamma}) Z}} \right| \\ 
    &\leq \| \bar{\Gamma} \| \frac{\sqrt{n} c_\alpha^{1/2}\rho^2}{d_L} \| Z \|_2,
    \end{align*}
    where the second inequality follows from $n>1$ in addition to the identity $(\sqrt{a} - \sqrt{b})(\sqrt{a} + \sqrt{b}) = a - b$ and the final inequality follows from the fact that $\bar{\Gamma}$ is positive semidefinite and $\bar{\Gamma} \leq \| \bar{\Gamma} \| I_{d_L}$. By Lemma \ref{lem:coupling_and_TV}, there exists a coupling of $\zeta_i|Z$ and $\tilde{\zeta}_i|Z$ such that
    \begin{equation}\label{eq:coupling_zetas}
    \P \left( \zeta_i \neq \tilde{\zeta}_i | Z \right) \leq \frac{\sqrt{n} c_\alpha^{1/2}\rho^2 \| \bar{\Gamma} \|  \| Z \|_2}{2d_L} \wedge 1.
    \end{equation}
    Take $\P^{\zeta_i,\tilde{\zeta}_i|Z}$ satisfying \eqref{eq:coupling_zetas}, $G_i = \tilde{G}_i$ and set $(X^{(j)}_L,\tilde{X}^{(j)}_L)=(C,\tilde{C})$ under $\P_{\pi,0}$. We obtain that $(\tilde{C}_i,C_i)$ is independent of $(\tilde{C}_k,C_k)$ for $k \neq i$ and $\tilde{\zeta}_i = \zeta_i$ implies $\tilde{C}_i = C_i$. By the independence structure, it holds for any joint distribution $\P$ of $V$, $Z$, $\zeta = (\zeta_i)_{i \in [n]}$, $\tilde{\zeta}_1,\dots,\tilde{\zeta}_n,(V,Z)$ and $G = (G_i)_{i \in [n]}$ that
    \begin{equation*}
        d\P^{C,\tilde{C}} =  d \underset{i=1}{\overset{n}{\bigotimes}}  \P^{C_i,\tilde{C}_i}.
       \end{equation*}
    Consequently,
    \begin{equation*}
         \underset{i=1}{\overset{n}{\sum}} \mathbbm{1}\{\tilde{C}_i \neq {C}_i \} \sim \text{Bin}(n,p),
    \end{equation*}
    with $p := \P \left( \tilde{C}_i \neq {C}_i  \right)$. By \eqref{eq:coupling_zetas}, $\| \bar{\Gamma} \| \asymp 1$ and the fact that $\|Z\|_2$ is $\sqrt{d}$-sub-exponential (using e.g. Proposition 2.7.1 in \cite{vershynin_high-dimensional_2018}), we obtain that 
    \begin{equation*}
    p = \E^Z \P \left( \zeta_i \neq \tilde{\zeta}_i | Z \right) \leq \frac{\tilde{C} n^{1/2} (c_\alpha^{1/4} \rho)^{2} }{d^{1/2}_L} \wedge 1,
    \end{equation*}
    for a universal constant $\tilde{C}>0$. Since $[\zeta_i = \tilde{\zeta}_i]|Z$ implies $[C_i = \tilde{C}_i]|Z$, we have that
    \begin{align*}
        p = \P( C_i = \tilde{C_i} ) \leq \E^Z \P \left( \zeta_i \neq \tilde{\zeta}_i | Z \right) \leq \frac{\tilde{C} n^{1/2} c_\alpha^{1/4} \rho^2}{d^{1/2}_L}.
    \end{align*}
    To summarize, we have now obtained that there exists a joint distribution $\P_{\pi,0}$ of ${(X^{(j)},\tilde{X}^{(j)})}$ such that $\left(X^{(j)},\tilde{X}^{(j)}\right)$ satisfy 
    \begin{equation*}
        S := \underset{i=1}{\overset{n}{\sum}} \mathbbm{1}\{\tilde{X}^{(j)}_i \neq {X}^{(j)}_i \} \sim \text{Bin}(n,p),
    \end{equation*}
    and
    \begin{equation*}
    p = \P \left( X_i^{(j)} \neq \tilde{X}_i^{(j)} \right) \leq \frac{\tilde{C} n^{1/2} (c_\alpha^{1/4} \rho)^{2} }{d^{1/2}_L} \wedge 1.
    \end{equation*}
    Let $\E_{\pi,0}$ denote its corresponding expectation. Consequently, by applying Lemma \ref{lem:coupling_general}, we have for any measurable $A$ that
    \begin{align*}
    P_\pi^n K^j(A| \tilde{X}^{(j)}) &= \E_{\pi,0} K^j(A| \tilde{X}^{(j)}) =  \E^{\tilde{X}^{(j)},X^{(j)}}_{\pi,0} K^j(A| \tilde{X}^{(j)})  \\
    &\leq  e^{4 \epsilon n p} P_0 K^j(A| {X}^{(j)}) + 2 \delta n p e^{2\epsilon n p}.
    \end{align*}
     It follows that 
    \begin{align*}
     e^{4 \epsilon n p} &\leq  1 + C' \frac{ \epsilon n^{3/2} \rho^{2} }{\sqrt{c_\alpha} d^{1/2}_L},
    \end{align*}
    for a universal constant $C'>0$, where the inequality follows from the fact that under the assumptions on $\rho^2$ (i.e. condition \eqref{eq:nonasymptotic_testing_lower_bound_local_randomness_dL} or \eqref{eq:nonasymptotic_testing_lower_bound_shared_randomness_dL}) that
    \begin{equation*}
    \frac{\epsilon n^{3/2} c_\alpha^{1/2} \rho^{2} }{d^{1/2}_L} \leq \sqrt{c_\alpha} / \sqrt{m}
    \end{equation*}
    and a sufficiently small enough choice of $c_\alpha > 0$. Similarly, using that $\delta \leq \epsilon/\sqrt{m}$,
    \begin{align*}
    \delta  n p e^{2\epsilon n p} &\leq \delta + C' c_\alpha / m^{3/2}.
    \end{align*}

    \textbf{Case 2:} Consider $\epsilon \leq 1/\sqrt{mnd_L}$. We will make use of the total variation coupling between $\tilde{X}^{(j)}_i \sim N(f,I_{d_L})$ and $X^{(j)}_i \sim N(0,I_{d_L})$, as given by Lemma \ref{lem:coupling_and_TV}. Since
    \begin{equation*}
    \| N(0,I_{d_L}) - N(f,I_{d_L}) \|_{\mathrm{TV}} \leq \left(\frac{1}{2} \|f\|_2 \right) \wedge 2
    \end{equation*}
    (see e.g. Lemma \ref{lem:TV_distance_bound_many-normal-means}), we can couple the two data sets observation wise independently (simply taking the product space) such that 
    \begin{equation*}
      \underset{i=1}{\overset{n}{\sum}} \mathbbm{1}\{\tilde{X}^{(j)}_i \neq X^{(j)}_i\}  | f \sim \text{Bin}(n,p_f)
    \end{equation*}
    where $p_f = (\|f\|_2/4) \wedge 1$. Given $k \in \N$, $\|f\|_2 \overset{d}{=} d^{-1/2}_Lc_\alpha^{1/4} \rho \| N(0,I_{d_L}) \|_2$ and $\| N(0,I_{d_L}) \|_2$ is $\sqrt{d_L}$-sub-exponential we obtain (using e.g. Proposition 2.7.1 in \cite{vershynin_high-dimensional_2018})
    \begin{align*}
    \int p^k_f d\pi(f) \leq \int (\|f\|_2/4)^k d\pi(f) \leq \tilde{C}^k k^k (c_\alpha^{1/4} \rho)^k,
    \end{align*}
    for a universal constant $\tilde{C} > 0$. The assumed condition on $\rho$ yields
    \begin{equation*}
    {\epsilon n c_\alpha^{1/4} \rho}  \leq c_\alpha^{1/4} / \sqrt{m},
    \end{equation*}
    which by similar arguments as before implies
    \begin{align*}
     e^{4 \epsilon n p} &\leq 1 + C' c_\alpha^{1/4} / \sqrt{m}, \\
    \delta  n p e^{2\epsilon n p} &\leq \delta + C' c_\alpha^{1/2} / m^{3/2},
    \end{align*}
    for a universal constant $C' > 0$. By applying the claim at the start of the lemma and using the assumptions on $\rho$, we obtain that
    \begin{align*}
    P_\pi^n K^j(A| \tilde{X}^{(j)}) &= \E_{\pi,0} K^j(A| \tilde{X}^{(j)}) = \int \E_{f,0} K^j(A| \tilde{X}^{(j)}) d\pi(f) \\
    &\leq (1 + C c_\alpha^{1/4} / \sqrt{m}) P_0 K^j(A| {X}^{(j)}) +  \delta + C c_\alpha^{1/2} / m^{3/2} 
    \end{align*}
    as desired. Again, \eqref{eq:coupling_every_set_lb} follows by similar steps.
    \end{proof}
    
    With the above two lemmas in hand, we are ready to prove Lemma \ref{lem:product_of_likelihoods_privacy_bound_impure_DP}.
    
    \begin{proof}[Proof of Lemma \ref{lem:product_of_likelihoods_privacy_bound_impure_DP}]
        Write $\cL_{\pi,u}^j(Y^{(j)}) \equiv \cL_{\pi}^j$ and let $V_\pi \equiv V_\pi^j := \cL_{\pi}^j - 1$. Using that $\E_0 \mathscr{L}_{\pi}(\tilde{X}^{(j)}) = 1$ and that by the law of total probability
        \begin{equation*}
        \E_0^{Y^{(j)}|U=u}  \E_0 \left[  \mathscr{L}_{\pi}(\tilde{X}^{(j)}) \bigg| Y^{(j)}, U=u \right] = 1,
        \end{equation*}
        it follows that $\E_0^{Y^{(j)}|U=u} V_\pi = 0$ and
        \begin{equation*}
        \E_0^{Y^{(j)}|U=u} \left(\cL_{\pi}^j\right)^2 = 1 + \E_0^{Y^{(j)}|U=u} \cL_{\pi}^j (\cL_{\pi}^j - 1) = 1 + \E_\pi^{Y^{(j)}|U=u} V_\pi.
        \end{equation*}
        Define $V_\pi^+ := 0 \vee V_\pi$ and let $V_\pi^- = - (0 \wedge V_\pi)$, which are both nonnegative random variables, with $V_\pi = V_\pi^+ - V_\pi^-$. We have
        \begin{align}
        \E_\pi^{Y^{(j)}|U=u} V_\pi^+ &= \int_0^\infty \P_\pi^{Y^{(j)}|U=u} \left(V_\pi^+ \geq t \right) dt \nonumber \\
        &= \int_0^T \P_\pi^{Y^{(j)}|U=u} \left( V_\pi^+ \geq t \right) dt + \int_T^\infty \P_\pi^{Y^{(j)}|U=u} \left( V_\pi^+ \geq t \right) dt. \label{eq:B_u_bound_privacy_integral_split}
        \end{align}
        Taking $T = \frac{5 m^{1/2}}{{\alpha}}$, the second term is equal to zero as 
        \begin{equation*}
        V_\pi^+ \leq | {\cL}_{\pi,u}^j(y) - 1| \leq \frac{5 m^{1/2}}{{\alpha}}  \;\;\;P_0 K^j(\cdot|X^{(j)},u)\text{-a.s.}
        \end{equation*}
        and $P_\pi \sim P_0$ (which in turn implies $\P_\pi^{Y^{(j)}|U=u} \sim \P_0^{Y^{(j)}|U=u}$). The integrand of the first term equals $P_\pi^n K^j( \{ V_\pi \geq t \} | X^{(j)},u)$. By Lemma \ref{lem:grand_coupling}, it holds that 
        \begin{equation*}
        P_\pi K^j(\{ V_\pi \geq t \} | X^{(j)},u) \leq \left(1+c_\alpha^{1/4} m^{-1/2}\right) P_0 K^j(\{ V_\pi \geq t \} | X^{(j)},u) + \delta + \frac{c_\alpha}{m^{3/2}}.
        \end{equation*}
        It follows that \eqref{eq:B_u_bound_privacy_integral_split} is bounded from above by
        \begin{align*}
        \left(1+c_\alpha^{1/4} m^{-1/2}\right) \int_0^T \P_0^{Y^{(j)}|U=u} \left( V_\pi^+ \geq t \right) dt + T \delta  + T \frac{c_\alpha}{m^{3/2}} &\leq \\
        \left(1+c_\alpha^{1/4} m^{-1/2}\right) \E_0^{Y^{(j)}|U=u}  V_\pi^+ + T \delta  + T \frac{c_\alpha}{m^{3/2}}.
        \end{align*}
        Similarly, we have
        \begin{align}
        \E_\pi^{Y^{(j)}|U=u} V_\pi^- &= \int_0^\infty \P_\pi^{Y^{(j)}|U=u} \left(V_\pi^- \geq t \right) dt \nonumber \\
        &= \int_0^T \P_\pi^{Y^{(j)}|U=u} \left( V_\pi^- \geq t \right) dt + \int_T^\infty \P_\pi^{Y^{(j)}|U=u} \left( V_\pi^- \geq t \right) dt. \label{eq:B_u_bound_privacy_integral_split_from_below}
        \end{align}
        Choosing $T \geq 1$ here results in the second term being zero, as $L_\pi^j \geq 0$. Applying Lemma \ref{lem:grand_coupling}, the right-hand side of the above display is further bounded from below by
        \begin{align*}
        \left(1 - c_\alpha^{1/4} m^{-1/2}\right) \int_0^T \P_0^{Y^{(j)}|U=u} \left( V_\pi^- \geq t \right) dt - T \delta  - T \frac{c_\alpha}{m^{3/2}} &\geq \\
        \left(1-c_\alpha^{1/4} m^{-1/2}\right) \E_0^{Y^{(j)}|U=u}  V_\pi^- - T \delta  - T \frac{c_\alpha}{m^{3/2}},
        \end{align*}
        where the inequality uses $V_\pi^- \leq 1$. 
        
        Combining the above bounds with the fact that $V_\pi^+ + V_\pi^- = |V_\pi|$ and $\E_0^{Y^{(j)}|U=u} V_\pi = 0$ yields that
        \begin{align*}
        \E_\pi^{Y^{(j)}|U=u} V_\pi &= \E_\pi^{Y^{(j)}|U=u} V_\pi^+ - \E_\pi^{Y^{(j)}|U=u} V_\pi^- \leq 
          & \frac{c_\alpha^{1/4} \E_0^{Y^{(j)}|U=u} |V_\pi|}{\sqrt{m}} + 2 T \delta  + 2 T \frac{c_\alpha}{m^{3/2}}.
        \end{align*}
        Plugging in the choice of $T = 5m^{1/2}/\alpha$ and using that $\delta \leq c_\alpha m^{-3/2}$, we obtain
        \begin{align*}
        \E_\pi^{Y^{(j)}|U=u} V_\pi \leq \frac{c_\alpha^{1/4} \E_0^{Y^{(j)}|U=u} |V_\pi|}{\sqrt{m}} + \frac{20 c_\alpha }{m\alpha}.
        \end{align*}
        If $\E_0^{Y^{(j)}|U=u} |V_\pi| \lesssim m^{-1/2}$, we obtain $\E_\pi^{Y^{(j)}|U=u} V_\pi \lesssim m^{-1}(c_\alpha^{1/4} + c_\alpha/\alpha)$. Assume next that $\E_0^{Y^{(j)}|U=u} |V_\pi| \gtrsim m^{-1/2}$. Then, 
        \begin{equation*}
        \E_\pi^{Y^{(j)}|U=u} V_\pi \lesssim \frac{c_\alpha^{1/4} \E_0^{Y^{(j)}|U=u} |V_\pi|}{\sqrt{m}}.
        \end{equation*}
        Since $\E_\pi^{Y^{(j)}|U=u} V_\pi = \E_0^{Y^{(j)}|U=u} V_\pi^2$ and using that by Cauchy--Schwarz $\E_0^{Y^{(j)}|U=u} |V_\pi|$ is bounded above by $\sqrt{\E_0^{Y^{(j)}|U=u} V_\pi^2}$, we obtain that
        \begin{equation*}
        \sqrt{\E_0^{Y^{(j)}|U=u} V_\pi^2} \lesssim  C' c_\alpha^{1/4} m^{-1/2}
        \end{equation*}
        for a universal constant $C' > 0$ depending only on $\alpha$. In both cases, we obtain that
        \begin{equation*}
        \mathrm{B}^\pi_u = \underset{j=1}{\overset{m}{\Pi}} \left( 1 + \E_\pi^{Y^{(j)}|U=u} V_\pi \right) = \underset{j=1}{\overset{m}{\Pi}} \left( 1 + \E_0^{Y^{(j)}|U=u} V_\pi^2 \right) \leq e^{C \sqrt{c_\alpha}}
        \end{equation*}
        for universal constant $C > 0$, finishing the proof of the lemma.
        \end{proof}
    
    \subsubsection{Proof of Lemma \ref{lem:data_processing_impure_DP_Fisher_info}}
    \begin{proof}
       Consider without loss of generality $\sigma = 1$ (the general result follows by the $\sigma^{-1}$ rescaling). The bound $\text{Tr}(\Xi_u^j) \leq nd_L$ follows by the fact that conditional expectation contracts the $L_2$-norm; let $v \in \R^{d_L}$, then
        \begin{align*}
        v^\top \Xi_u^j v &=  \E_0^{Y^{(j)}}\E_0^{Y|U=u} \left[ v^\top \underset{i=1}{\overset{n}{\sum}}  {X}^{(j)}_i \bigg| Y^{(j)}, U=u \right] \E_0^{Y|U=u}\left[  \left(\underset{i=1}{\overset{n}{\sum}}  {X}^{(j)}_i \right)^\top v \bigg| Y^{(j)}, U=u \right] \\
        &=  \E_0^{Y^{(j)}}\E_0^{Y|U=u}\left[ v^\top \left( \underset{i=1}{\overset{n}{\sum}}  {X}^{(j)}_i \right) \bigg| Y^{(j)}, U=u \right]^2.
        \end{align*}
        Since the conditional expectation contracts the $L_2$-norm, we obtain that the latter is bounded by
        \begin{equation*}
             \E_0 v^\top \left( \underset{i=1}{\overset{n}{\sum}}  {X}^{(j)}_i \right) \left( \underset{i=1}{\overset{n}{\sum}}  {X}^{(j)}_i \right)^\top v= n \|v\|_2^2,
        \end{equation*}
        which completes the proof of the statement ``$\Xi_u^j \leq n I_{d_L}$'' and ``$\text{Tr}(\Xi_u^j) \leq n d_L$''.
        
     For second other bound on the trace, we start introducing the notations $\overline{X^{(j)}_L} = n^{-1} \sum_{i=1}^n X^{(j)}_{L;i}$ and
        \begin{equation*}
            G_i = \left \langle \E_0 \left[ n \overline{X^{(j)}}_L | Y^{(j)}, U=u \right], X^{(j)}_{L;i} \right \rangle.
        \end{equation*}
        For the remainder of the proof, consider versions of $X^{(j)}_L$ and $Y^{(j)}$ defined on the same probability given $U=u$, and we shall write as a shorthand
        \begin{equation*}
        \P^{j} \equiv \P^{(X^{(j)},Y^{(j)})|U=u}_0 \text{ and }\E^{j} \equiv \E^{(X^{(j)},Y^{(j)})|U=u}_0.
        \end{equation*}
        Consider random variables $V,W$ defined on the same probability space. It holds that
        \begin{equation*}
            \E W \E[ W | V] = \E \E[ W | V] \E[ W | V],
        \end{equation*}
        since $W - \E[ W | V]$ is orthogonal to $\E[ W | V]$. Combining this fact with the linearity of the inner product and conditional expectation, we see that
        \begin{equation}\label{eq:trace_to_innerproduct_projection}
            \text{Tr}(\Xi_u^j) =  \E^{Y^{(j)}|U=u}_0 \left\| \E_0[ n \overline{X^{(j)}} | Y^{(j)}, U=u ]\right\|_2^2 =  \underset{i=1}{\overset{n}{\sum}}  \E^{j} G_i.
        \end{equation}
        Define also
        \begin{equation*}
            \breve{G}_i = \left \langle \E_0[ n \overline{X^{(j)}} | Y^{(j)}, U=u ], \breve{X}_i^{(j)} \right \rangle,
        \end{equation*}
        where $\breve{X}_i^{(j)}$ is an independent copy of ${X}_i^{(j)}$ (defined on the same, possibly enlarged probability space) and note that $\E^j \breve{G}_i = 0$. Write $G_i^+ := 0 \vee G_i$ and $G_i^- = - (0 \wedge G_i)$. We have
        \begin{align*}
            \E^j G_i^+ &=  \int_0^\infty \P^j \left( G_i^+ \geq t \right) dt = \int_0^T \P^j \left( G_i^+ \geq t \right) + \int_T^\infty \P^j \left( G_i^+ \geq t \right)  \\
            &\leq e^{\epsilon} \int_0^T \P^j \left( \breve{G}_i^+ \geq t \right) dt +T \delta + \int_T^\infty \P^j \left( G_i^+ \geq t \right) \\ 
            &\leq  \int_0^T \P^j \left( \breve{G}_i^+ \geq t \right)dt + 2 \epsilon \int_0^T \P^j \left( \breve{G}_i^+ \geq t \right)dt +T \delta + \int_T^\infty \P^j \left( G_i^+ \geq t \right) \\ 
            &\leq  \int_0^\infty \P^j_0 \left( \breve{G}_i^+ \geq t \right)dt + 2 \epsilon \int_0^\infty \P^j \left( \breve{G}_i^+ \geq t \right)dt +T \delta + \int_T^\infty \P^j \left( G_i^+ \geq t \right),
        \end{align*}
        where in the second to last inequality follows by Taylor expansion and the fact that $\epsilon \leq 1$. Similarly, we obtain
        \begin{align*}
            \E^j G_i^-  &\geq  \int_0^T \P^j \left( G_i^- \geq t \right)   \\
            &\geq e^{-\epsilon} \int_0^T \P^j \left( \breve{G}_i^- \geq t \right) dt - T \delta  \\ 
            &\geq  \int_0^T \P^j \left( \breve{G}_i^- \geq t \right)dt - 2 \epsilon \int_0^\infty \P^j \left( \breve{G}_i^- \geq t \right)dt - T \delta  \\
            &\geq \int_0^\infty \P^j \left( \breve{G}_i^- \geq t \right)dt - 2 \epsilon \int_0^\infty \P^j \left( \breve{G}_i^- \geq t \right)dt - T \delta - \int_T^\infty \P^j \left( \breve{G}_i^- \geq t \right) dt.
        \end{align*}
        Putting these together with $G_i = G_i^+ - G_i^-$, we get
        \begin{align*}
            \E^j G_i &\leq \int_0^\infty \P^j \left( \breve{G}_i^+ \geq t \right)dt - \int_0^\infty \P^j \left( \breve{G}_i^- \geq t \right)dt + 2 \epsilon \int_0^\infty \P^j \left( |\breve{G}_i| \geq t \right)dt \\ &\;\;\;\;\;\;\;\;+ 2 T \delta + \int_T^\infty \P^j \left( G_i^+ \geq t \right) dt + \int_T^\infty \P^j \left( \breve{G}_i^- \geq t \right) dt \\
            &= \E^j \breve{G}_i + 2 \epsilon \E^j | \breve{G}_i | + 2 T \delta + \int_T^\infty \P^j \left( G_i^+ \geq t \right) dt + \int_T^\infty \P^j \left( \breve{G}_i^- \geq t \right) dt.
        \end{align*}
        The first term in the last display equals $0$. For the second term, observe that 
        \begin{equation*}
            \breve{G}_i \bigg| \left[Y^{(j)},X^{(j)},U=u\right] \sim N(0, \| \E_0[ n \overline{X^{(j)}} | Y^{(j)}, U=u ] \|_2^2),
        \end{equation*}
        so
        \begin{equation*}
            \E^j | \breve{G}_i | = \E^{X^{(j)},Y^{(j)}} \E^{\breve{X}^{(j)}} | \breve{G}_i | = \E \| \E[ n \overline{X^{(j)}} | Y^{(j)}, U= u ] \|_2 \leq \sqrt{\text{Tr}(\Xi_u^j)}
        \end{equation*}
        where the last inequality is Cauchy--Schwarz. To bound the terms
    \begin{equation*}
    \int_T^\infty \P^j \left( G_i^+ \geq t \right) dt + \int_T^\infty \P^j \left( \breve{G}_i^- \geq t \right) dt
    \end{equation*}
    we shall employ tail bounds, which follow after showing that $G_i$ is $\sqrt{d_L n}$-sub-exponential. To see this, note that by Jensen's inequality followed by the law of total probability,
        \begin{align*}
            \E^j e^{t |G_i|} &= \E^j e^{t |\langle \E_0[ n \overline{X^{(j)}_L} | Y^{(j)},U=u ], X_{L;i}^{(j)} \rangle|} \\ &= \int \int e^{t |\langle \E_0[ n \overline{X^{(j)}_L} | Y^{(j)} = y, U=u ], x \rangle|} d\P^{X^{(j)}|Y^{(j)}=y,U=u}_0(x) d\P^{Y^{(j)}|U=u}_0(y) \\
            &\leq \int \int \E_0 \left[ e^{t |\langle  n \overline{X^{(j)}_L} , x \rangle|} | Y^{(j)} = y,U=u \right] d\P^{X^{(j)}|Y^{(j)}=y,U=u}_0(x) d\P^{Y^{(j)}|U=u}_0(y) \\
            &\leq \E^{X^{(j)}}_0 e^{t |\langle  n \overline{X^{(j)}}_{L} , {X}_{L;i}^{(j)} \rangle|},
        \end{align*}
        where the last equality follows from the fact that conditional expectation contracts the $L_1$-norm and the fact that $U$ is independent of $X^{(j)}$. Next, we bound $\E^{X^{(j)}}_0  e^{t |\langle  n \overline{X^{(j)}_L} , X_{L;i}^{(j)} \rangle|}$. 
    
        By the triangle inequality and independence,
        \begin{align*}
        \E^{X^{(j)}}_0 e^{t |\langle  n \overline{X^{(j)}_L} , {X}_{L;i}^{(j)} \rangle|} &\leq \E^{X^{(j)}}_0 e^{t |\langle \sum_{k \neq i}^n X^{(j)}_{L;k} , {X}_{L;i}^{(j)} \rangle|} \E^{X^{(j)}_i}_0 e^{t |\langle  {X^{(j)}_{L;i}} , {X}_{L;i}^{(j)} \rangle|}.
        \end{align*}
        The random variable $\langle  {X^{(j)}_{L;i}} , {X}_{L;i}^{(j)} \rangle$ is $\chi_{d_L}^2$-distributed, so by standard computations (see e.g. Lemma 12 in \cite{szabo2022optimal_IEEE}) we obtain that
        \begin{align*}
        \E^{X^{(j)}_i}_0 e^{t |\langle  {X^{(j)}_{L;i}} , {X}_{L;i}^{(j)} \rangle|} = \left(\E e^{t N(0,1)^2 \rangle|} \right)^{d_L} \leq e^{td_L + 2t^2d_L},
        \end{align*}
        whenever $t \leq 1/4$. We also find that
        \begin{align*}
        \E^{X^{(j)}}_0 e^{t |\langle \sum_{k \neq i}^n X^{(j)}_{L;k} , {X}_{L;i}^{(j)} \rangle|} &= \E^{(X^{(j)}_k)_{k \neq i}}_0 \E^{X^{(j)}_i}_0 e^{t |\langle \sum_{k \neq i}^n X^{(j)}_{L;k} , {X}_{L;i}^{(j)} \rangle|} \\
        &= \E^{X^{(j)}} e^{\frac{t^2}{2} \left\| \sum_{k \neq i}^n X^{(j)}_{L;k} \right\|_2^2 } \\
        &= \left(\E e^{\frac{t^2 (n-1)}{2} N(0,1)^2 }\right)^{d_L} \\
        &\leq e^{\frac{1}{2}({t^2(n-1)}d_L + t^4(n-1)^2d_L)},
        \end{align*} 
        where the inequality follows again by e.g. Lemma 12 in \cite{szabo2022optimal_IEEE} if ${t^2 (n-1)^2} \leq 1/2$. By the fact that $G_i^+ \leq |G_i|$ and Markov's inequality, 
    \begin{align*}
        \P^j(G_i^+ > T) \leq \P^j(|G_i| > T) \leq e^{-tT}\E^j e^{t|G_i|}, \text{ for all } T,t>0.
    \end{align*}
    Combined with the above bounds for the moment generating function means that for $\delta = 0$, the result follows from letting $T \to \infty$. If $\delta > 0$, take $T = 8 (d+L \vee \sqrt{nd_L})  \log(1/\delta)$ to obtain that
        \begin{align*}
            \int_T^\infty \P^j \left( G_i^+ \geq t \right) dt \leq  e^{ -  \log(1/\delta)}.
        \end{align*} 
        It is easy to see that the same bound applies to $\int_T^\infty \P^j_0 \left( \breve{G}_i^- \geq t \right) dt$. We obtain that
        \begin{align*}
            \underset{i=1}{\overset{n}{\sum}}  \E^{j} G_i &\leq 2 n \epsilon\sqrt{\text{Tr}(\Xi_u^j)} + 16 \delta (d_L \vee \sqrt{nd_L}) \log(1/\delta) + 2 n \delta.
        \end{align*}
        If $\sqrt{\text{Tr}(\Xi_u^j)} \leq n \epsilon$, the lemma holds (there is nothing to prove). So assume instead that $\sqrt{\text{Tr}(\Xi_u^j)} \geq n \epsilon$. Combining the above display with \eqref{eq:trace_to_innerproduct_projection}, we get
        \begin{align*}
            \sqrt{\text{Tr}(\Xi_u^j)}  &\leq {2 n \epsilon} +   16 \delta \frac{d_L \vee \sqrt{nd_L}}{n\epsilon} \log(1/\delta) +  \frac{2}{\epsilon}  \delta.
        \end{align*}
        Since $x^p \log(1/x)$ tends to $0$ as $x \to 0$ for any $p>0$, the result follows for $\delta \leq  \left( \left(\frac{n}{d_L} \wedge \frac{n^{1/2}}{\sqrt{d_L}}  \right) \epsilon^2 \right)^{1+p}$ for some $p>0$ as this implies that the last two terms are $O(n\epsilon)$.
    \end{proof}

    \section{Proofs related to the optimal testing strategies}\label{sec:supp:proofs_optimal_testing_strategies}
    
    The proofs concerning the three sections \ref{ssec:procedure_I}, \ref{ssec:procedure_II} and \ref{ssec:procedure_III} are given in this section, divided across the subsections \ref{ssec:supp:procedure_I}, \ref{ssec:supp:procedure_II} and \ref{ssec:supp:procedure_III} respectively.
    
    We recall the notations $d_L = \sum_{l=1}^L 2^l$ for $L \in \N$, $f_L = \Pi_L f$ for $f \in \ell_2(\N)$ and $\overline{X^{(j)}_L} = n^{-1} \sum_{i=1}^n X^{(j)}_{L;i}$ for $j=1,\dots,m$.
    
    \subsection{Procedure I}\label{ssec:supp:procedure_I}
    
    We briefly recall the testing procedure outlined in Section \ref{ssec:procedure_I}. Let $\tau > 0$, $L \in \N$, $d_L := \sum^L_{l=1} 2^l$ and $V^{(j)}_{L;\tau} \sim \chi^2_{d_L}$ independent of $X^{(j)}$ the random map from $(\R^{d_L})^n$ to $\R$ defined by
    \begin{equation}\label{eq:supp:clipped_Lips_stat_rescaled}
    \tilde{S}_{L;\tau}^{(j)}(x) = \left[\frac{1}{\sqrt{d_L}} \left(\left\| \sigma^{-1} \sqrt{n} \overline{ x} \right\|^2_2 - {V^{(j)}_{L;\tau}} \right)\right]_{-\tau}^{\tau}.
    \end{equation} 
    Define furthermore 
    \begin{equation}\label{eq:supp:lips_constant_defined}
        D_\tau = \frac{\tilde{\kappa}_\alpha \log (N) \left(  \sqrt{n \sqrt{d_L} \tau} \vee \sqrt{nd_L}  \right)}{n \sqrt{d_L}}.
    \end{equation}
    
    Set $K_\tau = \lceil 2\tau D^{-1}_\tau \rceil$ and consider the set $\cC_{L;\tau} = \cA_{L;\tau} \cap \cB_{L;\tau}$, where
    \begin{align}\label{eq:def_C1_C2}
    \cA_{L;\tau} &= \left\{ (x_i) \in (\R^\infty)^n :  \bigg|\| \sigma^{-1} \textstyle \sum_{i \in \cJ} \Pi_L x_i \|_2^2 - kd_L \bigg| \leq \frac{1}{8} k D_\tau n \sqrt{d_L} \;\; \forall \cJ \subset [n], |\cJ| = k \leq K_\tau \right\}, \\
    \cB_{L;\tau} &= \left\{ (x_i) \in (\R^\infty)^n : \left| \langle \sigma^{-1} \Pi_L x_i, \sigma^{-1} \textstyle \sum_{k \neq i} \Pi_L x_k \rangle \right| \leq  \frac{1}{8} k D_\tau n \sqrt{d_L},  \;\; \forall i=1,\dots,n \right\}. \nonumber
    \end{align}
    Lemma \ref{lem:concentration_on_C} below shows that $X^{(j)}$ concentrates on $\cC_{L;\tau}$ under the null hypothesis. 
    
    \begin{lemma}\label{lem:concentration_on_C}
        Whenever $\sigma^{-2} n \| \Pi_L f\|_2^2 d^{-1/2}_L \leq \tau/2$, $\tau \leq n R^2 / \sqrt{d_L}$ and $\tilde{\kappa}_\alpha$ is taken large enough, it holds that 
        \begin{equation*}
        \P_f \left( X^{(j)} \notin \cC_{L;\tau} \right) \leq \frac{\alpha}{2N}.
        \end{equation*}
        \end{lemma} 
        \begin{proof}
        Since $\cC_{L;\tau} = \cA_{L;\tau} \cap \cB_{L;\tau}$, it suffices to show that $\cA_{L;\tau}^c$ and $\cB_{L;\tau}^c$ as defined in \eqref{eq:def_C1_C2} are small in $\P_f$-probability for a large enough choice of $\tilde{\kappa}_\alpha > 0$. 
        
        Define $\eta_\tau := D_\tau n \sqrt{d_L}/8$. For both sets, we proceed via a union bound:
        \begin{align}
        \P_f \left( X^{(j)} \notin \cA_{L;\tau} \right) &= \P_f \left( \exists \cJ \subset [n], \, |\cJ| \leq K_\tau : \left| \| \textstyle \underset{i \in \cJ}{\overset{}{\sum}} \sigma^{-1} X_{L;i}^{(j)} \|_2^2 - |\cJ| d_L \right|  > |\cJ| \eta_\tau \right) \nonumber \\ 
        &\leq \underset{k=1}{\overset{K_\tau}{\sum}} {n \choose k}  \text{Pr} \left( \left| \| \sigma^{-1} \sqrt{k} f_L - Z \|_2^2 -  d_L \right|  >    \eta_\tau \right) \label{eq:combinatorial_union_bound}
        \end{align}
        where we recall that $f_L$ is the projection of $f$ onto its first $d_L$ coordinates and $Z \sim N(0,I_{d_L})$. We have
        \begin{equation*}
         \| \sigma^{-1} \sqrt{k}f_L - Z \|_2^2 = \sigma^{-2} k \|f_L\|_2^2 - 2 \sigma^{-1} \sqrt{k} f_L^\top Z + \|Z\|_2^2. 
        \end{equation*}
        Recalling that $K_\tau = \lceil 2\tau D^{-1}_\tau \rceil$, we obtain that
        \begin{equation}\label{eq:K_bounded_by_gamma}
        K_\tau  \leq \frac{2 \tau n \sqrt{d_L} }{\tilde{\kappa}_\alpha \log (N) (\sqrt{ n \sqrt{d_L} \tau} \vee \sqrt{nd_L})} \leq \frac{\sqrt{n \sqrt{d_L} \tau}}{\tilde{\kappa_\alpha} \log(N)} \lesssim   \frac{\eta_\tau}{ \tilde{\kappa}_\alpha \log(N)}.
        \end{equation}
        By the assumptions of the lemma ($\sigma^{-2} n \|f_L\|_2^2 \leq \tau \sqrt{d_L}/2$ and $\tau \leq n R^2/ \sqrt{d_L}$), we obtain that $\sigma^{-2} k \|f_L\|_2^2 \leq K_\tau R^2$. Consequently, we have that for $\tilde{\kappa}_\alpha > 0$ large enough $\sigma^{-2} K_\tau \|f_L\|_2^2 < \eta_\tau/2$, so it holds that
        \begin{equation*}
        \text{Pr} \left(  \| \sigma^{-1} \sqrt{k} f - Z \|_2^2 -  d_L   >    \eta_\tau \right) \leq \text{Pr} \left(  \|Z \|_2^2 -  d_L - 2  \sigma^{-1} \sqrt{k} f^\top_L Z    >    \eta_\tau/2 \right).
        \end{equation*}
        Using that $\text{Pr}(A \cap B) + \text{Pr}(A \cap B^c) \leq \text{Pr}(A') + \text{Pr}(A \cap B^c)$ for $  A \cap B \subset A'$, it follows that the latter display is bounded above by
        \begin{equation*}
        \text{Pr} \left(  \| Z \|_2^2 -  d_L   >    \eta_\tau/4 \right) + \text{Pr} \left(  - 2 \sigma^{-1} \sqrt{k} f^\top_L Z    >    \eta_\tau/4 \right).
        \end{equation*}
        By e.g. Lemma \ref{lem : Chernoff-Hoeffding bound chisq}, the first probability is bounded by $e^{- d\eta_\tau / 8}$. Again using $K_\tau \|f_L\|_2^2 < \eta_\tau/2$, the second term is bounded by $e^{-\eta_\tau / 32}$, where we note that the second term equals zero in the case that $f=0$. The bound
        \begin{align*}
        \text{Pr} \left(  \| \sigma^{-1} \sqrt{k}f_L - Z \|_2^2 -  d_L  <   -  \eta_\tau \right) \leq  e^{- d_L \eta_\tau / 4} + e^{-\eta_\tau / 8}
        \end{align*}
        follows by similar reasoning. Combining the above with the elementary bound $\sum^{K_\tau}_{k=1} {n \choose k} \leq e^{K_\tau \log(n)}$ and \eqref{eq:K_bounded_by_gamma} means that
        \begin{align*}
        \P_f \left( X^{(j)} \notin \cA_{L;\tau} \right) &\leq  2\exp \left(  K_\tau \log(N) - \frac{ (1+d_L/2)\eta_\tau }{4} \right) \leq  \alpha/(4mn).
        \end{align*}
        Turning our attention to $\cB_{L;\tau}$, we find that $\P_f \left( X^{(j)} \notin \cB_{L;\tau} \right)$ is equal to
        \begin{align*}
          \P_f \left( \underset{i \in [n]}{\overset{}{\max}} \sigma^{-2} \left| \underset{k \in [n]\setminus \{i\}}{\overset{}{\sum}}  \langle X_i^{(j)}, X_k^{(j)} \rangle  \right| >  \eta_\tau \right) \leq  n \text{Pr} \bigg(   \left|\langle \sigma^{-1}f + Z, (n-1)\sigma^{-1} f+\sqrt{n-1}Z'  \rangle  \right| > & \eta_\tau \bigg),
        \end{align*}
        where $Z$ and $Z'$ are independent $N(0,I_d)$ random vectors. Using another union bound, the above is further bounded by
        \begin{align}\label{eq:1st_mention_once_asdasd} 
        \text{Pr} \bigg( &\sqrt{n-1} \langle Z, Z'  \rangle > \eta_\tau/2 - (n-1)\|\sigma^{-1} f_L\|_2^2 \bigg) + \text{Pr} \bigg( ({n-1}) \langle\sigma^{-1} f, Z'  \rangle + \sqrt{n-1}\langle \sigma^{-1} f, Z  \rangle > \eta_\tau/2  \bigg). 
        \end{align}
        Using that $n \|\sigma^{-1} f_L\|_2^2 \leq \tau \sqrt{d_L}/2$ and $\tau \leq n R^2/ \sqrt{d_L}$ by assumption of the lemma, we see that
        \begin{equation}\label{eq:n-times-square-norm-f-bound}
        n \| \sigma^{-1} f_L \|_2^2 \leq \tau \sqrt{d_L}/2 \leq \sqrt{n \sqrt{d} \tau} \frac{\sqrt{\sqrt{d} \tau}}{2\sqrt{n}} \leq \frac{R}{\log(N)\tilde{\kappa}_\alpha} \eta_\tau.
        \end{equation}
        For $\tilde{\kappa}_\alpha > 0$, the latter can be seen to be larger than $\eta_\tau / 4$. Consequently, the first term in \eqref{eq:1st_mention_once_asdasd} can be seen to be bounded by
        \begin{equation*}
        \text{Pr} \bigg( {\sqrt{n-1}} \langle Z, Z'  \rangle >  \eta_\tau/4 \bigg) \leq e^{- \frac{\eta_\tau}{4\sqrt{nd_L}}} \leq e^{- \kappa_\alpha \log(N)},
        \end{equation*}
        where the inequality follows from Lemma \ref{lem:inner_product_ind_gaussians}. Since $Z$ and $Z'$ are independent standard Gaussian vectors, the second term in \eqref{eq:1st_mention_once_asdasd} is bounded above by
        \begin{equation*}
        \text{Pr} \bigg( {\sqrt{2}({n-1})} \langle \sigma^{-1} f, Z  \rangle > \eta_\tau/2  \bigg) \leq e^{- \frac{ \eta_\tau^2}{8 n^2 \|\sigma^{-2} f_L\|_2^2}}.
        \end{equation*}
        By $n \| \sigma^{-1} f_L \|_2^2 \leq \tau \sqrt{d_L}/2$,
        \begin{equation*}
            n^2 \| \sigma^{-1} f_L \|_2^2 \leq n \tau \sqrt{d_L}/2 \leq \frac{\eta_\tau^2}{\log^2(N) \tilde{\kappa}_\alpha^2}.
        \end{equation*}
        Hence, we have obtained that
        \begin{align*}
        \P_f \left( X^{(j)} \notin \cB_{L;\tau} \right) &\leq \frac{\alpha}{4mn}
        \end{align*}
        for $\tilde{\kappa}_\alpha > 0$ large enough. This concludes the proof of the lemma.
        \end{proof}
    
    In Lemma \ref{lem:Lipschitz_constant_on_C}, it is shown that $x \mapsto S^{(j)}(x)$ is $D_\tau$-Lipschitz with respect to the Hamming distance on $\cC_{L;\tau}$, with $D_\tau$ as defined in \eqref{eq:supp:lips_constant_defined}.

    \begin{lemma}\label{lem:Lipschitz_constant_on_C} 
    The map $x \mapsto S^{(j)}_\tau(x)$ defined in \eqref{eq:clipped_Lips_stat_rescaled} is $D_\tau$-Lipschitz with respect to $(\R^{d})^n$-Hamming distance on $\cC_{L;\tau}$.
    \end{lemma}
    \begin{proof}
    Consider $x=(x_i)_{i \in [n]},\breve{x} = (\breve{x})_{i \in [n]} \in \cC_{L;\tau}$ with $k:=\mathrm{d}_H(x,\breve{x})$. If $k > \lceil 2\tau D^{-1}_\tau \rceil$, we have $|S^{(j)}_\tau(x) - S^{(j)}_\tau(\breve{x})| \leq 2\tau \leq  D_\tau k$. If $k \leq  \lceil 2\tau D^{-1}_\tau \rceil$, let $\cJ \subset [n]$ denote the indexes of columns in which $x$ and $\breve{x}$ differ. Define the sum of the elements that $x$ and $\breve{x}$ have in common as 
    \begin{equation*}
    v = \sigma^{-1} \sum_{i \in [n] \setminus \cJ} x_i, \text{ such that } \sigma^{-1} \underset{i=1}{\overset{n}{\sum}}  x_i = v + w \text{ and } \sigma^{-1} \underset{i=1}{\overset{n}{\sum}} \breve{x}_i = v + \breve{w}.
    \end{equation*}
    We have
    \begin{align*}
    S^{(j)}_\tau(x) - S^{(j)}_\tau(\breve{x}) &= \frac{n}{\sqrt{d_L}} \left(\left\|  n^{-1} (v+w) \right\|^2- \frac{V^{(j)}}{n} \right) - \frac{n}{\sqrt{d_L}} \left(\left\| n^{-1} (v+\breve{w})\right\|^2 - \frac{V^{(j)}}{n} \right)  \\ &= \frac{1}{n\sqrt{d_L}} \left( 2 \left\langle w , v \right\rangle - 2 \left\langle \breve{w} , v \right\rangle + \left\| w \right\|_2^2 - \left\| \breve{w} \right\|_2^2 \right).
    \end{align*}
    The last two terms are bounded by $kD_\tau /4$ since $x,\breve{x} \in \cA_{L;\tau}$. The first two terms equal
    \begin{align*}
    \frac{2}{n\sqrt{d_L}}   \left(   \left\langle w , v + w \right\rangle -  \left\langle  \breve{w} , v + \breve{w} \right\rangle  + \left\| \breve{w} \right\|_2^2  - \left\| w \right\|_2^2 \right),
    \end{align*}
    where the last two terms are bounded by $kD_\tau/2$. It holds that
    \begin{align*}
    \left\langle w , v + w \right\rangle -  \left\langle  \breve{w} , v + \breve{w} \right\rangle  = \sigma^{-2} \underset{i \in \cJ}{\overset{}{\sum}} \left(  \left\langle x_i , \underset{i \in [n] \setminus \cJ}{\overset{}{\sum}} x_i \right\rangle -  \left\langle \breve{x}_i , \underset{i \in [n] \setminus \cJ}{\overset{}{\sum}} \breve{x}_i \right\rangle + \| x_i \|_2^2 - \| \breve{x}_i \|_2^2 \right),
    \end{align*}
    which is bounded by $k D_\tau / 4$ for $x \in \cA_{L;\tau} \cap \cB_{L;\tau}$. Putting it all together and by symmetry of the argument, we obtain that
    \begin{equation*}
    \left|S^{(j)}_\tau(x) - S^{(j)}_\tau(\breve{x})\right| \leq D_\tau k.
    \end{equation*}
    \end{proof}
    
    Lemma \ref{lem:Lipschitz_extension_theorem} below show that there exists a measurable function ${S}_{L;\tau}^{(j)} : \R^{d_L} \to \R$, $D_\tau$-Lipschitz with respect to the Hamming distance, such that ${S}_{L;\tau}^{(j)}(X^{(j)}_L) = \tilde{S}_{L;\tau}^{(j)}(X^{(j)}_L)$ whenever $X^{(j)} \in \cC_{L;\tau}$. That is, letting $\Psi_L : \R^{\infty} \to \R^{d_L}$ be the coordinate projection for the first $d_L$ coordinates, the lemma allows a Lipschitz extension of the test statistic defined in \eqref{eq:supp:clipped_Lips_stat_rescaled} on $\Psi_L \cC_{L;\tau}$ to all of $\R^{d_L}$. The proof follows essentially the construction of McShane \cite{mcshane1934extension} for obtaining a Lipschitz extension with respect to the Hamming distance, but our lemma verifies in addition the Borel-measurability of the resulting map.
    
     We follow a construction that is in essence that of McShane \cite{mcshane1934extension}, whilst also verifying that such an extension is Borel measurable.
    
    \begin{lemma}\label{lem:Lipschitz_extension_theorem}
        Let $\cC \subset (\R^{d_L})^n$ and $S : \cC \to \R$ be a (Borel) measurable $D$-Lipschitz map with respect to the Hamming distance on $ (\R^{d_L})^n$. Then, there exists a map $\tilde{S}:  (\R^{d_L})^n \to \R$ measurable with respect to the Borel sigma algebra such that it is $D$-Lipschitz with respect to the Hamming distance on $ (\R^{d_L})^n $ and $\tilde{S} = S$ on $\cC$.
    \end{lemma}   
    
    \begin{proof}
    The case where $\cC$ is empty is trivial, so we shall assume $\cC$ to be nonempty. Consider the map $\breve{S}: (\R^{d_L})^n \to [-\infty,\infty)$ given by
    \begin{equation*}
    \breve{S}(x) = \begin{cases}
    S(x) \; \text{ if } x \in \cC, \\
    \inf \left\{ S(c) + D \mathrm{d}_H(c,x) \, : \, c \in \cC \right\} \; \text{ otherwise.}
    \end{cases} 
    \end{equation*}
    Fix any $c' \in \cC$. Since $S$ is $D$-Lipschitz with respect to the Hamming distance, we have for all $c \in \cC$ that
    \begin{equation*}
    S(c) + D \mathrm{d}_H(c,x) \geq S(c') - D \mathrm{d}_H(c',c) + D \mathrm{d}_H(c,x) \geq S(c') - D \mathrm{d}_H(c',x) > -\infty
    \end{equation*}
    where the last step follows from the triangle inequality. So, $\breve{S}$ is real valued. For all $x \in (\R^{d_L})^n$ and $\gamma > 0$, there exists $c_\gamma \in \cC$ such that
    \begin{equation*}
    \breve{S}(x) \geq S(c_\gamma) + D \mathrm{d}_H(c_\gamma,x) - \gamma.
    \end{equation*}
    So for $x,x' \in  \ell_2(\N)^n$,
    \begin{equation*}
    \breve{S}(x') - \breve{S}(x)  \leq \breve{S}(c_\gamma) + D \mathrm{d}_H(c_\gamma,x') - \breve{S}(c_\gamma) - D \mathrm{d}_H(c_\gamma,x) + \gamma \leq D \mathrm{d}_H(x,x') + \gamma.
    \end{equation*}
    By symmetry of the argument and since $\gamma > 0$ is given arbitrarily, we conclude that $\breve{S}$ is $D$-Lipschitz with respect to the Hamming distance. Note however, that this construction does not guarantee that $\breve{S}$ is measurable.
    
    For any map $H:(\R^{d_L})^n \to [-\infty,\infty]$, let ${H}^*$ denote its minimal Borel-measurable majorant. That is, a measurable map ${H}^*: (\R^{d_L})^n \to [-\infty,\infty]$ such that
    \begin{enumerate}
        \item $H \leq H^*$ and
        \item $H^* \leq T$ $\P_0$-a.s. for every measurable $T: (\R^{d_L})^n \to [-\infty,\infty]$ with $T \geq H$.
    \end{enumerate}
    Such a map exists by e.g. Lemma 1.2.1 in \cite{van1996weak}. The map $\tilde{S}:(\R^{d_L})^n \to \R$ defined by
    \begin{equation*}
    \tilde{S}(x) = \breve{S}^*(x) \mathbbm{1}_{x \notin \cC} + S(x) \mathbbm{1}_{x \in \cC}
    \end{equation*}
    is measurable and can be seen to be a Borel-measurable majorant of $\breve{S}$; following from the fact that sums and products of measurable functions are measurable, $\breve{S} \leq \breve{S}^*$ and $\breve{S}(x) = S(x)$ for $x \in \cC$.
    
    Furthermore, by combining the fact that $\tilde{S}$ is measurable with e.g. Lemma 1.2.2 in \cite{van1996weak}, we get 
    \begin{align}\label{eq:bound_on_tilde_S_measurability_lemma}
    |\tilde{S}(x) - \tilde{S}(x')| = |(\tilde{S}(x) - \tilde{S}(x'))^*| \leq  |\breve{S}(x) - \breve{S}(x')|^*,
    \end{align}
    where $(x,x') \mapsto |\breve{S}(x) - \breve{S}(x')|^*$ is minimal Borel-measurable majorant of $(x,x') \mapsto |\breve{S}(x) - \breve{S}(x')|$. Since $\breve{S}$ is $D$-Lipschitz with respect to the Hamming distance $(x,x')\mapsto \mathrm{d}_H(x,x')$, which is a measurable map,
    \begin{equation*}
    |\breve{S}(x) - \breve{S}(x')|^* \leq D \mathrm{d}_H(x,x').
    \end{equation*}
    From \eqref{eq:bound_on_tilde_S_measurability_lemma} it follows that for all $x, x' \in  (\R^{d_L})^n$,
    \begin{align*}
    |\tilde{S}(x) - \tilde{S}(x')| \leq D \mathrm{d}_H(x,x').
    \end{align*}
    We have obtained a map $\tilde{S}$ that is $D$-Lipschitz with respect to the Hamming distance, measurable and $\tilde{S} = S$ on $\cC$, concluding the proof.
    \end{proof}

    \subsubsection{Proof of Lemma \ref{lem:typeI_typeII_error_control_Lipschitz_ext_test_privacy}}
    
    Consider the transcript of Section \ref{ssec:procedure_I}, which we recall here as
    \begin{equation}\label{eq:supp:transcript_private_Euclidean_norm_lips_ext_for_one_clipping}
    Y^{(j)}_{L;\tau} = \gamma_\tau \breve{S}_{L;\tau}^{(j)}(X^{(j)}) + W^{(j)}_\tau, \quad \text{ with } \gamma_\tau := \frac{\epsilon}{D_\tau  \sqrt{ 2 \mathfrak{c} \log (2 / \delta)}},
    \end{equation}
    $\mathfrak{c}>0$, $W^{(j)}_\tau \sim N(0,1)$ independent for $j=1,\dots,m$ and $\tau > 0$. These transcripts are $(\epsilon/{\mathfrak{c}},\delta)$-differentially private for any $\epsilon > 0$ (see e.g. \cite{dwork2014algorithmic}). Define the test
    \begin{equation}\label{eq:supp:Lipschitz_ext_test_privacy_for_one_clipping}
    \varphi_\tau := \mathbbm{1} \left\{ \frac{1}{\sqrt{m}} \underset{j=1}{\overset{m}{\sum}} Y^{(j)}_{L;\tau} \geq  \left( \gamma_\tau \vee 1 \right) \kappa  \right\}.
    \end{equation}
    We will prove the following more general version of Lemma \ref{lem:typeI_typeII_error_control_Lipschitz_ext_test_privacy}.
    
    \begin{lemma}\label{lem:typeI_typeII_error_control_Lipschitz_ext_test_privacy_more_general}
    Consider the test $\varphi_\tau$ as defined by \eqref{eq:Lipschitz_ext_test_privacy_for_one_clipping}. Whenever $\tau/4 \leq \frac{n\| f_L \|_2^2}{ \sigma^2 \sqrt{d_L}}  \leq \tau/2$ and 
    \begin{equation}\label{eq:supp:condition_on_n_f_L_procedure_I}
      \|f_L \|_2^2 \geq C_\alpha \kappa \sigma^{2} \sqrt{   \mathfrak{c} \log(1/\delta) } \log^{} (N) \left( \frac{\sqrt{d_L}}{\sqrt{N} \sqrt{n} (\sqrt{n} \epsilon \wedge 1)} \right) \bigvee \left( \frac{1}{N n\epsilon^2} \right),
    \end{equation}
    for $C_\alpha > 0$ large enough, it holds that $\P_f(1-\varphi_\tau) \leq \alpha$. 
    \end{lemma}
    
    \begin{proof}
    On the event that $X^{(j)} \in \cC_\tau$ for all $j \in [m]$, we have that
    \begin{equation}\label{eq:terms_here_mentioned_later}
    \sum_{j=1}^m  Y^{(j)}_{\tau} = \sum_{j=1}^m \left(\gamma_\tau \breve{S}_\tau(X^{(j)}) + W^{(j)}_\tau\right) = \sum_{j=1}^m \left(\gamma_\tau S_\tau(X^{(j)}) + W^{(j)}_\tau\right).
    \end{equation}
    Consequently, $\P_f(1 - \varphi_\tau)$ is bounded above by
    \begin{align}\label{eq:procedure_I_concentration_union_bound}
     \P_f \left(\frac{1}{\sqrt{m}} \sum_{j=1}^m  \left(\gamma_\tau S_\tau(X^{(j)}) + W^{(j)}_\tau\right) \geq \left( \gamma_\tau \vee 1 \right) \kappa \right)  + \P_f \left( \exists j : X^{(j)} \notin \cC_\tau \right).
    \end{align}
    By Lemma \ref{lem:concentration_on_C} and a union bound, the second term is bounded above by $\alpha/2$.

    Under $\P_f$, it holds that
    \begin{equation}\label{eq:symmetrized_chi-square_under_alt}
    \frac{n}{\sqrt{d_L}} \left(\left\| \sigma^{-1} \overline{X^{(j)}_L} \right\|^2_2-\frac{V^{(j)}}{n}\right) \overset{d}{=} \frac{n \sigma^{-2} \|f_L\|_2^2}{\sqrt{d_L}} + 2 \frac{\sqrt{n}}{\sqrt{d_L}} \langle Z, \sigma^{-1} f\rangle + \frac{\|Z \|_2^2 - V^{(j)}}{\sqrt{d_L}},
    \end{equation}
    where $Z \sim N(0,I_{d_L})$. By assumption, $\frac{ n \|f_L\|_2^2}{\sqrt{d_L} \sigma^{2}} \leq \tau/2$, $\text{Var}(\frac{\sqrt{n}}{\sqrt{d_L}} \langle Z,\sigma^{-1} f_L \rangle) = n \sigma^{-2} \|f_L\|_2^2/d_L \leq \tau/2$ and $({\|Z \|_2^2 - V^{(j)}})/{\sqrt{d_L}}$ tends to a Gaussian with variance $4$ for large $d_L$. The second and third term in \eqref{eq:symmetrized_chi-square_under_alt} are symmetric in distribution about $0$, have uniformly bounded densities (since the Chi-square and normal densities are bounded, and the third term tends weakly to a Gaussian in $d_L$) and $\sigma^{-2}d^{-1/2}_L n \|f_L\|_2^2 \leq \tau/2$, which means that the conditions of Lemma \ref{lem:clipping_preserves_mean_under_alt} are satisfied. Applying said lemma (with $\mu = \sigma^{-2} d^{-1/2}_L n \|f_L\|_2^2$), we get that there exists a uniform constant $c > 0$ such that
    \begin{align*}
    \E_f \frac{1}{\sqrt{m}} \sum_{j=1}^m  \left( \gamma_\tau S_\tau(X^{(j)}) + W^{(j)}_\tau\right) \geq c\frac{\sqrt{m} n^{}  \| f_L \|_2^2 \gamma_\tau }{\sigma^2 \sqrt{d_L}}.
    \end{align*}
    Under $\P_f$, by independence of the data and the Gaussian noise,
    \begin{align*}
    \text{Var}_f\left(\frac{1}{\sqrt{m}} \sum_{j=1}^m \gamma_\tau S_\tau(X^{(j)}) + W^{(j)}_\tau\right) &= 1 + \text{Var}_f\left(\gamma_\tau S_\tau(X^{(1)})\right).
    \end{align*}
    Since 
    \begin{equation*}
    \E_f \frac{n}{\sqrt{d_L}} \left(\left\| \sigma^{-1} \overline{X^{(j)}}_L \right\|^2_2-\frac{V^{(j)}}{n}\right) = \frac{n\sigma^{-2} \|f_L\|_2^2}{\sqrt{d_L}} \leq \tau/2,
    \end{equation*}
    Lemma \ref{lem:folklore_clipping_reduces_variance} yields
    \begin{align*}
    \text{Var}_f\left(\gamma_\tau S_\tau(X^{(1)})\right) &\leq \gamma_\tau^2 \text{Var}_f\left(\frac{n}{\sqrt{d_L}} \left(\left\| \sigma^{-1} \overline{X^{(j)}} \right\|^2_2-\frac{V^{(j)}}{n} \right)\right) \\
    &\leq \gamma_\tau^2\left(\frac{4n\|f_L\|_2^2}{\sigma^2 d_L} + 4\right).
    \end{align*}
    Assume now that for all $C_\alpha > 0$ large enough,
    \begin{equation}\label{eq:mention_this_holds_later_lips_ext_proof}
    \left( \gamma_\tau \vee 1 \right) \kappa \leq  c \frac{1}{2}{\sqrt{m} n^{} \sigma^{-2}  \| f_L \|_2^2 \gamma_\tau} / \sqrt{d_L},
    \end{equation}
    which is a claim we shall prove later on. Then, the first term in \eqref{eq:procedure_I_concentration_union_bound} is bounded above by
    \begin{equation}\label{eq:type_II_error_after_claim_lips_ext_test_proof}
    \P_f \left(\frac{1}{\sqrt{m}} \sum_{j=1}^m  \left( \gamma_\tau S_\tau(X^{(j)}) + W^{(j)}_\tau - \E_f(S_\tau(X^{(j)}) + W^{(j)}_\tau) \right) < - c\frac{\sqrt{m} n^{} \sigma^{-2}  \| f_L \|_2^2 \gamma_\tau}{2\sqrt{d_L}} \right),
    \end{equation}
    which, by Chebyshev's inequality is bounded by
    \begin{align*}
    \left( c \frac{\sqrt{m} n^{} \sigma^{-2}  \| f_L \|_2^2 \gamma_\tau}{2\sqrt{d_L}} \right)^{-2} \left({1 + \gamma_\tau^2 \left(\frac{4 \sigma^{-2} n\|f_L \|_2^2}{d_L} + 4\right)} \right) &\lesssim \\
    \left(\frac{\sqrt{m} n^{} \sigma^{-2} \| f_L \|_2^2 \gamma_\tau }{\sqrt{d_L} }\right)^{-2} + \left({{m} n^{} \sigma^{-2} \| f_L \|_2^2 }\right)^{-1} + \left(\frac{\sqrt{m} n^{} \sigma^{-2} \| f_L \|_2^2}{\sqrt{d_L}}\right)^{-2}. \;\; & 
    \end{align*}
    For $f$ satisfying \eqref{eq:lem:tau_specific_minimal_signal_size}, the last two terms are easily seen to be smaller than $\alpha/6$ for a large enough to choice for $C_\alpha$. To see that this is also true for the first term, recall that $\gamma_\tau := \frac{\epsilon}{D_\tau \sqrt{2 \mathfrak{c} \log(2/\delta)}}$, with
    \begin{equation*}
        D_\tau = \frac{\tilde{\kappa}_\alpha \log(N) ( \sqrt{n \sqrt{d_L} \tau} \vee \sqrt{nd_L})}{n \sqrt{d_L}},
    \end{equation*}
    which yields that the square root of the first term equals
    \begin{align*}
    \frac{\sqrt{m} n \sigma^{-2}  \| f_L \|_2^2 \epsilon }{ \sqrt{d_L} D_\tau \sqrt{2 \mathfrak{c} \log(2/\delta)} } = \frac{\sqrt{m} n^{2} \sigma^{-2}  \| f_L \|_2^2 \epsilon }{\tilde{\kappa}_\alpha \log(N) ( \sqrt{n \sqrt{d_L} \tau} \vee \sqrt{nd_L}) \sqrt{2 \mathfrak{c} \log(2/\delta)} }.
    \end{align*}
    When the maximum is taken in $\sqrt{nd_L}$, \eqref{eq:supp:condition_on_n_f_L_procedure_I} leads to the latter being larger than $C_\alpha$. When the maximum is taken in $\sqrt{n \sqrt{d_L} \tau}$, using that $4\sigma^{-2} n\|f\|_2^2/\sqrt{d_L} \geq \tau$ yields that the above display is bounded from below by
    \begin{equation*}
    \frac{\sqrt{m} \sigma^{-1} n^{} \| f_L \|_2 \epsilon }{ \tilde{\kappa}_\alpha \sqrt{\log(N) \sqrt{2 \mathfrak{c} \log(2/\delta)}}} \geq C_\alpha.
    \end{equation*}
    In either case, it follows that the Type II error (i.e. \eqref{eq:type_II_error_after_claim_lips_ext_test_proof}) can be made arbitrarily small per large enough choice of $C_\alpha > 0$.
    
    We return to the claim of \eqref{eq:mention_this_holds_later_lips_ext_proof}. When $\gamma_\tau \geq 1$, the claim is satisfied whenever $\kappa \lesssim \log(N)\sqrt{\mathfrak{c}\log(1/\delta)} $ for $C_\alpha > 0$ large enough. When $\gamma_\tau < 1$, it is required that
    \begin{equation*}
        \frac{\sqrt{m} \sigma^{-2} n^{} \gamma_\tau  \| f_L \|_2^2 \epsilon}{\sqrt{d_L}} = \frac{\sqrt{m} \sigma^{-2} n^{2} \| f_L \|_2^2 \epsilon}{ \tilde{\kappa}_\alpha \log(N) \sqrt{2 \mathfrak{c} \log(2/\delta)} (\sqrt{ n \sqrt{d_L} \tau} \vee \sqrt{nd_L})} \gtrsim \kappa,
    \end{equation*}
    which is satisfied whenever \eqref{eq:supp:condition_on_n_f_L_procedure_I} holds.
    \end{proof}
    
    \subsubsection{Proof of Lemma \ref{lem:lips_test_up_rate}}\label{sssec:supp:proof_test_I}
    
    We start by recalling the notation, before proving a slightly more general result. Consider for a set $\cS \subset \N$ the test
    \begin{equation}\label{eq:supp:definition_full_test_lips_ext_privacy}
    T_{\text{I}} := \underset{L \in \cS,\, \tau \in \mathrm{T}_L}{\max} \mathbbm{1} \left\{  \frac{1}{\sqrt{m}} \underset{j=1}{\overset{m}{\sum}} Y^{(j)}_{L;\tau} \geq \kappa_\alpha \left( \frac{\epsilon}{D_\tau \sqrt{2 |\mathrm{T}_L| |\cS| \log (2 / \delta)}} \vee 1 \right) \sqrt{\log |\mathrm{T}_L||\cS|}   \right\},
    \end{equation}
    where
    \begin{equation}\label{eq:supp:collection_of_threshold_lips_test}
    \mathrm{T}_L := \left\{ 2^{-k+2} \frac{n R^2}{\sqrt{2^{L+1}}}  : k=1,\dots,\lceil1+2\log_2(N R / \sigma )\rceil \right\},
    \end{equation}
    and $Y^{(j)}_L = \{ Y^{(j)}_{L;\tau} : \tau \in \mathrm{T}_L \}$ is generated according to \eqref{eq:supp:transcript_private_Euclidean_norm_lips_ext_for_one_clipping} with
    \begin{equation*}
    \gamma_\tau = \frac{\epsilon}{D_\tau  \sqrt{2 |\mathrm{T}_L| |\cS| \log (2 / \delta)}}.
    \end{equation*}
    By the same reasoning as in the proof of Lemma \ref{lem:typeI_typeII_error_control_Lipschitz_ext_test_privacy_more_general}, the test $T_{\text{I}}$ is $(\epsilon,\delta)$-differentially private. We will prove the following more general version of Lemma \ref{lem:lips_test_up_rate}.
    
    \begin{lemma}\label{lem:supp:lips_test_up_rate}
    For all $M>0$, $\alpha \in (0,1)$ there exists $\kappa_\alpha > 0$ and $C_\alpha > 0$ such that the test defined by \eqref{eq:supp:definition_full_test_lips_ext_privacy} satisfies $\P_0 T_{\text{I}}  \leq \alpha$.
    
    Furthermore, for any $f \in \cB^{{s},R}_{p,q}$ such that for some $L^* \in \cC$ it holds that
    \begin{equation}\label{eq:sup:tau_specific_minimal_signal_size}
     \|f_{L^*} \|_2^2 \geq C_\alpha M_N \left( \frac{\sqrt{2^{L^*}}}{\sqrt{m}n (\sqrt{n} \epsilon \wedge 1)} \right) \bigvee \left( \frac{1}{mn^2\epsilon^2} \right),
    \end{equation}
    where 
    \begin{equation*}
        M_N \gtrsim \log(N)  \sqrt{2 \log \log(N R/\sigma) \log(N R/\sigma) |\cS|  \log (2  / \delta)},
    \end{equation*}
    it holds that $\P_f (1 - T_{\text{I}}) \leq \alpha$ for $C_\alpha > 0$ large enough depending only on $\alpha$.
    \end{lemma}  
    \begin{proof} 
    We start by proving that the level of the test is controlled. Using a union bound, it holds that
    \begin{align}
    \P_0 T_{\text{I}} &= \P_0 \left( \underset{L \in \cS,\;\tau \in \mathrm{T}_L}{\max} \, \frac{1}{\left( \gamma_\tau \vee 1 \right) \sqrt{\log |\mathrm{T}||\cS|} \sqrt{m}} \underset{j=1}{\overset{m}{\sum}} [\gamma_L \breve{S}_{L;\tau}^{(j)}(X^{(j)}) + W^{(j)}_\tau] \geq \kappa_\alpha \left( \gamma_\tau \vee 1 \right) \sqrt{\log |\mathrm{T}_L||\cS|} \right) \nonumber \\
    &\leq \P_0 \left( \underset{L \in \cS,\,\tau \in \mathrm{T}_L}{\max} \;\frac{1}{\left( \gamma_\tau \vee 1 \right) \sqrt{\log |\mathrm{T}||\cS|} \sqrt{m}} \underset{j=1}{\overset{m}{\sum}} [\gamma_L \tilde{S}_{L;\tau}^{(j)}(X^{(j)}) + W^{(j)}_\tau] \geq \kappa_\alpha  \right) \label{eq:typeI_error_lips_test_bound} \\ & \quad \quad \quad + \P_0 \left( \exists j \in [m], \, L \in \cS, \, \tau \in \mathrm{T}_L \, : \,  X^{(j)} \notin \cC_{L;\tau}   \right) \nonumber
    \end{align}
    where it is used that the Lipschitz extension $\breve{S}_{L;\tau}^{(j)}$ of $\tilde{S}_{L;\tau}^{(j)}$ equals the latter function on $\cC_{L;\tau}$. By Lemma \ref{lem:concentration_on_C} and a union bound, the second probability on the right-hand side is bounded above by $\alpha m |\cS| |\mathrm{T}_L| / (2N)$, for a large enough choice of $\tilde{\kappa}_\alpha > 0$. 
    
    Considering the first probability on the right-hand side of the inequality displayed above, we first note that by Lemma \ref{lem:clipping_preserves_tails_exponential},
    \begin{align*}
    \tilde{S}_{L;\tau}^{(j)}(X^{(j)})  &= \left[\frac{1}{\sqrt{d_L}} \left(\left\| \sqrt{n} \overline{X^{(j)}_{L}} \right\|^2_2 - {V^{(j)}_{L;\tau}} \right)\right]_{-\tau}^{\tau}  
    \end{align*}
    is $1/(4\sqrt{d})$-sub-exponential, where the sub-exponentiality parameter follows from a straightforward calculation (see e.g. Lemma \cite{szabo2022optimal_IEEE}). Since $W^{(j)}_{L;\tau}$ is independent $N(0,1)$, it follows by Bernstein's inequality (see e.g. Theorem 2.8.1 in \cite{vershynin_high-dimensional_2018}) and a union bound (e.g. Lemma \ref{lem:iterated_logarithm_max_bound}) that the first probability in \eqref{eq:typeI_error_lips_test_bound} is bounded as follows
    \begin{align*}
    \P_0 \left(  \underset{L \in \cS,\,\tau \in \mathrm{T}_L}{\max} \;\frac{1}{\left( \gamma_\tau \vee 1 \right) \sqrt{\log |\mathrm{T}||\cS|} \sqrt{m}} \underset{j=1}{\overset{m}{\sum}} [\gamma_L \tilde{S}_{L;\tau}^{(j)}(X^{(j)}) + W^{(j)}_\tau] \geq \kappa_\alpha \right) &\leq \sum_{L \in \cS,\,\tau \in \mathrm{T}_L}   \frac{1}{( |\mathrm{T}_L||\cS|)^{\kappa_\alpha^2 /2}},
    \end{align*}
    which is less than $\alpha/2$ for a large enough choice of $\kappa_\alpha > 0$. For $f \in \cB^{{s},R}_{p,q}$, Lemma \ref{lem:L_2_to_besov_bound_on_f} yields that
    \begin{align}\label{eq:supp:bound_on_f_L_2_norm}
    \|f_L\|_2 \leq \|f\|_2 \leq  \frac{1}{(1-2^{-s})^{1-1/q}} \|f\|_{s,p,q} \leq \frac{R}{(1-2^{-s})^{1-1/q}}.
    \end{align}
    Consequently, when $f$ satisfies \eqref{eq:lem:tau_specific_minimal_signal_size}, it holds in particular that $\|f_L\|_2^2 \geq N^{-1}$. Thus, there exists $\tau^* \in \mathrm{T}_{L^*}$ such that the condition of Lemma \ref{lem:typeI_typeII_error_control_Lipschitz_ext_test_privacy_more_general} is satisfied, which now yields that
    \begin{equation*}
    \P_f (1 - T_{\text{I}}) \leq \P_f (1 - \varphi_{\tau^*}) \leq \alpha/2.
    \end{equation*}  
    \end{proof}

    \subsection{Procedure II}\label{eq:procedure_I_priv_coin}\label{ssec:supp:procedure_II}
    
    Let $\epsilon,\delta > 0$ and $\cS \subset \N$ be given. Consider for $L \in \cS$, $K_L =  \lceil n \epsilon^2 \wedge d_L \rceil$ and take sets $\cJ_{lk;L} \subset [m]$ such that $| \cJ_{lk;L} | = \lceil \frac{m K_L}{d_L} \rceil$ and each $j \in \{1,\dots,m\}$ is in $\cJ_{lk;L}$ for $K_L$ different indexes $k \in \{1,\dots,d_L\}$. For $(l,k) \in \{ l=1,\dots, L, k = 1,\dots,2^l \} =:I_L$, $j \in \cJ_{lk;L}$, generate the transcripts according to
    \begin{equation}\label{eq:transcript_impure_DP_coordinate_wise}
    Y^{(j)}_{L;lk}|X^{(j)} = \gamma_L \, \underset{i=1}{\overset{n}{\sum}}  \sigma^{-1} \left[  (X_{L;i}^{(j)})_{lk} \right]_{-\tau}^\tau + W^{(j)}_{L;lk}
    \end{equation}
    with $\gamma_L = \frac{\epsilon}{2\sqrt{ |\cS| K_L \log(2/\delta)\tau}}$, $\tau = \tilde{\kappa}_\alpha \sqrt{\log(N)}$ and $(W^{(j)}_{L;lk})_{j,k}$ i.i.d. standard Gaussian noise. 
    
    Define $Y^{(j)}_L : = (Y^{(j)}_{L;lk})_{(l,k) \in I_L:j \in \cJ_{lk;L}}$ and consider the transcripts $Y^{(j)} : = (Y^{(j)}_L)_{L \in \cS}$, for $j=1,\ldots,m$. The lemma below shows that $Y^{(j)}$ satisfies DP.   
    
    \begin{lemma}\label{lem:L_2-sensitivity_private_coin_bound}
    The transcript defined $Y^{(j)}$ is $(\epsilon,\delta)$-differentially private.
    \end{lemma}
    \begin{proof}
        The rescaled and clipped sums have at most $L_2$-sensitivity less than or equal to one:
        \begin{align}\label{eq:L_2-sensitivity_private_coin_bound}
            \underset{\breve{x} \in (\R^\N)^n : \mathrm{d}_H(x,\breve{x}) \leq 1}{\sup} \, \gamma_L \left\|  \left( \underset{i=1}{\overset{n}{\sum}} \sigma^{-1} [  (x_i)_{lk} ]_{-\tau}^\tau - \underset{i=1}{\overset{n}{\sum}} \sigma^{-1} [  (\breve{x}_i)_{lk} ]_{-\tau}^\tau\right)_{(l,k) \in I_L, L \in \cS} \right\|_2 \leq \\ \gamma_L \sqrt{ \sum_{L \in \cS} \underset{(l,k) \in I_L}{\overset{}{\sum}} \left( \underset{\breve{x} \in (\R^\N)^n : \mathrm{d}_H(x,\breve{x}) \leq 1}{\sup}\sigma^{-1} [  (x_i)_{lk} ]_{-\tau}^\tau - \sigma^{-1} [  (\breve{x}_i)_{lk} ]_{-\tau}^\tau \right)^2  } \leq 1. \nonumber
            \end{align}
            Consequently, the addition of the Gaussian noise assures that the transcript $Y^{(j)}$ is $(\epsilon,\delta)$-differentially private (see e.g. Appendix A in \cite{dwork2014algorithmic}). 
    \end{proof}
    
    Consider the test given by
    \begin{equation}\label{eq:supp:test_impure_DP_coordinate_wise_strat}
    T_{\text{II}} := \mathbbm{1}\left\{ \underset{L \in \cS}{\max} \, \frac{1}{\sqrt{d_L} \left( n \gamma_L^2 \vee 1\right)} \underset{(l,k) \in I_L}{\overset{}{\sum}} \left[ \left( \frac{1}{\sqrt{|\cJ_{lk;L}|}}  \underset{j \in \cJ_{lk;L}}{\overset{}{\sum}} Y^{(j)}_{L;lk}  \right)^2 - \nu_{\epsilon,L}  \right] \geq \kappa_\alpha \log \log (e\vee|\cS|) \right\},
    \end{equation}
    with $\nu_{\epsilon,L} := \E_0 ( |\cJ_{lk;L}|^{-1/2} \sum_{j \in \cJ_{lk;L}} Y^{(j)}_{L;lk}  )^2$.
    
    \begin{lemma}\label{lem:protocol_II_priv_coin_bound}
    For all $M>0$, $\alpha \in (0,1)$ there exists $\kappa_\alpha > 0$ and $C_\alpha > 0$ such that the test defined by \eqref{eq:supp:test_impure_DP_coordinate_wise_strat} satisfies $\P_0 T_{\text{II}}  \leq \alpha$. Furthermore, for any $f \in \cB^{{s},R}_{p,q}$ such that
    \begin{equation}\label{eq:sup:lem:tau_specific_minimal_signal_size}
     \|f_{L^*} \|_2^2 \geq C_\alpha \log \log (e\vee|\cS|) \log(N) |\cS| \log(1/\delta) \left( \frac{\sqrt{2^{(3/2)L^*}}}{\sqrt{m}n ({n} \epsilon^2 \wedge d_{L^*})} \right) \bigvee \left( \frac{1}{mn^2\epsilon^2} \right),
    \end{equation}
    for some $L^* \in \cS$ and $C_\alpha > 0$ large enough depending only on $\alpha$, it holds that $\P_f (1 - T_{\text{II}}) \leq \alpha$.
    \end{lemma}  
    \begin{proof}
    Under the null hypothesis, $(X_{L;i}^{(j)})_{lk}$ are independent standard Gaussian for $(l,k) \in I_L$, $j \in [m]$. Hence, by Lemma \ref{lem:clipping_preserves_tails_exponential}, the random variables $Y^{(j)}_{L;lk}$ are sub-gaussian, mean zero and i.i.d. for $k \in [d_L]$, $j \in [m]$ under the null hypothesis. More specifically, we have that
    \begin{equation*}
     \frac{1}{\sqrt{|\cJ_{lk;L}|}}  \underset{j \in \cJ_{lk;L}}{\overset{}{\sum}} Y^{(j)}_{L;lk} \quad \text{ is } \quad \left(\frac{\sqrt{n}\epsilon}{2\sqrt{ |\cS| K_L \log(2/\delta) \tau}} + 1\right)\text{-sub-gaussian},
    \end{equation*}
    which in turn implies that the random variables 
    \begin{equation*}
    \frac{1}{\sqrt{d_L} \left( n\gamma_L^2 \vee 1\right)} \underset{(l,k) \in I_L}{\overset{}{\sum}} \left[ \left( \frac{1}{\sqrt{|\cJ_{lk;L}|}}  \underset{j \in \cJ_{lk;L}}{\overset{}{\sum}} Y^{(j)}_{L;lk}  \right)^2 - \nu_{\epsilon,L}  \right]
    \end{equation*}
    are $C$-sub-exponential for $L \in \cS$, for some constant $C > 0$ (see e.g. Section 2.7 in \cite{vershynin_high-dimensional_2018}). It now follows from Lemma \ref{lem:iterated_logarithm_max_bound} that, for $\kappa_\alpha > 0$ large enough, $\P_0 T_{\text{II}}\leq \alpha / 2$.
    
    We now turn our attention to the Type II error. Assume now that, for $L \in \cS$, $f$ satisfies \eqref{eq:sup:lem:tau_specific_minimal_signal_size}. We have that
    \begin{equation*}\label{eq:test_impure_DP_coordinate_wise_strat}
       1 -  T_{\text{II}} \leq \mathbbm{1}\left\{ \frac{1}{\sqrt{d_L} \left( n \gamma_L^2 \vee 1\right)} \underset{(l,k) \in I_L}{\overset{}{\sum}} \left[ \left( \frac{1}{\sqrt{|\cJ_{lk;L}|}}  \underset{j \in \cJ_{lk;L}}{\overset{}{\sum}} Y^{(j)}_{L;lk}  \right)^2 - \nu_{\epsilon,L}  \right] < \kappa_\alpha \log \log (e\vee|\cS| \right\}),
        \end{equation*}
    so it suffices to bound the $\E_f$-expectation of the right-hand side. 

    Under $\P_f$, $(X_i^{(j)})_{lk} \overset{d}{=} f_{lk} + \sigma^{-1} Z_{lk;i}^{(j)}$ with i.i.d. $Z_{lk;i}^{(j)} \sim N(0,1)$ and is independent of the centered i.i.d. $W^{(j)}_{lk}$. Hence, the quantity
    \begin{equation*}
    V_{lk} := \left( \frac{1}{\sqrt{|\cJ_{lk;L}|}}  \underset{j \in \cJ_{lk;L}}{\overset{}{\sum}} \left[ \gamma_L \underset{i=1}{\overset{n}{\sum}}  \sigma^{-1} (X_i^{(j)})_{lk}  + W^{(j)}_{lk} \right] \right)^2
    \end{equation*}
    is in distribution equal to
    \begin{equation*}
    \left( \gamma_L \sqrt{|\cJ_{lk;L}|} n \sigma^{-1}  f_{lk} + \gamma_L \sqrt{n} \eta +\frac{1}{\sqrt{|\cJ_{lk;L}|}}  \underset{j \in \cJ_{lk;L}}{\overset{}{\sum}}  W^{(j)}_{lk} \right)^2
    \end{equation*}
    under $\P_f$, with $\eta \sim N(0,1)$ independent. Therefore, a straightforward calculation shows that $V_{lk}$ has mean $ \sigma^{-1} \gamma_L^2 n^2 |\cJ_{lk;L}| f_{lk}^2 +  n \gamma_L^2 + \E (W^{(j)}_{lk})^2$ under $\P_f$. Since $\E (Z_{lk;i}^{(j)})^4 = 3$ and $\E (W^{(j)}_{lk})^4 \asymp 1$, its variance equals
    \begin{align}
     n^2 \gamma_L^4  \text{Var}\left( \eta^2 \right)  + \text{Var}\left( \left(\frac{1}{\sqrt{|\cJ_{lk;L}|}}  \underset{j \in \cJ_{lk;L}}{\overset{}{\sum}}  W^{(j)}_{lk} \right)^2 \right) + \gamma_L^4 |\cJ_{lk;L}| n^{3} \sigma^{-2}  f_{lk}^2 \E \eta^2 \\
     + \gamma_L^2 |\cJ_{lk;L}| n^2 \sigma^{-2} f_{lk}^2 \E (W^{(j)}_l)^2 + n \gamma_L^2 \E (W^{(j)}_{lk})^2  \E \eta^2, \nonumber
    \end{align}
    which is of the order 
    \begin{equation}\label{eq:variance_of_coordinate_wise_privacy_transcript_under_Pf}
     (\gamma_L^4 |\cJ_{lk;L}| n^{3} \sigma^{-2} f_{lk}^2) \vee (\gamma_L^2 |\cJ_{lk;L}| n^2 \sigma^{-2} f_{lk}^2) \vee \gamma_L^4 n^2 \vee 1.
     \end{equation}
    Since \eqref{eq:supp:bound_on_f_L_2_norm} holds, we have that for $\tilde{\kappa}_\alpha$ large enough, 
    \begin{equation}\label{eq:privacy_clipping_general_coordinate_based_signal_under_clipping}
    \underset{((l,k) \in I_L)}{\max} |f_{lk}| \leq \tau/2.
    \end{equation}
    Consequently, an application of the triangle inequality and Lemma \ref{lem:gaussian_maximum} yield that, for $\tilde{\kappa}_\alpha > 0$ large enough, we have with probability at least $1 - 2 N d_l e^{-\tau^2/4} \geq 1 - (1+N)^{2 - \tilde{\kappa}_\alpha^2/4} \geq 1 - \alpha/4$ that
    \begin{equation*}
    \underset{i \in [n], j \in [m], (l,k) \in I_L}{\max} |(X_i^{(j)})_{lk}| \leq \tau.
    \end{equation*}
    Consequently, under the null hypothesis ($f=0$), using that $|\cJ_{lk;L}| = |\cJ_1|$, $\P_0 \varphi$ is bounded above by
    \begin{equation*}
    \P_0 \left( \frac{1}{\sqrt{d_L}|\cJ_{1d_L;L}|} \underset{(l,k) \in I_L}{\overset{}{\sum}} \left[ \left(   \underset{j \in \cJ_{lk;L}}{\overset{}{\sum}}  \left(\gamma_L \sigma^{-1} (n\bar{X^{(j)}})_{lk} + W^{(j)}_{lk}\right)  \right)^2 - \E_0 V_l \right] \geq \kappa (\gamma_L^2 n \vee 1) \right) + \frac{ \alpha}{4},
    \end{equation*}
    where we write $\kappa = \kappa_\alpha \log \log (e\vee|\cS|)$. Chebyshev's inequality yields that the first term on the left-hand side is bounded $\alpha/4$ for $\kappa_\alpha > 0$ large enough. In the case that \eqref{eq:privacy_clipping_general_coordinate_based_signal_under_clipping} holds, we also have that $\P_f (1 - T_{\text{II}})$ is bounded above by
    \begin{equation*}
    \text{Pr} \left( \frac{1}{\sqrt{d_L}|\cJ_{1d_L;L}|} \underset{(l,k) \in I_L}{\overset{}{\sum}} \left[ \left(   \underset{j \in \cJ_{lk;L}}{\overset{}{\sum}}  \left(\gamma_L (n f_{lk} + \sqrt{n} Z) + W^{(j)}_{lk}\right)  \right)^2 - \E_0 V_l \right] < \kappa (\gamma_L^2 n \vee 1) \right) + \frac{ \alpha}{4}.
    \end{equation*}
    Subtracting $d^{-1/2}_L \underset{(l,k) \in I_L}{\overset{}{\sum}} \gamma_L^2 n^2 |\cJ_{lk;L}| \sigma^{-1} f_{lk}^2$ on both sides, the first term is bounded above by
    \begin{align*}
    \text{Pr} \bigg( \frac{1}{\sqrt{d_L}|\cJ_{1d_L;L}|} \underset{(l,k) \in I_L}{\overset{}{\sum}} \left[ \left(   \underset{j \in \cJ_{lk;L}}{\overset{}{\sum}}  \left(\gamma_L (n \sigma^{-1} f_{lk} + \sqrt{n} Z) + W^{(j)}_l\right)  \right)^2 - \E_f V_l \right] &< \\ 
    - \frac{ \gamma_L^2 n^2 |\cJ_{1d_L;L}| \sigma^{-2} \| f \|_2^2}{2\sqrt{d_L}} \bigg)&
    \end{align*}
    whenever
    \begin{equation*}
    \sigma^{-2} \frac{\gamma_L^2 n^2 |\cJ_{lk;L}|}{\sqrt{d_L}} \|f_L\|_2^2 \geq 2 \kappa (\gamma_L^2 n \vee 1) \iff \frac{ 2 \sigma^2 \kappa_\alpha \log \log (e\vee|\cS| (\gamma_L^2 n \vee 1) d_L  \sqrt{d_L}}{\gamma_L^2) n^2 m K_L \|f_L\|_2^2} \leq 1.
    \end{equation*}
    This follows from the assumed inequality \eqref{eq:sup:lem:tau_specific_minimal_signal_size} in the lemma's statement.
    
    An application of Chebyshev's inequality and the variance bound computed in \eqref{eq:variance_of_coordinate_wise_privacy_transcript_under_Pf}, now yields that the probability in the second last display is of the order
    \begin{equation*}
     \frac{(d^{-1}_L \gamma_L^4 |\cJ_{lk;L}| n^3 \sigma^{-2} \|f_L\|_2^2) \vee (d^{-1}_L \gamma_L^2 |\cJ_{lk;L}| n^2 \sigma^{-2} \|f_L\|_2^2) \vee \gamma_L^4 n^2 \vee 1}{ \left( \frac{\gamma_L^2 n^2 |\cJ_{lk;L}|}{\sqrt{d_L}} \sigma^{-2} \|f_L\|_2^2 \right)^2}.
    \end{equation*}
    Plugging in $|\cJ_{lk;L}| \asymp mK_L/d_L$, $\gamma_L = \frac{\epsilon}{2\sqrt{ |\cS| K_L \log(2/\delta)\tau}}$ and $\tau = \tilde{\kappa}_\alpha \sqrt{\log(N)}$, we obtain that the latter display equals
    \begin{equation}\label{eq:supp:protocol_III_4max_condition}
     \frac{( \frac{m\epsilon^4}{16 K_L d_L \Lambda_{N,\delta,|\cS|}^2} n^3 \sigma^{-2} \|f_L\|_2^2) \vee ( \frac{m\epsilon^2}{4  K_L d_L \Lambda_{N,\delta,|\cS|}} n^2 \sigma^{-2} \|f_L\|_2^2) \vee ( \frac{\epsilon^4}{16 K_L \Lambda_{N,\delta,|\cS|}^2} ) n^2 \vee 1}{ \left( \frac{m\epsilon^2 n^2}{4  K_L d_L \Lambda_{N,\delta,|\cS|} \sqrt{d_L}} \sigma^{-2} \|f_L\|_2^2 \right)^2}
    \end{equation}
    with $\Lambda_{N,\delta,|\cS|} = |\cS| \log(2/\delta) \tilde{\kappa}_\alpha^2 {\log(N)}$. For $1/\sqrt{n} \lesssim \epsilon$ and $n \epsilon^2 \lesssim d_L$, this expression is of the order
    \begin{equation*}
        \frac{\Lambda_{N,\delta,|\cS|}^2 d^3_L}{ m^2 n^4 \epsilon^4 \| f_L \|_2^4 },
    \end{equation*}
    which is bounded by ${1}/{C_\alpha^2}$ when $f$ satisfies \eqref{eq:sup:lem:tau_specific_minimal_signal_size}. This can be seen to be arbitrarily small whenever $f$ is such that \eqref{eq:sup:lem:tau_specific_minimal_signal_size} by taking $C_\alpha > 0$ large enough. Lastly, when $n \epsilon^2 \gtrsim d_L$, \eqref{eq:supp:protocol_III_4max_condition} is of the order
    \begin{equation}
        \frac{ \kappa_\alpha^2 \Lambda_{N,\delta,|\cS|} d_L}{m^2 n^2  \|f_L\|_2^4},
    \end{equation}
    which also holds for $C_\alpha > 0$ large enough for $f$ satisfying \eqref{eq:sup:lem:tau_specific_minimal_signal_size}. Consequently, we obtain that $\P_f (1 - T_{\text{II}}^{\epsilon,\delta}) \leq \alpha/2$ for $C_\alpha$ large enough, as desired. 
    
    \end{proof}

    \subsection{Procedure III}\label{ssec:supp:procedure_III}
    
    Let $\cS \subset \N$ and consider for $L \in \N$, $d_L := \sum_{l=1}^L 2^l$, $K_L := \lceil n \epsilon^2 \wedge 2^L \rceil$, $I_L:= \{ (l,k) : l=1,\dots, \lceil \log_2(K_L) \rceil, \, k=1,\dots,2^l \}$ and $j=1,\dots,m$ the transcripts 
    \begin{equation}\label{eq:transcripts_pub_coin_privacy_impure_DP}
     Y^{(j)}_{lk} | (X^{(j)},U_L) =  \gamma_L \, \underset{i=1}{\overset{n}{\sum}}  [ (U_L X_{L;i}^{(j)})_{lk} ]^\tau_{-\tau} + W^{(j)}_{lk},
     \end{equation}
    with $\gamma_L = \frac{\epsilon}{2\sqrt{K_L |\cS| \log(2 / \delta)\tau}}$, $\tau = \tilde{\kappa}_\alpha \sqrt{\log(N)}$, $U_L$ a $d_L$-by-$d_L$ random rotation drawn uniformly (i.e. from the Haar measure on the special orthogonal group in $\R^{d_L \times d_L}$) and $(W^{(j)}_{lk})_{j,l,k}$ i.i.d. centered standard Gaussian noise. The following lemma shows that the transcript is $(\epsilon,\delta)$-differentially private.
    
    \begin{lemma}\label{lem:transcript_privacy_impure_DP_shared_randomness}
    The transcript defined in \eqref{eq:transcripts_pub_coin_privacy_impure_DP} is $(\epsilon,\delta)$-differentially private.
    \end{lemma}
    \begin{proof}
    The argument is similar to that of Lemma \ref{lem:L_2-sensitivity_private_coin_bound}, noting that rescaled and clipped sums have at most $L_2$-sensitivity less than or equal to one for any rotation $U_L$ in $ \R^{d_L \times d_L}$;
    \begin{align*}
        \underset{\breve{x} \in (\R^\N)^n : \mathrm{d}_H(x,\breve{x}) \leq 1}{\sup} \, \gamma_L \left\|  \left( \underset{i=1}{\overset{n}{\sum}} \sigma^{-1} [ U_L (x_i)_{lk} ]_{-\tau}^\tau - \underset{i=1}{\overset{n}{\sum}} \sigma^{-1} [ U_L (\breve{x}_i)_{lk} ]_{-\tau}^\tau\right)_{(l,k) \in I_L, L \in \cS} \right\|_2 \leq \\ \gamma_L \sqrt{ \sum_{L \in \cS} \underset{(l,k) \in I_L}{\overset{}{\sum}} \left( \underset{\breve{x} \in (\R^\N)^n : \mathrm{d}_H(x,\breve{x}) \leq 1}{\sup}\sigma^{-1} [ U_L (x_i)_{lk} ]_{-\tau}^\tau - \sigma^{-1} [ U_L (\breve{x}_i)_{lk} ]_{-\tau}^\tau \right)^2  } \leq 1. \nonumber
        \end{align*}
        so by the same reasoning as in Lemma \ref{lem:L_2-sensitivity_private_coin_bound}, the transcript $Y^{(j)}$ is $(\epsilon,\delta)$-differentially private. 
    \end{proof}
    
    Based on these transcripts, one can compute
    \begin{equation}\label{eq:test_pub_coin_privacy_impure_DP}
    T_{\text{III}} = \mathbbm{1}\left\{ \underset{L \in \cL}{\max} \; \frac{1}{\sqrt{K_L}\left(n \gamma_L^2 \vee 1\right)} \underset{(l,k) \in I_L}{\overset{}{\sum}} \left[ \left( \frac{1}{\sqrt{m}}  \underset{j =1}{\overset{m}{\sum}} Y^{(j)}_{lk}  \right)^2 -  \nu_{L} \right] \geq \kappa_\alpha  \log \log (e\vee|\cS|) \right\},
    \end{equation}
    where $ \nu_{L} = n\gamma_L^2 + 1$. The following lemma establishes the Type I and Type II error probability guarantees for the above test.
    
    \begin{lemma}\label{lem:procedure_III_test_up_rate}
    For all $s,R>0$, $\alpha \in (0,1)$ there exists $\kappa_\alpha, \tilde{\kappa}_\alpha > 0$ and $C_\alpha > 0$ such that the test $T_{\text{III}}$ defined by \eqref{eq:test_pub_coin_privacy_impure_DP} satisfies $\P_0 T_{\text{III}}  \leq \alpha$. Furthermore, for any $f \in \cB^{{s},R}_{p,q}$ such that
    \begin{equation}\label{eq:lem:sup:procedure_III_test_up_rate}
     \|f_{L^*} \|_2^2 \geq C_\alpha \log \log (e\vee|\cS|) \log^{} (N) |\cS| \log(1/\delta) \left( \frac{2^{L^*}}{\sqrt{m}n (\sqrt{n} \epsilon \wedge 1)} \right) \bigvee \left( \frac{1}{mn^2\epsilon^2} \right),
    \end{equation}
    for some $L^* \in \cS$ and $C_\alpha > 0$ large enough depending only on $\alpha$, it holds that $\P_f (1 - T_{\text{II}}^{\epsilon,\delta}) \leq \alpha$.
    \end{lemma}  
    \begin{proof}
    Under $\P_f$, $(U_L  X_{L;i}^{(j)})_{lk} \overset{d}{=} (U_L f)_{lk} + (U_L  Z_{L;i}^{(j)})_{lk}$, where $Z_{L;i}^{(j)} \sim N(0,I_{d_L})$ independent of $U_L$ and the centered i.i.d. $W^{(j)}_{lk}$. Furthermore, $U_L Z_{L;i}^{(j)} \overset{d}{=} Z_{L;i}^{(j)}$, $(U_L  Z_{L;i}^{(j)})_{lk} \sim N(0,1)$, still independent of $(W^{(j)}_{lk})_{l,k,j}$. We obtain that
    \begin{equation*}
    \gamma_L \underset{i=1}{\overset{n}{\sum}}  (U_L  X_{L;i}^{(j)})_{lk}  + W^{(j)}_{lk} \overset{d}{=} \gamma_L n (U_L f_L)_{lk} + \gamma_L \sqrt{n} \eta + W^{(j)}_{lk},
    \end{equation*}
    with $\eta \sim N(0,1)$ and all three terms independent. 
    
    A first consequence of the relationship above is that under the null hypothesis, the random variables $Y^{(j)}_{L;k}$ are sub-gaussian, mean zero and i.i.d. for $(l,k) \in I_L$, $j \in [m]$ under the null hypothesis (see Lemma \ref{lem:clipping_preserves_tails_exponential}). More specifically, we have that
    \begin{equation*}
     \frac{1}{\sqrt{m}}  \underset{j = 1}{\overset{m}{\sum}} Y^{(j)}_{lk} \quad \text{ is } \quad \left(\frac{\sqrt{n}\epsilon}{2\sqrt{ |\cS| {K_L} \log(2/\delta)\tau}} + 1\right)\text{-sub-gaussian},
    \end{equation*}
    which in turn implies that the random variables 
    \begin{equation*}
        \frac{1}{\sqrt{K_L}} \underset{(l,k) \in I_L}{\overset{}{\sum}} \left[ \left( \frac{1}{\sqrt{m}}  \underset{j = 1}{\overset{m}{\sum}} Y^{(j)}_{L;k}  \right)^2 - \nu_{L}  \right]
    \end{equation*}
    are $C$-sub-exponential for $L \in \cS$, for some constant $C > 0$ (see e.g. Section 2.7 in \cite{vershynin_high-dimensional_2018}). It now follows from Lemma \ref{lem:iterated_logarithm_max_bound} that, for $\kappa_\alpha > 0$ large enough, $\P_0 T_{\text{III}} \leq \alpha / 2$.
    
    Next, we turn to the Type II error probability. Assume now that, for $L \in \cS$, $f$ satisfies \eqref{eq:sup:lem:tau_specific_minimal_signal_size}. We have that
    \begin{equation*}\label{eq:test_impure_DP_coordinate_wise_strat}
       1 -  T_{\text{III}} \leq \mathbbm{1}\left\{  \frac{1}{\sqrt{K_L}} \underset{(l,k) \in I_L}{\overset{}{\sum}} \left[ \left( \frac{1}{\sqrt{m}}  \underset{j =1}{\overset{m}{\sum}} Y^{(j)}_{lk}  \right)^2 -  \nu_{L} \right] < \kappa_\alpha \left(n \gamma_L^2 \vee 1\right) \log \log (e\vee|\cS|) \right\},
        \end{equation*}
    so it suffices to bound the $\E_f$-expectation of the right-hand side. Since \eqref{eq:supp:bound_on_f_L_2_norm} holds, we have that for $\tilde{\kappa}_\alpha$ large enough, 
    \begin{equation}\label{eq:privacy_clipping_general_coordinate_based_signal_under_clipping}
    \underset{((l,k) \in I_L)}{\max} |f_{lk}| \leq \tau/2.
    \end{equation}
    Consequently, an application of the triangle inequality and Lemma \ref{lem:gaussian_maximum} yield that, for $\tilde{\kappa}_\alpha > 0$ large enough, we have with probability at least $1 - 2 N d_l e^{-\tau^2/4} \geq 1 - (1+N)^{2 - \tilde{\kappa}_\alpha^2/4} \geq 1 - \alpha/4$ that
    \begin{equation*}
    \underset{i \in [n], j \in [m], (l,k) \in I_L}{\max} |(X_i^{(j)})_{lk}| \leq \tau.
    \end{equation*}
    Thus, it suffices to show that
    \begin{equation}\label{eq:supp:type_II_error_bound_impure_DP}
        \P_f\left(  \frac{1}{\sqrt{K_L}} \underset{(l,k) \in I_L}{\overset{}{\sum}} \left[ \left( \frac{1}{\sqrt{m}}  \underset{j =1}{\overset{m}{\sum}} Y^{(j)}_{lk}  \right)^2 -  \nu_{L} \right] < \kappa_\alpha \left(n \gamma_L^2 \vee 1\right) \log \log (e\vee|\cS|) \right) \leq \alpha/4,
    \end{equation}
    for $C_\alpha > 0$ large enough. Define
    \begin{equation*}
        V_{lk} := \left( \frac{1}{\sqrt{m}}  \underset{j = 1}{\overset{m}{\sum}} \left[ \gamma_L \underset{i=1}{\overset{n}{\sum}}  (U_L X_i^{(j)})_{lk}  + W^{(j)}_{lk} \right] \right)^2.
        \end{equation*}
    Under $\P^U$, we have that
    \begin{equation*}
    (U_Lf_L)_{lk} \overset{d}{=} \|f_L\|_2 \frac{Z_{lk}}{\|Z \|_2},
    \end{equation*}
    for $Z=(Z_{11},\dots,Z_{1d_L}) \sim N(0,I_{d_L})$. As $\sum_{(l,k) \in I_L} \E {Z_{lk}^2}/{\|Z \|_2^2} = 1$, $\E {Z_{lk}^2}/{\|Z \|_2^2} = 1/d_L$ by symmetry. Hence,
    \begin{equation*}
    \E_f  V_{lk} = \gamma_L^2 m n^2 \E^U  (U_Lf_L)_{lk}^2 + n \gamma_L^2 + \E (W^{(j)}_{lk})^2 = \frac{\gamma_L^2 m n^2 \sqrt{K_L} \|f_L\|_2^2}{d_L} + n \gamma_L^2 + \E (W^{(j)}_{lk})^2.
    \end{equation*}
    Subtracting $d^{-1} \gamma_L^2 m n^2 \sqrt{K_L} \|f_L\|_2^2$ on both sides, the first term in \eqref{eq:supp:type_II_error_bound_impure_DP} is bounded above by
    \begin{equation*}
    \text{Pr} \left( \frac{1}{\sqrt{K_L}} \underset{(l,k) \in I_L}{\overset{}{\sum}} \left[ V_{lk} - \E_f V_{lk} \right] < - \frac{ \gamma_L^2 n^2 m \sqrt{K_L}  \| f_L \|_2^2}{2{d_L}} \right)
    \end{equation*}
    whenever
    \begin{equation}\label{eq:also_verified_pub_coin_privacy_general_lem}
    \frac{\gamma_L^2 m n^2 \sqrt{K_L} \|f_L\|_2^2}{{d_L}}  \geq 2 \kappa_\alpha (\gamma_L^2 n \vee 1).
    \end{equation}
    The latter is indeed satisfied whenever \eqref{eq:lem:sup:procedure_III_test_up_rate}. Under the alternative hypothesis, a straightforward calculation shows that $V_l$ has expectation conditionally on $U_L$ equal to $\gamma_L^2 m n^2 (U_L f)_{lk}^2 + n \gamma_L^2 + \E (W^{(j)}_{lk})^2$. Since $\E (Z_{L;i})_{lk}^4 = 3$ and $\E (W^{(j)}_{lk})^4 \asymp 1$, its variance conditionally on $U_L$ equals
    \begin{align*}
    n^2 \gamma_L^4 &\text{Var}\left( \eta^2 \big| U_L \right) + \text{Var}\left( \left(\frac{1}{\sqrt{m}}  \underset{j = 1}{\overset{m}{\sum}}W^{(j)}_{lk} \right)^2 \bigg|U_L \right) + 2 \gamma_L^4 n^3 m (U_L f)_{lk}^2 \E \eta^2 \\ &+ 2 \gamma_L^2 m n^2 (U_Lf)_{lk}^2 \E (W^{(j)}_{lk})^2 + 2 \gamma_L^2 n \E (W^{(j)}_{lk})^2 \E \eta^2, \nonumber 
    \end{align*}
    which is of the order 
    \begin{equation*}\label{eq:variance_of_public_coin_privacy_transcript_under_Pf_pub_coin}
    (\gamma_L^4 m n^3 (U_L f)_{lk}^2) \vee (\gamma_L^2 m n^2 (U_L f)_{lk}^2) \vee \gamma_L^4 n^2 \vee 1. 
    \end{equation*}
    Consequently, by applying Chebyshev's inequality, the probability on the left-hand side of \eqref{eq:supp:type_II_error_bound_impure_DP} is of the order
     \begin{equation}\label{eq:supp:protocol_III_4max_condition}
     {\left(\frac{d_L}{ m n \sqrt{K_L} \|f_L\|_2^2}\right) \vee \left(\frac{d_L  \Lambda_{N,\cS,\delta} \tilde{\kappa}_\alpha^2}{  m n^2 \epsilon^2 \|f_L\|_2^2}\right)  \vee \left(\frac{ \kappa_\alpha^2 d^2_L}{m^2 n^2 K_L \|f_L\|_2^4}\right) \vee \left(\frac{ \kappa_\alpha^2 \tilde{\kappa}_\alpha^4 \Lambda_{N,\cS,\delta}^2 d^2_L  }{  m^2 n^4 \epsilon^4 K_L \|f_L\|_2^4}\right)},
     \end{equation}
     where $\Lambda_{N,\cS,\delta} := \log(N) |\cS| \log(1/\delta) \log$, and we have plugged in $\gamma_L = \frac{\epsilon}{2\sqrt{K_L |\cS| \log(2 / \delta)\tau}}$, $\tau = \tilde{\kappa}_\alpha \sqrt{\log(N)}$. 
    The latter can be made arbitrarily small per choice of $C_\alpha > 0$ in \eqref{eq:lem:sup:procedure_III_test_up_rate}. 
    
    To see this, first consider the case where $\epsilon \leq n^{-1/2}$, such that $K_L \asymp 1$ and the above display is of the order
    \begin{equation*}
         \left(\frac{d_L  \Lambda_{N,\cS,\delta} \tilde{\kappa}_\alpha^2}{  m n^2 \epsilon^2 \|f_L\|_2^2}\right) \vee \left(\frac{ \kappa_\alpha^2 \tilde{\kappa}_\alpha^4 \Lambda_{N,\cS,\delta}^2 d^2_L  }{ K_L m^2 n^4 \epsilon^4 \|f_L\|_2^4}\right)
    \end{equation*}
    Whenever $f$ satisfies \eqref{eq:lem:sup:procedure_III_test_up_rate}, the first term is of the order $1/C_\alpha$. 
    
    For $1/\sqrt{n} \leq \epsilon \leq d_L $, $K_L \asymp n \epsilon^2 \wedge d_L \gtrsim 1$ and \eqref{eq:supp:protocol_III_4max_condition} reduces to
     \begin{equation*}
     \frac{ \kappa_\alpha^2 \tilde{\kappa}_\alpha^4 \Lambda_{N,\cS,\delta}^2 d^2_L  }{ m^2 n^3 \epsilon^2 \|f_L\|_2^4},
     \end{equation*}
    When \eqref{eq:lem:sup:procedure_III_test_up_rate} holds, the above display is bounded by $1/C_\alpha^2$. 
    
    Lastly, when $n \epsilon^2 \gtrsim d_L$, it holds that $K_L \asymp d_L$, so \eqref{eq:supp:protocol_III_4max_condition} is of the order ${ \kappa_\alpha^2 d_L}/({m^2 n^2 \|f_L\|_2^4})$. When \eqref{eq:lem:sup:procedure_III_test_up_rate} holds, the above display is bounded by $1/C_\alpha^2$. We conclude that $\P_f (1 - T_{\text{III}}) \leq \alpha/4$ for $C_\alpha$ large enough, as desired.
    \end{proof}

    \subsection{Proofs of the theorems in Section \ref{sec:main_results} and Theorem \ref{thm:attainment_nonadaptive}}\label{sec:proofs_main_results}
    
    The proofs of Theorems \ref{thm:rate_theorem_local_randomness} and \ref{thm:rate_theorem_shared_randomness} in Section \ref{sec:main_results} are direct consequences of Theorem \ref{thm:attainment_nonadaptive} and \ref{thm:nonasymptotic_testing_lower_bound}. To see this, note that for any sequences $m_N \equiv m$, $n_N \equiv N/m$, $\sigma_N \equiv \sigma$, $\epsilon_N \equiv \epsilon$ and $\delta_N \equiv \delta$, Theorem \ref{thm:nonasymptotic_testing_lower_bound} and Theorem \ref{thm:attainment_nonadaptive} can be applied with an arbitrarily slow decreasing sequence $\alpha_N \to 0$ or $\alpha_N \to 1$ in order to obtain the desired convergence of the minimax risk.
    
    Theorem \ref{thm:attainment_nonadaptive} and Theorem \ref{thm:adaptation_rate} are consequences the lemmas proven earlier in the section. Below, we tie together the results of these lemmas to obtain both theorems.
    
    \begin{proof}[Proof of Theorem \ref{thm:attainment_nonadaptive}]
    
    Consider $\cS = \{L\}$, for a given $L \in \N$ which is to be determined. In the case of local randomness only protocols, let $T = T_{\text{I}} \vee T_{\text{II}}$, where $T_{\text{I}}$ and $T_{\text{II}}$ are the tests defined in \eqref{eq:supp:definition_full_test_lips_ext_privacy} and \eqref{eq:supp:test_impure_DP_coordinate_wise_strat}, respectively, with their critical regions such that $\P_0 T_{\text{I}} \leq \alpha /4$ and $\P_0 T_{\text{II}} \leq \alpha / 4$ (see the first statements of Lemma \ref{lem:supp:lips_test_up_rate} and Lemma \ref{lem:protocol_II_priv_coin_bound}). If, for some $L \in \N$, $f$ satisfies 
    \begin{equation}\label{eq:supp:rate_condition_main-results-local-rand}
        \left\| {f}_L \right\|_2^2 \geq C_\alpha M_N \sigma^2 \left( \frac{2^{(3/2)L}}{ m n  (n\epsilon^2 \wedge 2^L)} \bigwedge \left( \frac{\sqrt{2^L}}{\sqrt{m}n \sqrt{n \epsilon^2 \wedge 1}} \bigvee \frac{1}{mn^2 \epsilon^2} \right) \right),
        \end{equation}
    where
    \begin{equation}\label{eq:supp:rate_condition_M_N}
        M_N \gtrsim \log(N)  \sqrt{2 \log \log(N ) \log(N) |\cS|  \log (2  / \delta)},
    \end{equation}
    Lemma \ref{lem:supp:lips_test_up_rate} and Lemma \ref{lem:protocol_II_priv_coin_bound} yield that 
    \begin{equation}
        \P_f (1 - T) \leq \P_f (1 - T_{\text{I}}) + \P_f (1 - T_{\text{II}}) \leq \alpha / 2.
    \end{equation}
    In view of $(a+b)^2/2-b^2\leq a^2$,
    \begin{equation*}
    \| f_L \|_{2}^2 \geq \frac{\| f \|_{2}^2}{2} - \| f - f_L \|_{2}^2.
    \end{equation*}
    Since $f \in \cB^{{s},R}_{p,q}$, we have that $\| f - f_L \|_{2}^2 \leq 2^{-2Ls} R^2$ (see e.g. Lemma \ref{lemma:proof-of-bound-on-tail-sum}).  Consequently, taking $L = 1 \vee \lceil - \frac{1}{s} \log_2 ( \rho^{}) \rceil$,
    \begin{equation*}
     \| f_L \|_{2}^2 \geq   \rho^2 M_N  C_\alpha^2/2- R^2 2^{-2Ls} \geq \rho M_N^2 ( C_\alpha^2/2 - {R^2}).
    \end{equation*}
    We note here that the latter bound could be sharper by choosing $L = 1 \vee \lceil - \frac{1}{s} \log_2 ( \rho^{}) \rceil$, but we choose the simpler bound for the sake of clarity. Given the above display, this choice of $L$ means that \eqref{eq:supp:rate_condition_main-results-local-rand} is satisfied for 
    \begin{equation}\label{eq : priv coin nonparametric testability condition privacy}
        \rho^2 \gtrsim \sigma^2 \left( \frac{(1 \vee \rho^{-1/s})^{3/2}}{ m n  (n\epsilon^2 \wedge (1 \vee \rho^{-1/s}))} \bigwedge \left( \frac{\sqrt{1 \vee \rho^{-1/s}}}{\sqrt{m}n \sqrt{n \epsilon^2 \wedge 1}} \bigvee \frac{1}{mn^2 \epsilon^2} \right) \right).
    \end{equation}
    Solving this for $\rho$, we obtain the rate in Theorem \ref{thm:attainment_nonadaptive}, \eqref{eq:attainment_nonadaptive_local_randomness}. 
    
    In case of shared randomness, the same argument applies, with the only difference that the test $T = T_{\text{I}} \vee T_{\text{III}}$ is used. Lemmas \ref{lem:supp:lips_test_up_rate} and \ref{lem:procedure_III_test_up_rate} yield that this test satisfies $\P_0 T \leq \alpha / 2$ and also $\P_f (1 - T) \leq \alpha / 2$ whenever 
    \begin{equation*}
        \rho^2 \gtrsim  \sigma^2  \left( \frac{1 \vee \rho^{-1/s}}{ m n \sqrt{n \epsilon^2 \wedge 1} \sqrt{n\epsilon^2 \wedge (1 \vee \rho^{-1/s})}} \bigwedge \left( \frac{\sqrt{1 \vee \rho^{-1/s}}}{\sqrt{m}n \sqrt{n \epsilon^2 \wedge 1} } \bigvee \frac{1}{mn^2 \epsilon^2} \right) \right)
    \end{equation*}
    and $M_N$ satisfying \eqref{eq:supp:rate_condition_M_N}. Solving the above display for $\rho$, we obtain the rate in Theorem \ref{thm:attainment_nonadaptive}, \eqref{eq:attainment_nonadaptive_shared_randomness}.
    \end{proof}
    
    \begin{proof}[Proof of Theorem \ref{thm:adaptation_rate}]
    The proof is very similar to that of Theorem \ref{thm:attainment_nonadaptive} in the sense that it follows from the lemmas proven earlier in the section. 
    
    In the case of local randomness, consider $T$ as defined in \eqref{eq:final_adaptive_test_local_randomness}, with $T_{\text{I}}$ and $T_{\text{II}}$ as defined in \eqref{eq:definition_full_test_lips_ext_privacy_adaptive} and \eqref{eq:test_impure_DP_coordinate_wise_strat_adaptive}, respectively. Applying Lemma \ref{lem:supp:lips_test_up_rate} with $\cS = S^{\texttt{LOW}}$ and Lemma \ref{lem:protocol_II_priv_coin_bound} with $\cS = S^{\texttt{HIGH}}$ yields that $\P_0 T_{\text{I}} \leq \alpha / 4$ and $\P_0 T_{\text{II}} \leq \alpha / 4$, so $T$ has the correct level. 
    
    If $f \in \cB^{{s},R}_{p,q}$, with $s \in [s_{\min}, s_{\max}]$, we have that $L_s$ is either in $S^{\texttt{LOW}}$ or $S^{\texttt{HIGH}}$. By applying Lemma \ref{lem:supp:lips_test_up_rate} and Lemma \ref{lem:protocol_II_priv_coin_bound}, a similar argument as given in the proof of Theorem \ref{thm:attainment_nonadaptive} above yields that  $\P_f (1 - T) \leq \alpha / 2$. The shared randomness case is analogous, with $T$ defined as in \eqref{eq:final_adaptive_test_shared_randomness}.
    \end{proof}
    
    \section{Addendum to the main results of Section \ref{sec:main_results}}\label{sec:supp:additional_rates}
    
    In this section, we provide additional discussion on the minimax testing rate under privacy constraints, presented in Section \ref{sec:main_results}.

    We first consider the case of local randomness protocols, where the minimax rate is derived in Theorem \ref{thm:rate_theorem_local_randomness}. The minimax separation rate satisfies (up to logarithmic factors) the relationship:
    \begin{equation*}
        \rho^2 \asymp  \left(\frac{\sigma^{2}}{mn}\right)^{\frac{2{s}}{2{s}+1/2}} + \left(\frac{\sigma^2}{m n^{2} \epsilon^2} \right)^{\frac{2{s}}{2{s}+3/2}} \wedge \left( \left(\frac{\sigma^2}{\sqrt{m} n^{} \sqrt{1 \wedge n \epsilon^2}} \right)^{\frac{2{s}}{2{s}+1/2}} + \left( \frac{\sigma^2}{mn^2 \epsilon^2} \right) \right).
        \end{equation*}
    
    We can find the dominant term in the above relationship for different values of $\epsilon$. For $s > 1/4$, this yields that the minimax rate is given by
    
    \begin{align}\label{eq:nonparametric_local_randomness_rate_privacy}
        \rho^2 \asymp   \begin{cases} 
        (mn / \sigma^2)^{-\frac{2s}{2s+1/2}} &\text{ if } \epsilon \geq \sigma^{-\frac{2}{4s+1}} m^{\frac{1}{4s+1}} n^{\frac{1/2-2s}{4s+1}}, \\
        (m n^{2} \epsilon^2 / \sigma^2 )^{-\frac{2s}{2s+3/2}} &\text{ if } \sigma^{-\frac{2}{4s+1}} m^{\frac{1/4-s}{4s+1}} n^{\frac{1/2-2s}{4s+1}} \leq \epsilon < \sigma^{-\frac{2}{4s+1}} m^{\frac{1}{4s+1}} n^{\frac{1/2-2s}{4s+1}} \text{ and }  \epsilon \geq n^{-1/2}, \\
        (m n^{2} \epsilon^2 / \sigma^2 )^{-\frac{2s}{2s+3/2}} &\text{ if } \sigma^{-\frac{4}{4s-1}} m^{-\frac{1}{2}} n^{\frac{5/2-2s}{4s-1}} \leq  \epsilon < n^{-1/2}, \\
        (\sqrt{m} n^{} / \sigma^2 )^{-\frac{2s}{2s+1/2}} &\text{ if } n^{-1/2} \leq \epsilon < \sigma^{-\frac{2}{4s+1}} m^{\frac{1/4-s}{4s+1}} n^{\frac{1/2-2s}{4s+1}}, \\
         (\sqrt{m} n^{3/2} \epsilon / \sigma^2 )^{-\frac{2s}{2s+1/2}} &\text{ if } \sigma^{ \frac{1}{2s+1}} m^{-\frac{1}{2}} n^{-\frac{1+s}{2s+1}} \leq \epsilon < \sigma^{-\frac{4}{4s-1}} m^{-\frac{1}{2}} n^{\frac{5/2-2s}{4s-1}}  \text{ and } \epsilon < n^{-1/2}, \\
        (mn^2 \epsilon^2 / \sigma^2)^{-1} &\text{ if }  \epsilon <  \sigma^{\frac{1}{2s+1}} m^{-\frac{1}{2}} n^{-\frac{1+s}{2s+1}}. 
        \end{cases}
        \end{align}
    
    Whenever $s \leq 1/4$, the first two cases in the above display become vacuous for $\epsilon \lesssim 1$, since $mn \to \infty$. Furthermore, the condition
    \begin{equation*}
        \left(\frac{\sigma^2}{m n^{2} \epsilon^2} \right)^{\frac{2{s}}{2{s}+3/2}} \asymp \left(\frac{\sigma^2}{\sqrt{m} n^{3/2} \epsilon} \right)^{\frac{2{s}}{2{s}+1/2}} \iff \epsilon^{\frac{2s-1/2}{2s+3/2}} \asymp \sigma^{-\frac{1}{2s+3/2}} m^{\frac{1}{2s+3/2}} n^{\frac{5/4-s}{2s+3/2}},
    \end{equation*}
    cannot be satisfied when $s \leq 1/4$, $\epsilon \gtrsim (mn)^{-1}$ and $\sigma^2 = O(1)$. Consequently, whenever $s \leq 1/4$, the minimax rate simplifies to 
    \begin{align}\label{eq:nonparametric_local_randomness_rate_privacy_small_s}
    \rho^2 \asymp   \begin{cases} 
    (\sqrt{m} n^{} / \sigma^2 )^{-\frac{2s}{2s+1/2}} &\text{ if } \epsilon \geq n^{-1/2}, \\
     (\sqrt{m} n^{3/2} \epsilon / \sigma^2  )^{-\frac{2s}{2s+1/2}} &\text{ if } \sigma^{ \frac{2}{4s+1}} m^{-\frac{1}{2}} n^{-\frac{1+s}{2s+1}} \leq  \epsilon < n^{-1/2}, \\
    ( mn^2 \epsilon^2 / \sigma^2 )^{-1} &\text{ if }  \epsilon < \sigma^{ \frac{2}{4s+1}} m^{-\frac{1}{2}} n^{-\frac{1+s}{2s+1}}. 
    \end{cases}
    \end{align}
    That is, the minimax rates corresponding to the high privacy-budget regimes cannot be attained for any value of $\epsilon$ when $s \leq 1/4$.
    
    Next, we consider the case of shared randomness protocols, where the minimax rate is derived in Theorem \ref{thm:rate_theorem_shared_randomness}. The minimax separation rate satisfies (up to logarithmic factors) the relationship:
    \begin{equation*}
        \rho^2 \asymp \left(\frac{\sigma^2}{mn}\right)^{\frac{2{s}}{2{s}+1/2}} + \left(\frac{\sigma^2}{m n^{3/2} \epsilon \sqrt{1 \wedge n \epsilon^2} } \right)^{\frac{2{s}}{2{s}+1}} \wedge \left( \left(\frac{\sigma^2}{\sqrt{m} n^{} \sqrt{1 \wedge n \epsilon^2}} \right)^{\frac{2{s}}{2{s}+1/2}} + \frac{\sigma^2}{mn^2 \epsilon^2} \right).
    \end{equation*}
    
    Finding the dominant term in the above relationship, we obtain that the minimax rate for shared randomness protocols is given by
    
    \begin{align}\label{eq:nonparametric_shared_randomness_rate_privacy}
    \rho^2 \asymp   \begin{cases} 
    (mn/\sigma^2)^{-\frac{2{s}}{2{s}+1/2}} &\text{ if } \epsilon \geq \sigma^{-\frac{2}{4s+1}} m^{\frac{1}{4{s}+1}} n^{\frac{1/2-2{s}}{4{s}+1}}, \\
    (m n^{3/2} \epsilon /\sigma^2 )^{-\frac{2{s}}{2{s}+1}} &\text{ if } \sigma^{-\frac{2}{4s+1}} m^{-\frac{2{s}}{4{s}+1}} n^{\frac{1/2-2{s}}{4{s}+1}} \leq \epsilon < \sigma^{-\frac{2}{4s+1}} m^{\frac{1}{4{s}+1}} n^{\frac{1/2-2{s}}{4{s}+1}}, \text{ and }  \epsilon \geq n^{-1/2},\\
    (m n^{2} \epsilon^2 /\sigma^2 )^{-\frac{2{s}}{2{s}+1}} &\text{ if } \sigma^{-\frac{1}{2s}} m^{-\frac{1}{2}} n^{\frac{1-2{s}}{4{s}}} \leq \epsilon < n^{-1/2}, \\
    (\sqrt{m} n^{} /\sigma^2 )^{-\frac{2{s}}{2{s}+1/2}} &\text{ if } n^{-1/2} \leq \epsilon < \sigma^{-\frac{2}{4s+1}} m^{-\frac{2{s}}{4{s}+1}} n^{\frac{1/2-2{s}}{4{s}+1}}, \\
     (\sqrt{m} n^{3/2} \epsilon /\sigma^2 )^{-\frac{2{s}}{2{s}+1/2}} &\text{ if } \sigma^{\frac{1}{2s+1}}  m^{-\frac{1}{2}} n^{-\frac{1+s}{2{s}+1}} \leq \epsilon < \sigma^{-\frac{1}{2s}} m^{-\frac{1}{2}} n^{\frac{1-2{s}}{4{s}}} \text{ and } \epsilon < n^{-1/2}, \\
    (mn^2 \epsilon^2 /\sigma^2 )^{-1} &\text{ if }  \epsilon < \sigma^{\frac{1}{2s+1}} m^{-\frac{1}{2}} n^{-\frac{1+s}{2{s}+1}}.
    \end{cases}
    \end{align}
    
    Interestingly, even when $s \leq 1/4$, the high privacy-budget regimes occur for shared randomness protocols for some large enough value of $\epsilon$. This is in contrast to the case of local randomness protocols, where the minimax rates corresponding to the high privacy-budget regimes cannot be attained for any value of $\epsilon$ when $s \leq 1/4$.
    
    In Figure \ref{fig:nonpar_testing_large_m} below, we illustrate how the cost of privacy effectively increases the more data is distributed across multiple servers. In particular, we consider the case of $N=16$ observations, $m=8$ servers with each $n=2$ observations. Even though the total number of observations is the same, comparison with Figure \ref{fig:nonpar_testing_large_n} shows that the separation rate in the more distributed setting ($m=8$) is much larger for any value of $\epsilon$.

    \section{Auxiliary lemmas} 
    
    \subsection{General privacy related lemmas}
    
    A random variable $V$ is called $\nu$-\emph{sub-gaussian} if
    \begin{equation}\label{eq:def:subgaus}
    \P( |V| \geq t ) \leq 2 e^{- t^2/\nu^2}, \; \text{ for all } t>0.
    \end{equation}
    It is well known (see e.g. \cite{vershynin_high-dimensional_2018}) that if $\E V = 0$, the above inequality holds if and only if
    \begin{equation*}
    \E e^{t V} \leq e^{C \nu^2 t^2}, \; \text{ for all } t\geq 0,
    \end{equation*}
    for a constant $C> 0$. A random variable $V$ is called $\nu$-\emph{sub-exponential} if
    \begin{equation}\label{eq:def:subexp}
    \P( |V| \geq t ) \leq 2 e^{- t/\nu}, \; \text{ for all } t\geq0.
    \end{equation}
    If $\E V = 0$, the above inequality holds if and only if $\E e^{t V} \leq e^{C \nu^2 t^2}$ for all $0 \leq t \leq 1/(c\nu)$, with constants $c,C > 0$.
    
    The following lemma shows that clipping symmetric, mean zero random variables preserves sub-gaussianity and sub-exponentiality.
    
    \begin{lemma}\label{lem:clipping_preserves_tails_exponential}
    Let $V_1,\dots,V_m$ denote independent random variables, each symmetrically distributed around zero. If $\sum_{j=1}^m V_j$ are sub-gaussian (resp. sub-exponential) with parameter $\nu$, then so are the random variables $\sum_{j=1}^m [V_j]_{-\tau}^\tau$, for any $\tau > 0$. 
    \end{lemma}
    \begin{proof}
    For any symmetric about $0$ function $g : \R \to \R$ such that $x \mapsto g(x)$ is increasing on $[0,\infty)$, it holds that
    \begin{equation}\label{eq:fact_monotonicity_and_clipping}
    g\left( [x]_{-\tau}^\tau \right) \leq g\left( x \right), \quad \text{ for all } \tau > 0.
    \end{equation}
    Since $V_j$ is symmetric about zero, so is $[V_j]_{-\tau}^\tau$. For an independent Rademacher random variable $R_j$, we have by the afformentioned symmetry about zero that $[V_j]_{-\tau}^\tau \overset{d}{=} R_j [V_j]_{-\tau}^\tau$ and consequently
    \begin{equation*}
    \E e^{t [V_j]_{-\tau}^\tau } = \E e^{t R_j [V_j]_{-\tau}^\tau } = \E \cosh\left(t [V_j]_{-\tau}^\tau \right).
    \end{equation*}
    Using the fact that $V_1,\dots,V_m$ are independent, we obtain that
    \begin{align*}
    \E e^{t \underset{j=1}{\overset{m}{\sum}} V_j } &= \underset{j=1}{\overset{m}{\prod}}   \E \cosh \left(t [V_j]_{-\tau}^\tau \right) \leq \underset{j=1}{\overset{m}{\prod}}   \E \cosh\left(t V_j\right),
    \end{align*}
    where the inequality follows from \eqref{eq:fact_monotonicity_and_clipping}. The conclusion can now be drawn from the moment generating function characterization of sub-gaussianity (resp. sub-exponentiality).
    \end{proof}
    
    The next lemma gives a lower bound on the expectation of a clipped random variable that is symmetric around a real number $\mu > 0$.
    
    \begin{lemma}\label{lem:clipping_preserves_mean_under_alt}
    Let $\tau,\mu > 0$ satisfy $\tau/4 \leq \mu \leq \tau/2$, let $V$ be a random variable symmetric about zero ($V \overset{d}{=} -V$) with Lebesgue density bounded by $M > 0$ and
    \begin{equation*}
    \text{Pr} \left( |V| \leq \frac{1}{12M} \vee (\tau/2) \right) \geq c
    \end{equation*}
    for some constant $c > 0$. It then holds that
    \begin{equation}\label{eq:to_show_mean_of_clipping}
    \E \left[ \mu + V \right]_{-\tau}^\tau \geq (c \wedge 1/2) \mu.
    \end{equation}
    \end{lemma}
    \begin{proof}
    By definition of clipping,
    \begin{align*}
    \E \left[ \mu + V \right]_{-\tau}^\tau &= \E \left[ V \right]_{-\tau-\mu}^{\tau-\mu} + \mu.
    \end{align*}
    The first term equals
    \begin{align*}
    \E  \left[ V \right]_{-(\tau-\mu)}^{\tau-\mu} + \E \mathbbm{1}_{\{ V \in [-\tau-\mu,{-\tau+\mu}] \}}\left(\left[ V \right]_{-\tau-\mu}^{-\tau+\mu} + (\tau-\mu)  \right) &\geq \\
    \E \left[ V \right]_{-(\tau-\mu)}^{\tau-\mu} - (\tau+\mu) \text{Pr} \left( -\tau - \mu \leq V \leq -\tau +\mu \right) &=\\ 
     - (\tau+\mu) \text{Pr} \left( -\tau - \mu \leq V \leq -\tau +\mu \right),
    \end{align*}
    where the last equality follows from the symmetry of $V$. By the condition on the Lebesgue density of $V$, the second term in the above display can be further bounded from below by $- 2 M (\tau+\mu) \mu$. When $3M \tau < 1/2$, we obtain \eqref{eq:to_show_mean_of_clipping} with the constant $1/2$. Assume $3M \tau \geq 1/2$. Then, since $\mu > 0$ and $V$ is symmetric about zero,
    \begin{align*}
    \E \left[ \mu + V \right]_{-\tau}^\tau 
    &= \E \left( \mu + V \right) \mathbbm{1}\{ |\mu + V| \leq \tau  \} + \tau \left( \text{Pr} \left( \mu + V > \tau \right) - \text{Pr} \left( \mu + V < \tau \right) \right) \\
    &\geq \E \left( \mu + V \right) \mathbbm{1}\{ |V| \leq \tau - \mu \} \\
    &\geq \E \left( \mu + V \right) \mathbbm{1}\{ |V| \leq \tau/2 \} \\
    &= \mu \text{Pr} \left( |V| \leq (1/12M) \vee (\tau/2) \right)
    \end{align*}
    where the last inequality follows from $\mu \leq \tau/2$ and the last equality follows from the symmetry of $V$ about zero. 
    \end{proof}
    
    \begin{lemma}\label{lem:folklore_clipping_reduces_variance}
    For any $\tau > 0$ and random variable $V$ with $|\E V| \leq \tau$,
    \begin{equation*}
     \text{Var}\left( \left[ V \right]_{-\tau}^{\tau}\right) \leq \text{Var}\left( V \right).
    \end{equation*}
    If both $V$ and $\left[ V \right]_{-\tau}^{\tau}$ are mean $0$, $\E \left(\left[ V \right]_{-\tau}^{\tau} \right)^4 \leq \E V^4$.
    \end{lemma}
    \begin{proof}
    Let $\mu = \E V$. Since the expectation of a random variable is the constant minimizing the $L_2$-distance to that random variable,
    \begin{equation*}
    \text{Var}\left( \left[ V \right]_{-\tau}^{\tau}\right) \leq \E  \left( \left[ V  \right]_{-\tau}^{\tau} - \mu \right)^2 .
    \end{equation*}
    The latter expectation can be written as
    \begin{align*}
    \E \mathbbm{1}_{V \in [-\tau,\tau]} \left( V  - \mu \right)^2 + \E \mathbbm{1}_{V > \tau} \left( \tau - \mu \right)^2 + \E \mathbbm{1}_{V < -\tau} \left( - \tau  - \mu \right)^2.
    \end{align*}
    Assuming $\mu \geq 0$, $V < -\tau$ implies that $|- \tau  - \mu| \leq |V  - \mu|$. Since $\mu \leq \tau$, $V > \tau$ implies that $|\tau - \mu| \leq |V  - \mu|$. Consequently, the above display is bounded from above by
    \begin{align*}
     \E \mathbbm{1}_{V \in [-\tau,\tau]} \left( V  - \mu \right)^2 + \E \mathbbm{1}_{V > \tau} \left( V - \mu \right)^2 + \E \mathbbm{1}_{V < -\tau} \left( V  - \mu \right)^2 = \E  \left(  V   - \mu \right)^2.
    \end{align*}
    The case where $\mu < 0$ follows by the same reasoning. The last statement follows by a similar argument.
    \end{proof}

    \subsection{General auxiliary lemmas}
    
    The following lemmas are well known but included for completeness.
    
    \begin{lemma}\label{lem:petrov_subgaussianity}
       Consider for $L \in \N$ and a nonnegative positive integer sequence $K_n$,
    \begin{equation*}
    S_n := \frac{1}{\sqrt{K_n}} \underset{i=1}{\overset{K_n}{\sum}} \zeta_{i}
    \end{equation*}
    where $(\zeta_{1},\dots,\zeta_{K_n})$ independent random variables with mean zero and unit variance.
    
    Assume that the random variables satisfy Cram\'er's condition, i.e. for some $\epsilon > 0$ and all $t \in (-\epsilon,\epsilon)$, $i=1,\dots,K_n$,
    \begin{equation*}
    \E e^{t \zeta_{i}} < \infty.
    \end{equation*} 
    Then, for any sequence of positive numbers $t_n$ such that $K_n \gtrsim t_n^3$, it holds that
    \begin{equation*}
    {\text{Pr} \left(S_n \geq t_n  \right)} \leq O(1) \cdot e^{-t_n^2/2}.
    \end{equation*}
    \end{lemma}
    \begin{proof}
    By Cram\'er's theorem (see e.g. Theorem 7 in Section 8.2 of \cite{petrov2022sums}),
    \begin{equation*}\label{eq : law of iterated log lem condition to verify}
    \frac{\text{Pr} \left(S_n \geq t_n  \right)}{1 - \Phi( t_n) } = \exp \left( O(1) \cdot \frac{t_n^3}{\sqrt{K_n}} \right) \left(1 + O \left(\frac{t_n }{ \sqrt{ K_n}} \right) \right)  =O(1).
    \end{equation*}  
    Note that the above statement holds for $-S_n$ also. The statement now follows by using $1 - \Phi(x) \leq e^{-x^2/2}$. 
    \end{proof}
    
    The following lemma is a standard tail bound for the maximum of sub-gaussian (resp. sub-exponential) random variables that is used to control the type I error of the adaptive tests.
    \begin{lemma}\label{lem:iterated_logarithm_max_bound}
    Consider a sequence of subsets $\cS_n \subset \N$ and let $S_n(L)$ be $\nu$-sub-gaussian. Then, 
    \begin{equation*}
    \text{Pr} \left( \underset{L \in \cS_n}{\max} \; |S_n(L)| \geq  t_n  \right) \leq |\cS_n| e^{-t_n^2/(2\nu^2)}.
    \end{equation*}
    If $S_n(L)$ is $\nu$-sub-exponential with sub-exponentiality constant equal to one, the statement is true for any positive sequence $C_n$ and integer sequence $K_n$.
    \end{lemma}
    \begin{proof}
    The proof follows immediately from union bounds and the properties of sub-gaussian (resp. sub-exponential random variables), see e.g. \eqref{eq:def:subgaus} (resp. \eqref{eq:def:subexp}).
    \begin{align*}
    \text{Pr} &\left( \underset{L \in \cS_n}{\max} \, S_n(L) \geq t_n  \right) \leq  \underset{L \in \cS_n}{\overset{}{\sum}} \text{Pr} \left(|S_n(L)| \geq t_n  \right) \leq  2|\cS_n| e^{-t_n^2/(2\nu^2)}.
    \end{align*}
    The proof follows by showing that $S_n(L)$ and $-S_n(L)$ are or tend to sub-Gaussian variables with sub-gaussianity constant less than or equal to one, since this allows for bounding the above display by
    \begin{align*}
     2\underset{L \in \cS_n}{\overset{}{\sum}} e^{- \frac{c^2}{2} \log C_n } \leq \frac{2 |\cS_n|}{C_n^{c^2/2}}
    \end{align*}
    and the result follows.

    \end{proof}
    
    Consider the following formal definition of coupling.
    
    \begin{definition}\label{def:coupling}
    Consider probability measures $P$ and $Q$ on a measurable space $(\cX, \mathscr{X})$. A \emph{coupling} of $P$ and $Q$ is any probability measure $\P$ on $(\cX \times \cX, \mathscr{X} \otimes \mathscr{X})$ such that $\P$ has marginals $P$ and $Q$:
    \begin{equation*}
    P = \P \circ \pi^{-1}_1, \; Q = \P \circ \pi^{-1}_2
    \end{equation*}
    where $\pi_i : \cX \times \cX \to \cX$ is the projection onto the $i$-th coordinate (i.e. $\pi_i(x_1,x_2) = x_i$ for $i=1,2$).
    \end{definition}
    
    Lemma \ref{lem:coupling_and_TV} below is a well known result showing that, for random variables $X$ and $\tilde{X}$ defined on a Polish space, small total variation distance between their corresponding laws guarantees the existence of a coupling such that they are equal with high probability. 
    
    \begin{lemma}\label{lem:coupling_and_TV}
    For any two probability measures $P$ and $Q$ on a measurable space $(\cX, \mathscr{X})$ with $\cX$ a Polish space and $\mathscr{X}$ its Borel sigma-algebra. There exists a coupling $\P^{X,\tilde{X}}$ such that
    \begin{equation*}
    \| P - Q \|_{TV} = 2 \P^{X,\tilde{X}} \left( X \neq \tilde{X} \right).
    \end{equation*}
    \end{lemma}
    For a proof, see e.g. Section 8.3 in \cite{thorisson2000coupling}.
    
    The following lemma is well known and included for completeness.
    
    \begin{lemma}\label{lem:TV_distance_bound_many-normal-means}
    Let $P_f$ denote the distribution of a $N(f,\sigma I_d)$ distributed random vector for $f \in \R^d$ and let $P_f^n$ denote the distribution of $n$ i.i.d. draws (i.e. $P_f^n = \bigotimes_{i=1}^n P_f$).
    
    It holds that
    \begin{equation*}
    \left\| P_f^n - P_g^n \right\|_{\mathrm{TV}} \leq \frac{n}{2\sigma} \left\| f - g \right\|_2.
    \end{equation*}
    \end{lemma}
    \begin{proof}
    By Pinsker's inequality,
    \begin{equation*}
    \| P_{f}^n - P_{g}^n\|_{\mathrm{TV}}  \leq \sqrt{\frac{n}{2} D_{\mathrm{KL}}(P_{f} ; {P}_{g})}.
    \end{equation*}
    A straightforward calculation gives that the latter is bounded by $\frac{\sqrt{n}}{2\sigma}\left\| f - g \right\|_2$.
    \end{proof}
    
    The following lemma bounds the maximum of a possibly correlated Gaussian random vector.
    
    \begin{lemma}\label{lem:gaussian_maximum}
    Let $K\in \N$ and $M \in \R^{K \times K}$ be symmetric and positive definite. Consider the random vector $G=(G_1,\dots,G_K) \sim N(0,M)$. It holds that $\E \underset{1 \leq i \leq K}{\max} |G_i| \leq 3 \| M \| \sqrt{\log(K) \vee \log(2)}$ and
    \begin{equation*}
    \text{Pr} \left(\underset{1 \leq i \leq K}{\max} G_i^2 \geq \|M\|^2 x \right) \leq \frac{2K}{e^{x/4}},
    \end{equation*}
    for all $ x > 0$.
    \end{lemma}
    \begin{proof}
    It holds that
    \begin{equation*}
    G \overset{d}{=} \sqrt{M} Z, \quad \text{ with } \quad Z \sim N(0,I_K).
    \end{equation*}
    Since $M$ is symmetric, positive definite, it has SVD $M = V \text{Diag}(\lambda_1,\dots,\lambda_K) V^\top$. Since $V$ is orthonormal, 
    \begin{equation*}
    \sqrt{M} Z = V \sqrt{\text{Diag}(\lambda_1,\dots,\lambda_K)}(V^\top Z) \overset{d}{=}  V \sqrt{\text{Diag}(\lambda_1,\dots,\lambda_K)} Z.
    \end{equation*}
    Writing $V = [v_1 \; \dots \; v_K]$ where $v_k$ are orthogonal unit vectors, the latter display equals
    \begin{equation*}
    \underset{k=1}{\overset{K}{\sum}} \sqrt{\lambda_k} v_k Z_k \; \sim \; N\left(0, \text{Diag}(\lambda_1,\dots,\lambda_K) \right).
    \end{equation*} 
    Consequently,
    \begin{equation*}
    \underset{k \in [K]}{\max} |G_k| \overset{d}{=} \underset{k \in [K]}{\max} |\lambda_k Z_k| \leq \| M \| \underset{k \in [K]}{\max} | Z_k|.
    \end{equation*}
    Hence, it suffices to show that
    \begin{equation*}
    \text{Pr} \left(\underset{1 \leq i \leq K}{\max} Z_i^2 \geq x \right) \leq \frac{2K}{e^{x/4}}
    \end{equation*}
    The case where $K = 1$ follows by standard Gaussian concentration properties. Assume $K\geq 2$. For $ 0 \leq t \leq 1/4$,
    \begin{align*}
     \E e^{t \max_i (Z_i)^2} = e^{t}\E \max_i e^{t (Z_i^2 - 1)} \leq K  e^{{2t^2} + t}.
     \end{align*}
    Taking $t=1/4$ and applying Markov's inequality yields the second statement of the lemma. Furthermore, in view of Jensen's inequality
    \begin{equation*}
    \E \max_i (Z_i)^2 \leq \frac{\log(K)}{t} + 2t + 1,
    \end{equation*}
    which in turn yields $\max_i  |Z_i| \leq 3 \sqrt{\log(K)}$.
    \end{proof}
    
    The next lemma is a standard bound for the KL-divergence, see for instance Lemma 2.7 of \cite{tsybakov:2009}.
    \begin{lemma}\label{lem:chi_sq_div_bounds_KL}
    Let $P,Q$ probability measures on some measure space such that $Q\ll P$. Then,
    \begin{equation*}
    D_{KL}(P\|Q) \leq \int \left(\frac{dP}{dQ} - 1 \right)^2 dQ.
    \end{equation*}
    \end{lemma}
    
    The following lemma is a standard concentration result for chi-square random variables, and is included for completeness.
    \begin{lemma}\label{lem : Chernoff-Hoeffding bound chisq}
        Let $X_d$ be chi-square random variable with $d$-degrees of freedom. For $0<c<1$ it holds that
        \begin{equation*}
        \text{Pr}\left( X_d \leq cd \right) \leq e^{-d \frac{c-1 - \log (c)}{2}}.
        \end{equation*}
        Similarly, for $c > 1$ it holds that
        \begin{equation*}
        \text{Pr}\left( X_d \geq cd \right) \leq e^{-d \frac{c-1 - \log (c)}{2}}.
        \end{equation*}
        \end{lemma}
        \begin{proof}
        Let $t < 0$. We have
        \begin{align*}
        \text{Pr}\left( X_d \leq cd \right) &= \text{Pr}\left( e^{tX_d} \geq e^{tcd} \right) \leq \frac{\E e^{tX_d} }{e^{tcd}}.
        \end{align*}
        Using that $\E e^{tX_d} = (1-2t)^{-d/2}$, the latter display equals
        \begin{equation*}
        \exp \Big( - d \big( tc + \frac{1}{2} \log(1-2t) \big) \Big).
        \end{equation*}
        The expression $tc + \frac{1}{2} \log(1-2t)$ is maximized when $t = \frac{1}{2}(1 - \frac{1}{c}) < 0$ which leads to the result. The second statement follows by similar steps.
        \end{proof}
    
    \begin{lemma}\label{lem:inner_product_ind_gaussians}
    Let $a \in \R$ and let $Z,Z' \overset{\text{i.i.d.}}{\sim} N(0,I_d)$ for $d \in \N$. 
    
    Then, $ a \langle  Z , Z' \rangle$ is $C a^2$-sub-exponential for a universal constant $C>0$ and
    \begin{equation*}
    \E  e^{t |a \langle  Z , Z' \rangle|} \leq 2 e^{ {t^2 a^2}},
    \end{equation*}
    whenever $|t| \leq (2a^2)^{-1}$.
    \end{lemma}
    \begin{proof}
    Since $\langle Z, Z' \rangle |Z' \sim N(0,\|Z'\|_2)$,
    \begin{align*}
    \E e^{t a \langle  Z , Z' \rangle} &= \E^{Z'} \E^{Z|Z'} e^{t a \langle  Z , Z' \rangle} = \E^{Z'} e^{ \frac{t^2 a^2}{2} \|Z'\|_2^2}.
    \end{align*}
    By standard arguments, the latter is further bounded by
    \begin{equation*}
    e^{ \frac{t^2 a^2}{2} + \frac{t^4 a^4}{2}} \leq e^{{t^2 a^2}},
    \end{equation*}
    whenever $t^2 a^2 \leq 1/2$. The conclusion then follows by e.g. Proposition 2.7.1 in \cite{vershynin_high-dimensional_2018}, since $\langle Z, Z' \rangle$ is mean zero. For the last statement,
    \begin{equation*}
    \langle Z, Z' \rangle |Z' \overset{d}{=} - \langle Z, Z' \rangle |Z'.
    \end{equation*}
    Consequently,
    \begin{align*}
    \E^{Z|Z'} e^{t |a \langle  Z , Z' \rangle|} &= \E^{Z|Z'} \mathbbm{1}_{\{ \langle  Z , Z' \rangle > 0 \}} e^{t a \langle  Z , Z' \rangle} + \E^{Z|Z'} \mathbbm{1}_{\{ \langle  Z , Z' \rangle \leq 0 \}} e^{- t a \langle  Z , Z' \rangle} \\
    &\leq 2 \E^{Z|Z'} e^{t a \langle  Z , Z' \rangle},
    \end{align*}
    and the proof follows by what was shown above.
    \end{proof}

    \begin{lemma}\label{lem:DepowerD_binomial_expectation}
    Let $S \sim \text{Bin}(p,n)$ and $0 \leq \epsilon \leq 1$. It holds that
    \begin{equation*}
    \E S e^{\epsilon S} \leq  e^{4\epsilon np}.
    \end{equation*}
    \end{lemma}
    \begin{proof}
    Write $S = \sum_{i=1}^n B_i$, with $B_1,\dots,B_n \overset{\text{i.i.d.}}{\sim} \text{Ber}(p)$. 
    \begin{align*}
    \E S e^{\epsilon S} &= \underset{i=1}{\overset{n}{\sum}} \E B_i e^{\epsilon S}  \\
    &= \underset{i=1}{\overset{n}{\sum}} p e^{\epsilon} \E e^{ \epsilon \sum_{k \neq i} B_k} \\
    &= n p e^{\epsilon} \left( \E e^{ \epsilon B_1} \right)^{n-1} \\
    &= n p e^{\epsilon} \left( 1 + p(e^{\epsilon} - 1)  \right)^{n-1} \\
    &\leq e^{4 n p \epsilon},
    \end{align*}
    where the inequality follows from the fact that $e^x - 1 \leq 2x$ for $0 \leq x \leq 1$.
    \end{proof}
    
    \begin{lemma}\label{lem:bracamp_densities_exist}
        Consider a sample space $\cX$ and a distributed protocol with Markov kernels $K^j: \mathscr{Y}^{(j)} \times \cX \times \cU \to [0,1]$ for $j=1,\dots,m$ and shared randomness distribution $\P^U$. Writing $\mathscr{Y} = \bigotimes^m_{j=1} \mathscr{Y}^{(j)}$ for the product sigma-algebra, consider $K = \bigotimes^m_{j=1} K^j : \mathscr{Y} \times \cX^m \times \cU^m \to [0,1]$. It holds that
        \begin{equation*}
        K(\cdot | x, u) \ll \P^{Y|U=u}_f(\cdot), \; \; \P^{(X,U)}_f-\text{almost surely,}
        \end{equation*}
        for all $f \in \R^d$.
        \end{lemma}
        \begin{proof}
        Let $A \in \mathscr{Y}^{(j)}$. We have that
        \begin{equation*}
        B:=\left\{ (x,u) \, : \, K^{j}( A |x,u) > 0 \right\} = \underset{l \in \N}{\bigcup} \left\{ (x,u) \, : \, K^{j}( A |x,u) > \frac{1}{l} \right\},
        \end{equation*}
        so if $\P^{(X,U)}_f(B)>0$ for some $L \in \N$ it holds that $\P^{(X,U)}_f \left( (x,u) \, : \, K^{j}( A |x,u) > \frac{1}{L} \right) > 0$. Since $K^j$ is nonnegative, by Markov's inequality,
        \begin{align*}
        \P^{Y^{(j)}|U=u}_f(A) &= \int K^{j}( A |x,u) d\P^{X^{(j)}}_f \times \P^U (x,u) \\ &\geq \int_B K^{j}( A |x,u) d\P^{X^{(j)}}_f\times \P^U(x,u) \\ &\geq \frac{1}{L}\P^{(X,U)}_f \left( (x,u) \, : \, K^{j}( A |x,u) > \frac{1}{L} \right) > 0.
        \end{align*}
        Since given $U$, $Y^{(1)},\dots,Y^{(m)}$, the statement for $K$ follows as $K(\cdot|x^1,\dots,x^m,u) := \bigotimes_{j=1}^m K^j(\cdot|x^j,u)$ and $\P^{Y|U=u}_f = \bigotimes_{j=1}^m \P^{Y^{(j)}|U=u}_f$.
        \end{proof}
    
        \begin{lemma}\label{lem:TV_distance_product_measures}
            Let $P = \bigotimes_{j=1}^m P_j$ and $Q = \bigotimes_{j=1}^m Q_j$ for probability measures $P_j,Q_j$ defined on a common measurable space $(\cX,\mathscr{X})$, with probability densities $p_j,q_j$ for $j=1,\dots m$. It holds that
            \begin{equation*}
            \| P - Q \|_{\mathrm{TV}} \leq \underset{j=1}{\overset{m}{\sum}} \|P_j - Q_j\|_{\mathrm{TV}}. 
            \end{equation*}
            \end{lemma}
            \begin{proof}
            The measures $P_j$ and $Q_j$ admit densities with respect to $P_j + Q_j$, which we shall denote by $p_j$ and $q_j$, respectively, with
            \begin{align*}
            p:= \underset{j=1}{\overset{m}{\Pi}} p_j = \frac{d \bigotimes_{j=1}^m P_j}{d\bigotimes_{j=1}^m (P_j+Q_j)} \; \text{ and } \; q:=  \underset{j=1}{\overset{m}{\Pi}} q_j = \frac{d \bigotimes_{j=1}^m Q_j}{d\bigotimes_{j=1}^m (P_j+Q_j)}.
            \end{align*}
            Writing $\mu = \bigotimes_{j=1}^m (P_j+Q_j)$, we obtain 
            \begin{equation}\label{eq:lem_bounding_total_var}
            \| P - Q \|_{\mathrm{TV}} =  \frac{1}{2}\int| \underset{j=1}{\overset{m}{\Pi}} p_j(x_j) - \underset{j=1}{\overset{m}{\Pi}} q_j(x_j) | d\mu(x_1,\dots,x_m).
            \end{equation}
            By the telescoping product identity
            \begin{equation}
            a_1 \cdot a_2 \cdots a_m - b_1 \cdot b_2 \cdots b_m = \underset{j=1}{\overset{m}{\sum}} (a_j - b_j) \underset{k=1}{\overset{j-1}{\Pi}} a_k \underset{k=j+1}{\overset{m}{\Pi}} b_k 
            \end{equation}
            and Fubini's Theorem, the right-hand side of \eqref{eq:lem_bounding_total_var} is bounded by
            \begin{equation*}
            \underset{j=1}{\overset{m}{\sum}} \frac{1}{2} \int|  p_j(x_j) -  q_j(x_j) | d(P_j+Q_j)(x_j) = \underset{j=1}{\overset{m}{\sum}} \| P_j - Q_j \|_{\mathrm{TV}}.
            \end{equation*}
            \end{proof}

    \section{Appendix on Besov norms in sequence space}\label{sec:appendix_wavelet_properties}
    
    In this section we briefly introduce Besov spaces as subspaces of $\ell_2(\N)$ and collect some properties used in the paper.
    
    For $s \in \R$, $1 \leq p \leq \infty$ and $1 \leq q \leq \infty$, we define the Besov norm for $f \in \ell_2(\N)$ as
    \begin{equation*}
       \| f\|_{\cB^{s}_{p,q}} :=  \begin{cases}
          \left( \underset{l=1}{\overset{\infty}{\sum}} \left(2^{l({s}+1/2-1/p)} \left\| (f_{lk})_{k=1}^{2^l} \right\|_p\right)^{q}\right)^{1/q} &\text{ for } 1 \leq q < \infty, \\
          \; \underset{l\geq 1}{\sup} \; 2^{l({s}+1/2-1/p)} \left\| (f_{lk})_{k=1}^{2^l} \right\|_p   &\text{ for } q = \infty.
       \end{cases}
    \end{equation*}
    This definition of the Besov norm is equivalent to the Besov norm as typically defined on a function space when considering the wavelet transform of a function in $L_2[0,1]$, 
    \begin{align*}
        f=\sum_{j=j_0}^{\infty}\sum_{k=0}^{2^{j}-1}f_{jk}\psi_{jk},
    \end{align*}
    where
    \begin{align*}
        \big\{ \phi_{j_0m},\psi_{jk}:\, m\in\{0,\ldots,2^{j_0}-1\},\quad j> j_0,\quad k\in\{0,\ldots,2^{j}-1\} \big\},
    \end{align*}
    are the orthogonal wavelet basis functions, for father $\phi(.)$ and mother $\psi(.)$ wavelets with $N$ vanishing moments and bounded support on $[0,2N-1]$ and $[-N+1,N]$, respectively, following e.g. the construction of Cohen, Daubechies and Vial \cite{cohen1993wavelets, daubechies1992ten}, with $N > s$. The Besov norm on the function space is then equivalent to the one defined above for the wavelet coefficients $(f_{jk})_{j\geq j_0, k \in \{0,...,2^{j}-1\} }$. See e.g. Chapter 4 in \cite{gine_mathematical_2016} for details.
    
    Let $\cB^s_{p,q}(R) \subset \ell_2(\N)$ denote the Besov ball of radius $R$. The following lemmas are standard results, for which we provide a proof for completeness.
    
    \begin{lemma}\label{lem:L_2_to_besov_bound_on_f}
    There exists a constant $C_{s,q} > 0$ such that $\|f\|_2 \leq C_{s,q} R$ for all $f \in \cB^s_{p,q}(R)$ with $1 \leq q \leq \infty$ and $2 \leq p \leq \infty$.
    \end{lemma}
    \begin{proof}
    Using that $p\geq 2$, H\"older's inequality yields that
    \begin{align*}
    \| (f_{lk})_{k=1}^{2^l} \|_2 \leq 2^{1/2 - 1/p} \|(f_{lk})_{k=1}^{2^l}\|_p.
    \end{align*}
    Define $c_s := 1/(1-2^{-s})$. By another application of H\"older's inequality, we obtain that
    \begin{equation*}
    \underset{l=1}{\overset{\infty}{\sum}} 2^{1/2 - 1/p} \|(f_{lk})_{k=1}^{2^l}\|_p \leq c_s^{1 - 1/q} \left\| \left(2^{l(s+1/2 - 1/p)} \|(f_{lk})_{k=1}^{2^l}\|_p \right)_{l=1}^\infty \right\|_q.
    \end{equation*}
    \end{proof}
    
    Using a similar proof, one obtains also the following lemma, see e.g. Chapter 9 in \cite{johnstone2019manuscript}.
    
    \begin{lemma}\label{lemma:proof-of-bound-on-tail-sum}
        Let $f_{lk}$ are the wavelet coefficients of the function $f \in B^\alpha_{p,q}(R)$. For any $1\leq q \leq \infty$, $2 \leq p \leq \infty$, $L >0$, we have
        \[
        \sum_{l > L} \sum_{k=0}^{2^l-1} f_{lk}^2 \leq c_\alpha 2^{-2L\alpha} R^{2},
        \]
        where $c_\alpha > 0$ is a universal constant depending only on $\alpha$.
    \end{lemma}
    
    \begin{lemma}\label{lem:uniform_bound_on_f}
    There exists a constant $C_{\alpha,R} > 0$ such that $\|f\|_\infty \leq C_{\alpha,R}$ for all $f \in B^\alpha_{p,q}(R)$ with $\alpha - 1/2 - 1/p > 0$, $1 \leq q \leq \infty$ and $2 \leq p \leq \infty$.
    \end{lemma}

    \begin{lemma}\label{lemma:proof-of-bound-on-tail-sum}
       Let $f_{lk}$ are the wavelet coefficients of the function $f \in \cB^s_{p,q}(R)$. For any $1\leq q \leq \infty$, $2 \leq p \leq \infty$, $L >0$, we have
       \[
       \sum_{l > L} \sum_{k=0}^{2^l-1} f_{lk}^2 \leq C_{s,q} 2^{-2Ls} R^{2},
       \]
       where $C_{s,q} > 0$ is a universal constant depending only on $s$ and $q$.
    \end{lemma}

\end{document}